%% file: main.tex
\documentclass[10pt,bibliography=totoc]{article}
\usepackage{etex}
\usepackage[utf8]{inputenc} 
\usepackage[T1]{fontenc} 
\usepackage{latexsym}
\usepackage{empheq}
\usepackage{amsmath,amssymb,amsfonts,amsthm}
\usepackage{makecell}

\usepackage[american]{babel}
\usepackage{mathrsfs}
\usepackage{eqnarray}
\usepackage{graphicx}
\usepackage{tabularx}
\usepackage{caption}
\usepackage{subcaption}
\usepackage[table]{xcolor}
\usepackage{colortbl}
\usepackage{multicol}
\usepackage{multirow}
\usepackage{tikz}
\usepackage{bbm}
\usetikzlibrary{intersections,shapes.arrows}
\usetikzlibrary{decorations.pathreplacing,decorations.markings}
\usetikzlibrary{arrows,decorations.markings}
\usetikzlibrary{spy}
\usepackage{pgfplots}
\pgfplotsset{compat=1.14}
\usepgfplotslibrary{groupplots}
\usepgfplotslibrary{colorbrewer}
\pgfplotsset{
    cycle list/Set1,
}
\usepackage{epsfig}
\usepackage{textcomp}
\usepackage{color,xcolor}
\usepackage{enumitem} 
\usepackage{marginnote} 
\usepackage{todonotes}
\presetkeys{todonotes}{size=\footnotesize, color=white, linecolor = black}{}
\usepackage{mathtools}
\usepackage{extarrows}
\usepackage{xspace}
\usepackage{microtype}
\usepackage{csquotes} 
\usepackage{algorithm} 
\usepackage{algpseudocode}
\usepackage{float}
\usepackage{authblk}
\usepackage{tikz-cd}
\usepackage[text={6.25in,9in},centering]{geometry}
\usepackage{cleveref}%
\usepackage{comment}  
\usepackage{empheq}
\usepackage{soul} 

\newcommand{\bl}[1]{\textcolor{black}{#1}}

\input{macros.tex}

\allowdisplaybreaks

\title{Rank-adaptive structure-preserving model order reduction of Hamiltonian systems}
\author[]{%
Jan S. Hesthaven\thanks{Chair of Computational Mathematics and Simulation Science (MCSS),
			  \'Ecole Polytechnique F\'ed\'erale de Lausanne (EPFL),
			  CH-1015 Lausanne, Switzerland.
Email: \texttt{Jan.Hesthaven@epfl.ch}}\;,
Cecilia Pagliantini\thanks{\textit{Corresponding author}. Centre for Analysis, Scientific computing and Applications,
			  Department of Mathematics and Computer Science,
			  Eindhoven University of Technology (TU/e),
			  5600 MB Eindhoven, The Netherlands.
Email: \texttt{c.pagliantini@tue.nl}}\;, and
Nicol\`o Ripamonti\thanks{Chair of Computational Mathematics and Simulation Science (MCSS),
			  \'Ecole Polytechnique F\'ed\'erale de Lausanne (EPFL),
			  CH-1015 Lausanne, Switzerland.
Email: \texttt{nicolo.ripamonti@epfl.ch}}
}
\date{}

\usepackage[hyphens]{url}
\usepackage[
    bibencoding=utf8,
	style=numeric,
	citestyle=numeric,sorting=nyvt,sortcites=false,
	giveninits=true,
	maxnames=10,
	url=true,eprint=false,doi=false,isbn=false,
	backref=false,
]{biblatex}
\renewbibmacro{in:}{} 
\AtEveryBibitem{\clearfield{note}} 
\AtEveryBibitem{\ifentrytype{book}{\clearfield{pages}}{}}
\usetikzlibrary{external}
\usepackage{standalone}

\begin{document}

\maketitle

\begin{abstract}
This work proposes an adaptive structure-preserving model order reduction method
for finite-dimensional parametrized Hamiltonian systems modeling
non-dissipative phenomena.
To overcome the slowly decaying Kolmogorov width
typical of transport problems,
the full model is approximated on
local reduced spaces that are adapted in time
using dynamical low-rank approximation techniques.
The reduced dynamics is prescribed by approximating the symplectic projection of the Hamiltonian vector field
in the tangent space to the local reduced space. This ensures that the canonical symplectic structure
of the Hamiltonian dynamics is preserved during the reduction.
In addition, accurate approximations with low-rank reduced solutions are obtained by allowing
the dimension of the reduced space to change during the time evolution.
Whenever the quality of the reduced solution, assessed via an error indicator, is not satisfactory,
the reduced basis is augmented in the parameter direction that is worst approximated
by the current basis. 
Extensive numerical tests involving wave interactions, nonlinear transport problems, and the Vlasov equation demonstrate the superior stability properties
and considerable runtime speedups of the proposed method
as compared to global and traditional reduced basis approaches.
\end{abstract}




\section{Introduction}
Hamiltonian systems describe conservative dynamics and non-dissipative phenomena in, for example, classical
mechanics, transport problems, fluids and kinetic models.
We consider finite-dimensional Hamiltonian systems, in canonical symplectic form, that depend on a set of parameters
associated with the geometric configuration of the problem or which represent physical properties of the problem. 
The development of numerical methods for the solution of
parametric Hamiltonian systems in many-query and long-time simulations is challenged by two major factors: the high computational cost
required to achieve sufficiently accurate approximations, and
the possible onset of
numerical instabilities
resulting from failing
to satisfy the conservation laws underlying non-dissipative dynamics.
%
Model order reduction (MOR) and reduced basis methods (RBM) provide an effective procedure
to reduce the computational cost of such simulations by replacing the original high-dimensional problem with
models of reduced dimensionality without compromising the accuracy of the approximation.
The success of RBM relies on the assumption that the problem possesses a low-rank nature, i.e. that the set of solutions, obtained as time and parameters vary, is of low dimension. However, non-dissipative phenomena
do not generally exhibit such global low-rank structure and are characterized by slowly decaying Kolmogorov $n$-widths.
This implies that traditional reduced models derived via linear approximations are generally not effective.

In recent years, there has been a growing interest in the development of model order reduction techniques for transport-dominated problems to overcome the limitations of
linear global approximations. A large class of methods consists in
constructing nonlinear transformations of the solution manifold and to recast
it in a coordinate framework where it admits a low-rank structure, e.g. \cite{OR13,IL14,W17,RSSM18,ELMV19,CMS19,LC20,T20}.
A second family of MOR techniques focuses on online adaptive methods that update local reduced spaces depending on parameter and time, e.g. 
\cite{C15,PW15,RPM19}.
To the best of our knowledge, none of the aforementioned methods provides any guarantee on the preservation of the physical properties and the geometric structure of the problem considered, and they might therefore be unsuitable to treat non-dissipative phenomena.

In parametric dynamical systems, the state can be represented, at each time, as a matrix whose columns are the solution vectors associated with different parameter values.
In this perspective, finding a low-dimensional space in which the solution state can be well approximated is strictly related to low-rank matrix approximations.
In a time-dependent setting, dynamical low-rank approximation \cite{KL07} provides a low-rank factorization updating technique to efficiently compute approximations of time-dependent large data matrices.
This approach can be equivalently seen as a reduced basis method based on a modal decomposition of the approximate solution with dynamically
evolving modes.
A geometric perspective on the relation between dynamical low-rank approximation and model order reduction in the context of time-dependent matrices has been proposed in \cite{FL18}.
To the best of our knowledge the only dynamical low-rank approximation methods able to preserve the geometric structure of Hamiltonian dynamics
were proposed in \cite{MN17} to deal with the spatial approximation of the stochastic wave equation and in \cite{P19} to deal with finite-dimensional Hamiltonian systems.
The gist of these methods is to approximate the full model solution in a low-dimensional manifold that evolves in time and possesses the symplectic structure of the full phase-space. The reduced dynamics is then derived via a symplectic projection of the Hamiltonian vector field onto the tangent space of the reduced symplectic manifold at each reduced state.

Their success notwithstanding, traditional dynamical low-rank approximation techniques are
based on a reduced (low-rank) space whose dimension
is fixed at the beginning of the evolution. This is a major limitation since
it frequently happens that the rank of the initial condition does not correctly
reflect the effective rank of the solution at all times.
Consider, as an example, a linear advection problem in 1D, where the parameter represents the transport velocity.
It is clear that, if the initial condition does not depend on the parameter, its rank is equal to one. However, as the initial condition is advected in time with different velocities, its rank rapidly increases.
Approximating such dynamics with a time-dependent sequence of reduced manifolds
of rank-1 matrices yields poor approximations. Conversely, an overapproximation of the initial condition, and possibly of the solution at other times, could improve the accuracy but will inevitably yield situations of rank-deficiency, as observed in \cite[Section 5.3]{KL07}.
This example demonstrates that, in a dynamical reduced basis approach, it is crucial to accurately capture
the rank of the full model solution at each time.
This issue has, however, received little attention so far \cite{SL12,CHZ13}.
In this work, we propose a novel dynamical low-rank approximation scheme for the solution of parametric Hamiltonian systems that combines adaptivity in the rank of the solution with preservation of the Hamiltonian structure of the dynamics.

The proposed rank-adaptive algorithm can be summarized as follows.
\begin{itemize}
\item Given a fixed partition of the temporal domain, we consider, in each temporal subinterval,
the discretized reduced dynamical system obtained with the structure-preserving approach of \cite{P19}.
While in \cite{P19} the rank of the approximate solution is fixed a priori, here we \emph{change the rank adaptively} from one temporal interval to the next one.
\item To this aim, a surrogate error based on a linearization of the problem residual is computed at chosen times and for all tested parameters.
If the error indicator reveals, according to a specific criterion, that the current reduced space is too small to approximate the state, we augment it
in the direction that is worst approximated by the current reduced basis. The reduced dynamical system is then evolved, in the subsequent temporal interval, in the augmented manifold.
In case of overapproximation, the size of the reduced space is, instead, decreased.
\item Two major difficulties are associated with this approach: (i) to maintain the \emph{global} Hamiltonian structure of the dynamics while modifying the reduced phase space; and
(ii) to evolve the system on the updated space starting from a rank-deficient initial condition.
To address these problems, we devise a regularization of the velocity field of the reduced flow so that the resulting vector belongs to the tangent space of the updated reduced manifold, and the Hamiltonian structure is then preserved.
\end{itemize}

The remainder of the paper is organized as follows.
In \Cref{sec:pbm}, we introduce parametrized Hamiltonian systems and describe their symplectic structure.
In \Cref{sec:DLR}, we derive the evolution equations of the
reduced Hamiltonian dynamics.
The problem of overapproximation and rank-deficiency is discussed in \Cref{sec:evol-rank-deficient}, where the regularization algorithm is introduced.
\Cref{sec:pRK} deals with the numerical temporal integration of the reduced dynamics: first, we summarize the structure-preserving methods introduced in \cite{P19} for the evolution of the reduced basis,
and then we design novel partitioned RK schemes that are accurate
with order 2 and 3 and preserve the geometric structure of the evolution problem.
\Cref{sec:rank-adaptivity} pertains to the rank-adaptive algorithm. We describe the major steps: computation of the error indicator, criterion for the rank update, and update of the reduced state.
The computational complexity of the adaptive dynamical reduced basis algorithm is thoroughly analyzed in \Cref{sec:cost}.
Furthermore, an approach that combines tensorial and splitting techniques with coarsening strategies is proposed to efficiently deal with polynomial nonlinearities of the Hamiltonian gradient.
\Cref{sec:numerical_tests} is devoted to extensive numerical simulations
of the proposed algorithm and its numerical comparisons with non-adaptive and global reduced basis methods. Finally, \Cref{sec:conclusions} concludes with a few remarks.

\section{Problem formulation}\label{sec:pbm}
Let $\Tcal:=(t_0,T]$ be a temporal interval and let
$\Sprm\subset \mathbb{R}^d$, with $d\geq 1$,
be a compact set of parameters.
For each $\prm\in\Sprm$,
we consider the
Hamiltonian system described by the
initial value problem:
For $u_0(\prm)\in\Vd{\Nf}$, find $u(\cdot,\prm)\in C^1(\Tcal,\Vd{\Nf})$ such that
\begin{equation}\label{eq:HamSystem}
	\left\{
	\begin{array}{ll}
		\dot{u}(t;\prm) = \J{\Nf}\nabla_u\Ham(u(t;\prm);\prm), & \quad \quad\mbox{for }\;t\in\Tcal,\\
		u(t_0;\prm) = u_0(\prm),&
	\end{array}\right.	
\end{equation}
where the dot denotes the derivative with respect to time $t$, $\Vd{\Nf}$ is a $\Nf$-dimensional vector space,
and $C^1(\Tcal,\Vd{\Nf})$ denotes continuous differentiable functions in time taking values in $\Vd{\Nf}$.
Moreover, the function $\Ham:\Vd{\Nf}\times\Sprm\rightarrow\r{}$ is the Hamiltonian of the system,
$\nabla_u$ is the gradient with respect to the state variable $u$,
and $\J{\Nf}$ is the so-called canonical symplectic tensor defined as
\begin{equation}\label{eq:J}
\J{\Nf} :=
	\begin{pmatrix}
		 0_{\Nfh} & \Idm_{\Nfh} \\
		-\Idm_{\Nfh} & 0_{\Nfh} \\
	\end{pmatrix}\in\R{\Nf}{\Nf},
\end{equation}
with $\Idm_{\Nfh}, 0_{\Nfh} \in\R{\Nfh}{\Nfh}$ denoting the identity and zero matrices, respectively.
The operator $\J{\Nf}$ identifies
a symplectic structure on the phase-space of the Hamiltonian system
\eqref{eq:HamSystem}.
Equivalently, the vector space $\Vd{\Nf}$ admits a
global basis that is symplectic and orthonormal according to the following definition.
%
%
\begin{definition}[Orthosymplectic basis]\label{def:ortsym}
The set of vectors $\{e_i\}_{i=1}^{\Nf}$ is said to be \emph{orthosymplectic} in the $\Nf$-dimensional vector space $\Vd{\Nf}$
if
\begin{equation*}
	e_i^\top \J{\Nf} e_j=(\J{\Nf})_{i,j}\,,\quad\mbox{ and }\quad (e_i,e_j)=\delta_{i,j}\,,\qquad\forall i,j=1\ldots,\Nf,
\end{equation*}
where $(\cdot,\cdot)$ is the Euclidean inner product and
$\J{\Nf}$ is the canonical symplectic tensor \eqref{eq:J} on $\Vd{\Nf}$.
\end{definition}

\section{Dynamical reduced basis method for Hamiltonian systems}
\label{sec:DLR}
We are interested in solving the Hamiltonian
system \eqref{eq:HamSystem} for a given set of $\Np$ vector-valued parameters $\{\prm_j\}_{j=1}^{\Np}\subset\Sprm$, that,
with a small abuse of notation, we denote $\prmh\in \Sprmh$.
Then, the state variable $u$ in \eqref{eq:HamSystem}
can be thought of as a matrix-valued application
$u(\cdot;\prmh):\Tcal\rightarrow \Vd{\Nf}^{\Np}\subset\R{\Nf}{\Np}$
where $\Vd{\Nf}^{\Np}:=\Vd{\Nf}\times\ldots\times\Vd{\Nf}$.
Throughout, for a given matrix $\Rcal\in\R{\Nf}{\Np}$, we denote with
$\Rcal_{j}\in\r{\Nf}$ the vector corresponding to the $j$-th column of $\Rcal$,
for any $j=1,\ldots,\Np$.
Let $[a_1|a_2|\ldots|a_r]$ denote the matrix of size $\Nf\times (m_1+\ldots + m_r)$
resulting from the horizontal concatenation of the matrices $a_j\in\mathbb{R}^{\Nf\times m_j}$ for $j=1,\ldots,r$.
The Hamiltonian system \eqref{eq:HamSystem},
evaluated at $\prmh$,
can be recast as a set of ordinary differential equations
in a $\Nf\times\Np$ matrix unknown in $\Vd{\Nf}^{\Np}$ as follows.
For $\Rcal_0(\prm_h):=\big[u_0(\prm_1)|\ldots|u_0(\prm_{\Np})\big]\in\Vd{\Nf}^{\Np}$,
find $\Rcal\in C^1(\Tcal,\Vd{\Nf}^{\Np})$ such that
\begin{equation}\label{eq:HamSystemMatrix}
	\left\{
	\begin{array}{ll}
		\dot{\Rcal}(t) = \Xcal_{\HamN}(\Rcal(t),\prm_h)
		:= \J{\Nf}\nabla\HamN(\Rcal(t);\prmh),&\quad\quad\mbox{for }\; t\in\Tcal,\\
		\Rcal(t_0) = \Rcal_0(\prm_h), &
	\end{array}\right.	
\end{equation}
where $\HamN:\Vd{\Nf}^{\Np}\rightarrow\r{\Np}$ and,
for any $\Rcal\in\Vd{\Nf}^{\Np}$, its gradient
$\nabla\HamN(\Rcal;\prmh)\in\Vd{\Nf}^{\Np}$
is defined as $(\nabla\HamN(\Rcal;\prmh))_{i,j}=\frac{\partial \HamN_j}{\partial \Rcal_{i,j}}$,
for any $i=1,\ldots,\Nf$, $j=1,\ldots,\Np$.
The function $\HamN_j$ is the Hamiltonian of 
the dynamical system \eqref{eq:HamSystem} corresponding
to the parameter $\prm_j$, for $j=1,\ldots,\Np$.
We assume that, for a fixed sample of parameters $\prmh\in\Sprmh$,
the vector field $\Xcal_{\HamN}(\cdot;\prmh)\in\Vd{\Nf}^{\Np}$
is Lipschitz continuous in the Frobenius norm $\norm{\cdot}$ uniformly
with respect to time, so that
\eqref{eq:HamSystemMatrix} is well-posed.

Let us consider the splitting of the time domain $\Tcal$ into the
union of intervals $\Tcalt:=(t^{\tau-1},t^{\tau}]$, $\tau=1,\ldots,N_{\tau}$,
with $t^0:=t_0$ and $t^{N_{\tau}}:=T$, and let the local time step be defined as
$\dt_{\tau}=t^{\tau}-t^{\tau-1}$ for every $\tau$.
For the model order reduction of \eqref{eq:HamSystemMatrix}
we propose an adaptive dynamical scheme based on approximating the full model solution in a lower-dimensional space that is evolving, and whose dimension may also change over time. To this aim, we adopt a local perspective by considering, in each temporal interval, an
approximation of the solution of \eqref{eq:HamSystemMatrix} of the form
\begin{equation}\label{eq:Rrb}
	\Rcal(t)\approx R(t) 
		= \sum_{i=1}^{\Nrt} U_i(t) Z_i(t,\prm_h) = U(t)Z(t),
		\qquad\forall\, t\in\Tcalt,
\end{equation}
where $U(t)=\big[U_1|\ldots|U_{\Nrt}\big]\in\R{\Nf}{\Nrt}$, and
$Z\in\R{\Nrt}{\Np}$
is such that
$Z_{i,j}(t)=Z_i(t,\prm_j)$ for $i=1,\ldots,\Nrt$, $j=1,\ldots,\Np$,
and any $t\in\Tcalt$. Here $\Nrht\in\mathbb{N}$ satisfies $\Nrt\leq \Np$
and $\Nrht\ll\Nfh$, and is updated over time according to Algorithm~\ref{algo:rank-update} that we will thoroughly discuss in Section~\ref{sec:rank-adaptivity}.
With this notation, we introduce the collection of reduced spaces
of $\Nf\times\Np$ matrices having rank at most $\Nrt$, and characterized as
\begin{equation*}
	\Mcal_{\Nrt} := \{R\in\R{\Nf}{\Np}:\; R = UZ\;\mbox{ with }\;
		U\in\Ucalt,\, Z\in \Zcalt \},\qquad\forall\,\tau=1,\ldots,N_{\tau},
\end{equation*}
where $U$ represents the reduced basis and it is taken to be orthogonal and symplectic, while $Z$ are the expansion coefficients in the reduced basis, i.e.
\begin{equation}\label{eq:ManUZ}
\begin{aligned}
	\Ucalt & :=\{U\in\R{\Nf}{\Nrt}:\;U^\top U=\Idm_{\Nrt},\; U^\top \J{\Nf} U = \J{\Nrt}\}, \\
	\Zcalt & :=\{Z\in\R{\Nrt}{\Np}:\;\rank{ZZ^\top + \J{\Nrt}^\top ZZ^\top \J{\Nrt}} = \Nrt\}. 
\end{aligned}
\end{equation}
To approximate the Hamiltonian system \eqref{eq:HamSystemMatrix} in $\Tcalt$ with an
evolution problem on the reduced space $\Mcal_{\Nrt}$
we need to prescribe
evolution equations for the reduced basis $U(t)\in\Ucalt$ and
the expansion coefficients $Z(t)\in\Zcalt$.
For this, we follow the approach proposed in 
\cite{MN17} and \cite{P19}, and
derive the reduced flow describing the dynamics of the reduced state $R$ in \eqref{eq:Rrb} by applying to the
Hamiltonian vector field $\Xcal_{\HamN}$
the symplectic projection
$\Pi_{\T{R(t)}{\Mcal_{\Nrt}}}$
onto the tangent space of the reduced manifold at the current state.
The resulting local evolution problem reads:
Find $R\in C^1(\Tcalt,\Mcal_{\Nrt})$ such that
\begin{equation}\label{eq:dynn}
	\dot{R}(t) = \Pi_{\T{R}{\Mcal_{\Nrt}}}\Xcal_{\HamN}(R(t),\prm_h),\qquad\quad\mbox{for }\; t\in\Tcalt,
\end{equation}
where we assume, for the time being, that the initial condition
of \eqref{eq:dynn}
at time $t^{\tm}$, $\tau\geq 1$, is given, and we refer to
\Cref{sec:rank-update} for a complete description of how
such an initial condition is prescribed.

By exploiting the characterization of the projection operator
$\Pi_{\T{R(t)}{\Mcal_{\Nrt}}}$ in \cite[Proposition 4.2]{P19},
we obtain the local evolution equations for the factors $U$ and
$Z$ in the modal decomposition of the reduced solution \eqref{eq:Rrb},
as in \cite[Proposition 6.9]{MN17} and \cite[Equation (4.10)]{P19}.
In more details, for any $\tau\geq 1$, given
$(U(t^{\tm}),Z(t^{\tm}))\in \Ucalt\times \Zcalt$ we seek
$(U,Z)\in C^1(\Tcalt,\Ucalt)\times C^1(\Tcalt,\Zcalt)$ such that
\begin{subequations}\label{eq:UZred}
\begin{empheq}[left = \empheqlbrace\,]{align}
	& \dot{Z}(t) = \J{\Nr}\nabla_Z \HamU(Z,\prmh), &\mbox{for }\; t\in\Tcalt,\label{eq:Z}\\
	& \dot{U}(t) = (\Idm_{\Nf}-UU^\top)(\J{\Nf}YZ^\top -
	YZ^\top\J{\Nrt}^\top)
	(ZZ^\top+\J{\Nrt}^\top ZZ^\top\J{\Nrt})^{-1}, &\mbox{for }\; t\in\Tcalt,
	\label{eq:U}
\end{empheq}
\end{subequations}
where $Y(t):=\nabla\HamN(R(t);\prmh)\in\Vd{\Nrt}^{\Np}$,
and $R(t)=U(t)Z(t)$ for all $t\in\Tcalt$.
Observe that the local expansion coefficients $Z\in\Zcalt$ satisfy
a Hamiltonian system \eqref{eq:Z}
of reduced dimension $\Nrt$, where the
reduced Hamiltonian is defined as
$\HamU(Z;\prmh):=\HamN(UZ;\prmh)$.

To compute the initial condition of the reduced problem
at time $t_0$ we perform the
complex SVD \cite[Section 4.2]{peng2016symplectic}
of $\Rcal_0(\prm_h)\in\R{\Nf}{\Np}$ in \eqref{eq:HamSystemMatrix}, truncated at the $\Nrh_1$-th mode. Then, the
initial reduced basis $U_0\in\Ucal_1$ can be derived from the unitary matrix of left singular vectors of 
$\Rcal_0(\prm_h)$, via the isomorphism between $\Ucal_1$ and
the Stiefel manifold of unitary $\Nfh\times\Nrh_1$ complex matrices, \emph{cf.} \cite[Lemma 6.1]{MN17}.
The expansion coefficients matrix is initialized as
$Z_0 = U_0^\top \Rcal_0(\prm_h)$.

\section{Partitioned Runge--Kutta methods}\label{sec:pRK}
For the numerical time integration of the reduced dynamical system \eqref{eq:dynn} we rely on partitioned Runge--Kutta (RK) methods.
Partitioned RK methods were originally introduced to deal with stiff evolution problems by splitting the dynamics
into a stiff and a nonstiff part so that the two subsystems could be treated with
different temporal integrators. There are many other situations where a dynamical system
possesses a natural partitioning, for example Hamiltonian or singularly perturbed problems,
or nonlinear systems with a linear part.
In our setting, the factorization of the reduced solution \eqref{eq:Rrb}
into the basis $U$ and the coefficients $Z$ provides the natural splitting
expressed in \eqref{eq:UZred}.

In this section we first consider structure-preserving numerical approximations
of the evolution problems \eqref{eq:U} and \eqref{eq:Z}, treated separately.
Then, for the numerical integration of the coupled system \eqref{eq:UZred}, we design partitioned RK schemes
that are accurate with order $2$ and $3$ and preserve the geometric structure of
each evolution problem.

Since the evolution equation \eqref{eq:Z} is a Hamiltonian system (of reduced dimension) we can rely on
symplectic methods for its temporal approximation,
so that the symplectic properties of the flow are preserved at the discrete level, \emph{cf.} e.g. \cite{HaLuWa06}.
The evolution equation \eqref{eq:U} for the reduced basis
is approximated using tangent methods
analogous to the ones proposed in \cite{P19}, and that
we briefly summarize here.
The idea of tangent methods for the solution of differential equations on manifolds, as introduced in \cite{CO02},
is to recast the local dynamics on the tangent space of the manifold, which is a linear space. The temporal approximation of \eqref{eq:U} by
tangent methods allow to obtain, at a computational cost linear in $\Nfh$, a discrete reduced basis that is orthogonal and symplectic.
Let
$\Fcal(\cdot,\cdot;\prmh):\R{\Nf}{\Nrt}\times \Zcalt\rightarrow\R{\Nf}{\Nrt}$
denote the velocity field of the evolution \eqref{eq:U} of the reduced basis, namely
\begin{equation}\label{eq:Fcal}
     \Fcal(U,Z;\prmh) := (\Idm_{\Nf}-UU^\top)(\J{\Nf}YZ^{\top} - YZ^{\top}\J{\Nrt}^\top)S^{-1},
     \qquad\forall\, U\in\R{\Nf}{\Nrt},\, Z\in\Zcalt.
\end{equation}
It can be easily shown that, for any $Q\in\Ucalt$, $\Fcal(Q,Z;\prmh)$ belongs to the space
\begin{equation}\label{eq:HU}
H_Q:=\{X\in\R{\Nf}{\Nrt}:\; X^\top Q=0,\,X\J{\Nrt}=\J{\Nf}X\}.
\end{equation}
This is a subspace of the
tangent space of the manifold $\Ucalt$ of orthosymplectic $\Nf\times\Nrt$
matrices at the point $Q\in\Ucalt$.
Let us assume to know, in each temporal interval $\Tcalt$,
the approximate
solution $Q:=U_{\tau-1}\in\Ucal_{\tau}$ of $U(t^{\tau-1})$.
Then, any element of $\Ucalt$, in a neighborhood of $Q$,
can be expressed as the image of a vector $V\in H_Q$
via the retraction
\begin{equation}\label{eq:retraction}
\begin{aligned}
    \Rcal_Q:H_Q & \longrightarrow \Ucalt\\
	V & \longmapsto \cay(VQ^\top-Q V^\top) Q,
	\end{aligned}
\end{equation}
where $\cay$ is the Cayley transform,
defined as $\cay(M)=(I_{\Nfh}-M/2)^{-1}(I_{\Nfh}+M/2)$ for any skew-symmetric and Hamiltonian square matrix $M\in\R{\Nf}{\Nf}$.
%
Since $\Rcal_Q$ is a retraction, rather than
solving \eqref{eq:U} for $U$, one can
derive the local behavior of $U$ in a neighborhood of $Q$ by
evolving $V(t)$, with $U(t)=\Rcal_Q(V(t))$, in the \bl{space $H_Q$}.
By computing the local inverse of the tangent map of the retraction $\Rcal_Q$, the evolution problem for the vector $V$ reads: for any $t\in\Tcalt$,
\begin{equation}\label{eq:EvolTM}
    \dot{V}(t) = f_{\tau}(V(t),Z(t);\prmh) := -Q(\Rcal_Q(V)^{\top} Q+\Idm_{\Nrt})^{-1}(\Rcal_Q(V)+Q)^{\top} \Phi+ \Phi - Q\Phi^\top Q,
\end{equation}
where
$\Phi:=\big(2\Fcal(\Rcal_Q(V),Z;\prmh)
- (\bl{VQ^\top-QV^\top})\Fcal(\Rcal_Q(V),Z;\prmh)\big)
(Q^{\top}\Rcal_Q(V)+\Idm_{\Nrt})^{-1}$.
\bl{We refer to \cite[Section 5.3.1]{P19} for further details on the derivation of the function $f_{\tau}$.}

The resulting set of evolution equations describes the reduced dynamics in each temporal interval $\Tcalt$ as: given $(U_{\tau-1},Z_{\tau-1})\in\Ucalt\times\Zcalt$,
find $Z(t)\in\Zcalt$ and $V(t)\in\bl{H_{U_{\tau-1}}}$
such that $U(t)=\Rcal_{U_{\tau-1}}(V(t))$ for all $t\in\Tcalt$ and
\begin{equation}\label{eq:PRK-ZV}
\left\{
\begin{array}{ll}
	\dot{Z}(t) = \Gcal(\Rcal_{U_{\tau-1}}(V(t)),Z(t);\prmh), &\quad\quad\mbox{for }\;t\in\Tcalt,\\
	\dot{V}(t) = f_{\tau}(V(t),Z(t);\prmh), &\quad\quad\mbox{for }\;t\in\Tcalt,\\
	V(t^{\tau-1}) = 0\in\bl{H_{U_{\tau-1}}}, &\\
	Z(t^{\tau-1}) = Z_{\tau-1}\in\Zcalt, &
\end{array}
\right.
\end{equation}
where $\Gcal:=\J{\Nr}\nabla \HamU(Z,\prmh)$ from \eqref{eq:Z}
and $f_{\tau}$ is defined in \eqref{eq:EvolTM}.

For the numerical approximation of \eqref{eq:PRK-ZV},
we \bl{derive} partitioned Runge--Kutta methods.
Let $P_Z=(\{b_i\}_{i=1}^{\Ns},\{a_{ij}\}_{i,j=1}^{\Ns})$
be the collection of coefficients of the Butcher tableau describing an
$\Ns$-stage \emph{symplectic} RK method, and
let $\widehat{P}_U=(\{\widehat{b}_i\}_{i=1}^{\Ns},\{\widehat{a}_{ij}\}_{1\leq j<i\leq \Ns})$
be the set of coefficients of an
$\Ns$-stage \emph{explicit} RK method.
Then, the numerical approximation of \eqref{eq:PRK-ZV} via partitioned RK integrators reads
\begin{equation}\label{eq:PRK-ZVh}
\begin{array}{ll}
	& Z_{\tau} = Z_{\tau-1}+\dt\sum\limits_{i=1}^\Ns b_i k_i, \qquad V_{\tau} = \dt\sum\limits_{i=1}^\Ns \widehat{b}_i \widehat{k}_i, \\[0.5em]
	& \qquad k_1 = \Gcal(U_{\tau-1},
	    Z_{\tau-1} + \dt\sum\limits_{j=1}^{\Ns} a_{1,j}k_j;\prmh),
	    \qquad \widehat{k}_1 = \Fcal(U_{\tau-1},
	    Z_{\tau-1} + \dt\sum\limits_{j=1}^{\Ns} a_{1,j}k_j;\prmh), \\
	& \qquad k_i = \Gcal\bigg( \Rcal_{U_{\tau-1}}\big(\dt\sum\limits_{j=1}^{i-1} \widehat{a}_{i,j}\widehat{k}_j\big),
	    Z_{\tau-1} + \dt\sum\limits_{j=1}^{\Ns} a_{i,j}k_j;\prmh\bigg),   \qquad i=2,\ldots,\Ns,\\
	& \qquad \widehat{k}_i = f_{\tau}\bigg(\dt\sum\limits_{j=1}^{i-1} \widehat{a}_{i,j}\widehat{k}_j,
	    Z_{\tau-1} + \dt\sum\limits_{j=1}^{\Ns} a_{i,j}k_j;\prmh\bigg)\bl{,} \qquad\qquad\quad\quad\!\!\! i=2,\ldots,\Ns,\\
	& U_{\tau} = \Rcal_{U_{\tau-1}}(V_{\tau}). 
\end{array}
\end{equation}
Runge--Kutta methods of order 2 and 3 with the aforementioned properties can be characterized in terms of the coefficients
$P_Z$ and $\widehat{P}_U$ as in the following result.

\begin{lemma}\label{lemma:pRKcond}
Consider the numerical approximation of \eqref{eq:PRK-ZV}
with the $\Ns$-stage partitioned Runge--Kutta method \eqref{eq:PRK-ZVh}
obtained by coupling
the Runge--Kutta methods
$P_Z=(\{b_i\}_{i=1}^{\Ns},\{a_{ij}\}_{i,j=1}^{\Ns})$
and
$\widehat{P}_U=(\{\widehat{b}_i\}_{i=1}^{\Ns},\{\widehat{a}_{ij}\}_{1\leq j<i\leq \Ns})$.
Then, the following statements hold.
\begin{itemize}
    \item \emph{Symplectic condition} \cite[Theorem VI.4.3]{HaLuWa06}. The Runge--Kutta method $P_Z$ is symplectic if
    \begin{equation}\label{eq:symCond}
        b_ia_{ij}+b_ja_{ji}=b_ib_j,\qquad \forall\;i,j=1,\ldots,\Ns.
    \end{equation}
    \item \emph{Order condition} \cite[Theorem II.2.13]{HNW93}. The Runge--Kutta method $P_Z$ has order $k$, with
    \begin{align}
	& k=2\quad\mbox{iff}\qquad\sum_{i=1}^{\Ns} \bl{b_i}=1,\quad
	\sum_{i,j=1}^{\Ns} b_i a_{ij}=\dfrac12;\label{eq:RKp2}\\
	& k=3\quad\mbox{iff}\qquad\sum_{i=1}^{\Ns} b_i=1,\quad
    \sum_{i,j=1}^{\Ns} b_i a_{ij}=\dfrac12,\quad
    \sum_{i=1}^{\Ns} b_i \bigg(\sum_{j=1}^{\Ns} a_{ij}\bigg)^2 = \dfrac13,\quad
    \sum_{i,j,\ell=1}^{\Ns}b_i a_{ij}a_{j\ell} = \dfrac16.\label{eq:RKp3}
    \end{align}
    \item \emph{Coupling condition} \cite[Section III.2.2]{HaLuWa06}.
    The partitioned Runge--Kutta method $(P_Z,\widehat{P}_U)$ has order $p$, if
    $P_Z$ and $\widehat{P}_U$ are both of order $k$ and
    \begin{align}
	& k=2\quad\mbox{if}\qquad\sum_{i=1}^{\Ns}\sum_{j=1}^{i-1}  b_i \widehat{a}_{ij} = \dfrac12,\qquad
    \sum_{i=1}^{\Ns}\sum_{j=1}^{\Ns}  \widehat{b}_i a_{ij} = \dfrac12;
    \label{eq:pRKp2}\\
    & k=3\quad\mbox{if}\qquad\sum_{i=1}^{\Ns} a_{ij} = \sum_{i=1}^{j-1} \widehat{a}_{ij},\qquad
    \sum_{i,\ell=1}^{\Ns}\sum_{j=1}^{i-1}  b_i \widehat{a}_{ij}a_{j\ell} = \dfrac16,\qquad
    \sum_{i,j,\ell=1}^{\Ns}  \widehat{b}_i a_{ij}a_{j\ell} = \dfrac16.\label{eq:pRKp3}
\end{align}
\end{itemize}
\end{lemma}
Partitioned Runge--Kutta of order 2 and 3 can be derived as described in
\Cref{app:pRK}.

\section{Reduced dynamics under rank-deficiency}
\label{sec:evol-rank-deficient}
In Section~\ref{sec:DLR} we have proposed to approximate the phase space of the full Hamiltonian system \eqref{eq:HamSystemMatrix} by an evolving low-rank matrix manifold. Particular attention needs to be devoted to the case of \emph{overapproximation} in which
a full model solution with effective rank $r<\Nrh$
is approximated by a rank-$\Nrh$ matrix,
as pointed out first in \cite[Section 5.3]{KL07}.
In this case, a rank-deficient reduced dynamical system
needs to be solved and it is not clear how the effective rank of the reduced solution
will evolve over time.
\bl{Indeed,} in each temporal interval $\Tcalt$, the dynamics
may not remain on the reduce manifold $\Mcal_{\Nrt}$
and the matrix $S(Z):=ZZ^\top+\J{\Nrt}^\top ZZ^\top\J{\Nrt}$ may become
singular or severely ill conditioned.
This happens, for example,
when the full model state at time $t_0$ is approximated
with a rank deficient matrix, or,
as we will see in the rank-adaptive algorithm in \Cref{sec:rank-adaptivity},
when the reduced solution at a fixed time is used as initial condition to
evolve the reduced system on a manifold
of states with increased rank.

In this \bl{s}ection, we propose an algorithm to deal with the overapproximation
while maintaining the geometric structure of the Hamiltonian dynamics and of the factors $U$ and $Z$ in \eqref{eq:Rrb}.
\begin{lemma}[Characterization of the matrix $S$]\label{lem:S}
Let $S:=ZZ^\top+\J{\Nr}^\top ZZ^\top\J{\Nr}\in\R{\Nr}{\Nr}$
with $Z\in\R{\Nr}{\Np}$ and $\Np\geq \Nr$.
$S$ is symmetric positive semi-definite and it is skew-Hamiltonian, namely
$S\J{\Nr}-\J{\Nr}S^\top=0$. Moreover,
if $S$ has rank $\Nr$ then $S$ is non-singular and $S^{-1}$ is also skew-Hamiltonian.
In particular,
the null space of $S$ is even dimensional and contains all pairs of
vectors $(v,\J{\Nr}v)\in\r{\Nr}\times\r{\Nr}$ such that
both $v$ and $\J{\Nr}v$ belong to 
the null space of $Z^\top$.
\end{lemma}
\begin{proof}
It can be easily verified that $S$ is symmetric positive semi-definite and skew-Hamiltonian. Any eigenvalue of a skew-Hamiltonian matrix has even multiplicity, hence the null space of $S$ has even dimension.
Since $S$ is positive semi-definite, $v\in\ker{S}$ if and only if
$ZZ^\top v=0$ and $ZZ^\top\J{\Nr}v=0$, that is $\ker{S}=\ker{Z^\top}\cap\ker{Z^\top\J{\Nr}}$. Observe that
all the elements $v$ of the kernel
of $Z^\top$ are such that $\J{\Nr}^\top v\in\ker{Z^\top\J{\Nr}}$.
\end{proof}
In addition to the algebraic limitations associated with the solution of a rank-deficient
system, the fact that the matrix $S$ might be singular or ill conditioned
prevents the reduced basis from evolving on the manifold
of the orthosymplectic matrices.
As shown in \cite[Proposition 4.3]{P19}, if $U(t^{\tm})\in\Ucalt$ then $U(t)\in\R{\Nf}{\Nrt}$ solution of
\eqref{eq:U} in $\Tcalt$ satisfies $U(t)\in \Ucalt$ for all $t\in\Tcalt$,
owing to the fact that $\Fcal(U,Z;\prmh)$ belongs to the space 
$H_U$ \bl{in \eqref{eq:HU}}.
%
\begin{lemma}\label{lem:FinTM}
The function
$\Fcal(\cdot,\cdot;\prmh):\R{\Nf}{\Nrt}\times \Zcalt\rightarrow\R{\Nf}{\Nrt}$ defined in \eqref{eq:Fcal}
is such that
$\Fcal(U,Z;\prmh)\in H_U$ if and only if $U\in\Ucalt$
and $Z\in \Zcalt$.
\end{lemma}
\begin{proof}
Let $X_U:=\Fcal(U,Z;\prmh)=(\Idm_{\Nf}-UU^\top) AS^{-1}$, where
$A:=\J{\Nf}YZ^{\top} - YZ^{\top}\J{\Nrt}^\top$.
The condition $X_U^\top U=0$ is satisfied for every $U\in\R{\Nf}{\Nrt}$ orthogonal and $Z\in\R{\Nrt}{\Np}$.
Concerning the second condition, it can be easily shown 
that $\J{\Nf}A\J{\Nrt}^\top=A$ and $\J{\Nf}(\Idm_{\Nf}-UU^\top) = (\Idm_{\Nf}-UU^\top)\J{\Nf}$.
Hence, $\J{\Nf}X_U = (\Idm_{\Nf}-UU^\top)A \J{\Nrt}S^{-1}$
and this is equal to $X_U\J{\Nrt}$ if and only if
$\J{\Nrt}S^{-1}=S^{-1}\J{\Nrt}$. This condition follows from \Cref{lem:S}.
\end{proof}
\Cref{lem:FinTM} can be equivalently stated by considering the velocity field $\Fcal$ as a function of the triple $(U,Z,S(Z))$. Then
$\Fcal(U,Z,S(Z);\prmh)$ belongs to $H_U$ if and only if
$U\in\Ucalt$, $Z\in\R{\Nrt}{\Np}$ and $S(Z)$ is non-singular, symmetric and skew-Hamiltonian.
If the matrix $S$ is not invertible, i.e. $Z\notin\Zcalt$,
its inverse needs to be replaced by
some approximation $S^\dagger$.
By \Cref{lem:FinTM}, if $S^\dagger$ is not symmetric skew-Hamiltonian,
then $\Fcal^\dagger(U,Z;\prmh):=(\Idm_{\Nf}-UU^\top) AS^\dagger$ does no longer belong
to the horizontal space $H_U$.
If, for example, $S^\dagger$  is the pseudo inverse of $S$,
then the above condition is theoretically satisfied, but in numerical computations
only up to a small error, because, if $S$ is rank-deficient,
then its pseudoinverse
corresponds to the pseudoinverse of the truncated SVD of $S$.

To overcome these issues in the numerical solution
of the reduced dynamics \eqref{eq:UZred},
we introduce two approximations:
first we replace the rank-deficient matrix $S$ with an $\reg$-regularization that preserves the skew-Hamiltonian structure of $S$
and then, in finite precision arithmetic, we set as velocity field for the evolution of the reduced basis $U$ an approximation of $\Fcal$ in the space $H_{U(t)}$, for all $t\in\Tcalt$.
The $\reg$-regularization consists in
diagonalizing $S$ and then replacing, in the resulting diagonal factor,
the elements below a certain threshold with a fixed factor $\reg\in\r{}$.
This is possible since 
(real) symmetric matrices are always diagonalizable by orthogonal transformations. However, unitary transformations do not preserve the skew-Hamiltonian structure. We therefore consider the following
Paige Van Loan (PVL) decomposition, based on symplectic equivalence transformations.
\begin{lemma}[{\cite{VanLoan84}}]\label{lem:PVL}
Given a skew-Hamiltonian matrix $S\in\R{\Nr}{\Nr}$ there exists a symplectic orthogonal matrix $W\in\R{\Nr}{\Nr}$ such that
$W^\top S W$ has the PVL form
\begin{equation}\label{eq:PVL}
    W^\top S W =
    \begin{pmatrix}
       S_{\Nrh} & R\\
       & S_{\Nrh}^\top
    \end{pmatrix},
\end{equation}
where $S_{\Nrh}\in\R{\Nrh}{\Nrh}$ is an upper Hessenberg matrix.
\end{lemma}
In our case, since the matrix $S$ is symmetric, its PVL decomposition
\eqref{eq:PVL}
yields tridiagonal matrices with identical blocks $S_{\Nrht}=S_{\Nrht}^\top$.
We further diagonalize $S_{\Nrht}$ using orthogonal transformations to obtain
$S_{\Nrht}= T^\top D_{\Nrht}T$, with $T^\top T=\Idm_{\Nrht}$
and diagonal $D_{\Nrht}\in\R{\Nrht}{\Nrht}$. Hence,
\begin{equation*}
    S =
    W\begin{pmatrix}
       T^\top D_{\Nrht}T & \\
       & T^\top D_{\Nrht}T
    \end{pmatrix} W^\top=:QDQ^\top,\;\mbox{ with }
    \;
     Q:=W\begin{pmatrix}
       T^\top &\\
       & T^\top
    \end{pmatrix},
    \quad
    D:=\begin{pmatrix}
       D_{\Nrht} & \\
       & D_{\Nrht}
    \end{pmatrix}.
\end{equation*}
It can be easily verified that $Q\in\R{\Nrt}{\Nrt}$
is orthogonal and symplectic.
The PVL factorization \Cref{lem:PVL} can be implemented
as in, e.g., \cite[Algorithms 1 and 2]{BKM05},
with arithmetic complexity $O(\Nrht^3)$.
The factorization is based on
orthogonal symplectic transformations obtained from
Givens rotations \bl{\cite{Givens58}} and
symplectic Householder matrices, defined as
the direct sum of Householder reflections \cite{PVL81}.

Once the matrix $S$ has been brought in the PVL form,
we perform the $\reg$-regularization.
Introduce the diagonal matrix $D_{\Nrht,\reg}\in\R{\Nrht}{\Nrht}$ defined as,
\begin{equation*}
(D_{\Nrht,\reg})_i = \left\{
    \begin{array}{ll}
         (D_{\Nrht})_i & \mbox{if}\; (D_{\Nrht})_i>\reg\\
         \reg & \mbox{otherwise},
    \end{array}\right.
    \qquad \forall\, 1\leq i\leq \Nrht,
\end{equation*}
and let us denote with $D_{\reg}\in\R{\Nrt}{\Nrt}$ the diagonal matrix composed of two blocks, both equal to $D_{\Nrht,\reg}$.
The matrix $\Sreg:=Q D_{\reg} Q^\top \in\R{\Nrt}{\Nrt}$
is symmetric positive definite and skew-Hamiltonian.
Its distance to $S$
is bounded, in the Frobenius norm, as
$\norm{S-\Sreg} = \norm{Q(D-D_{\reg})Q^\top} = \norm{D-D_{\reg}}\leq \sqrt{m_{\reg}}\,\reg$,
where $m_{\reg}$ is the number of
elements of $D_{\Nrht}$ that are smaller than $\reg$.
Since the $\reg$-regularized matrix $\Sreg$ is invertible,
$\Sreg^{-1}$ exists and is skew-Hamiltonian.
This property allows to construct the vector field
$\Fcal_{\reg}:=(\Idm_{\Nf}-UU^\top)(\J{\Nf}YZ^{\top} - YZ^{\top}\J{\Nrt}^\top)\Sreg^{-1}\in\R{\Nf}{\Nrt}$
with the property that $\Fcal_{\reg}$ belongs to the tangent space
of the orthosymplectic $\Nf\times\Nrt$ matrix manifold.
To gauge the error introduced by approximating the velocity field $\Fcal$
in \eqref{eq:Fcal} with $\Fcal_{\reg}$, let 
us denote with $\Lcal$ the operator $\Lcal:=(\Idm_{\Nf}-UU^\top)(\J{\Nf}YZ^{\top} - YZ^{\top}\J{\Nrt}^\top)$, so that \eqref{eq:U} reads $\dot{U}S=\Lcal$.
Then, the error made in the evolution of the reduced basis \eqref{eq:U},
by the $\reg$-regularization, is
\begin{equation*}
\begin{aligned}
\norm{\Fcal_{\reg}S-\Lcal}
    & = \norm{\Lcal(\Sreg^{-1}S-\Idm_{\Nrt})}
    = \norm{\Lcal Q(D_{\reg}^{-1}D-\Idm_{\Nrt})Q^\top}\\
    & \leq\norm{\Lcal}\norm{D_{\reg}^{-1}D-\Idm_{\Nrt}}
    = \frac{\sqrt{2}}{\reg}\,\norm{\Lcal}\,
    \sqrt{\sum_{j=\Nrht-m_{\reg}+1}^{\Nrht}|D_j-\reg|^2}\,.
\end{aligned}
\end{equation*}
%
%
Observe that the resulting vector field $\Fcal_{\reg}$
belongs to the space $H_U$ by construction.
However, in finite precision arithmetic,
the distance of the computed $\Fcal_{\reg}$ from $H_U$
might be affected by a small error that depends on the
norm of the operators $\Lcal$ and $\Sreg$.
This rounding error can affect the symplecticity of the reduced basis
over time, whenever the matrix $S$ is severely ill conditioned.
To guarantee that the evolution of the reduced basis computed in finite precision remains on the manifold of orthosymplectic matrices with an error of the order of machine precision, we introduce a correction
of the velocity field $\Fcal_{\reg}$.
    Observe that any $X_U\in H_U$ is of the form
    $X_U=[F|\J{\Nf}^\top F]$, with $F\in\R{\Nf}{\Nrht}$ satisfying $U^\top F=0_{\Nrt\times \Nrht}$.
    Let us write $\Fcal_{\reg}$ as $\Fcal_{\reg}=[F|G]$, with $F^\top=[F_1^\top|F_2^\top]\in\R{\Nrht}{\Nf}$
    and $G^\top=[G_1^\top|G_2^\top]\in\R{\Nrht}{\Nf}$.
    Since $U^\top \Fcal_{\reg}=[U^\top F|U^\top G]=0_{\Nrt\times\Nrt}$, we can take $\Fcal_{\reg,\star}:=[F|\J{\Nf}^\top F]$. Alternatively,
    we can define $\Fcal_{\reg,\star}:=[W|\J{\Nf}^\top W]$ where $W^\top=[X^\top|-Y^\top]\in\R{\Nrht}{\Nf}$
    and
    $2 X:=F_1+G_2$, $2Y:=G_1-F_2$. It easily follows that, with either definitions, $\Fcal_{\reg,\star}$ belongs to $H_U$ and
    the error in the Frobenius norm is
    \begin{equation*}
        \norm{\Fcal_{\reg}-\Fcal_{\reg,\star}}^2 =
        \dfrac14 \norm{\Fcal_{\reg}\J{\Nrt}-\J{\Nf}\Fcal_{\reg}}^2
        =
        \norm{G-\J{\Nf}^\top F}^2.
    \end{equation*}
    We summarize the regularization scheme in \Cref{algo:reg}.
\begin{algorithm}
\caption{$\reg$-regularization}\label{algo:reg}
\begin{algorithmic}[1]
 \Procedure{\textsc{Regularization}}{$U\in\Ucalt, Z\in\R{\Nrt}{\Np},\reg$}
 \State Compute $S \gets ZZ^\top+\J{\Nrt}^\top ZZ^\top\J{\Nrt}$
 \If{$\rank{S}<\Nrt$}
 \State Compute the PVL factorization $Q D Q^\top = S$
 \State Set $\Sreg\gets QD_{\reg}Q^\top$ where $D_{\reg}$ is the $\reg$-regularization of $D$
 \State Compute $\Fcal_{\reg}\gets(\Idm_{\Nf}-UU^\top)(\J{\Nf}YZ^{\top} - YZ^{\top}\J{\Nrt}^\top)\Sreg^{-1}$ \label{line:Fe}
 \State Compute $\Fcal_{\reg,\star}$ by enforcing the skew-Hamiltonian constraint
 \State Set $\Fcal\gets \Fcal_{\reg,\star}$
 \Else
 \State Compute $\Fcal\gets(\Idm_{\Nf}-UU^\top)(\J{\Nf}YZ^{\top} - YZ^{\top}\J{\Nrt}^\top)S^{-1}$ \label{line:F}
 \EndIf
 \State \Return velocity field $\Fcal\in H_U$ 
 \EndProcedure
\end{algorithmic}
\end{algorithm}

\section{Rank-adaptivity}\label{sec:rank-adaptivity}

\bl{The dynamical reduced basis method that we have introduced in Section~\ref{sec:DLR} is based on approximating the full model solution, in each temporal interval $\Tcalt$, on a low-dimensional space of size $\Nrht$. The fact that the size of the reduced space can change over time allows to fully exploit the local low-rank nature of the solution. In this section,
we propose an algorithm to detect when the reduced space needs to be enlarged or reduced and how this operation is performed.}
%
The method is summarized in \Cref{algo:rank-update}.

Here we focus on the case where the current rank of the reduced solution is too small
to accurately reproduce the full model solution.
In cases where the rank is too large,
one can perform an $\reg$-regularization
following \Cref{algo:reg} or decrease the rank
by looking at the spectrum of the reduced state and remove the modes associated with the lowest singular values.

\subsection{Error indicator}\label{sec:err-indicator}

Error bounds for parabolic problems are long-established and have been widely used to certify global reduced basis methods, \emph{cf.} e.g. \cite{grepl2005posteriori,urban2014improved}. However, their extension to noncoercive problems often results in pessimistic bounds that cannot be used to properly assess the quality of the reduced approximation.
Few works have focused on the development of error estimates (not bounds) for reduced solutions of advection-dominated problems.
In this work, we propose an error indicator based on the linearized residual of the full model. A related approach, known as Dual-Weighted Residual method (DWR) \cite{meyer2003efficient}, consists in deriving an estimate of the approximation error
via the dual full model and the linearization of the error of a certain functional of interest (e.g. surface integral of the solution, stress, displacement, ...).
Despite the promising results of this approach, the arbitrariness in the choice of the functional clashes with the goal of having a procedure as general as possible.

We begin with the continuous full model \eqref{eq:HamSystemMatrix}
and, for its time integration, we consider the implicit RK scheme used in the temporal
discretization of the dynamical system for the expansion coefficients $Z$ in \eqref{eq:PRK-ZVh}, and having coefficients 
$(\{b_i\}_{i=1}^{\Ns},\{a_{ij}\}_{i,j=1}^{\Ns})$. Then,
assuming that $\Rcal_{\tm}\in\R{\Nf}{\Np}$ is known,
\begin{equation}\label{eqn:start_error_equation}
\begin{array}{lll}
& \Rcal_{\tau}=\Rcal_{\tau-1} + \dt \sum\limits_{i=1}^{s} b_i k_i, \\[0.5em]
&\qquad k_1 = \J{2N}\nabla_{\Rcal}\HamN( \Rcal_{\tau-1}),\\[0.5em]
&\qquad k_i = \J{2N}\nabla_{\Rcal}\HamN\bigg( \Rcal_{\tau-1}+\dt \sum\limits_{j=1}^{s} a_{i,j}k_j; \prmh \bigg) \qquad i=2,\dots, s.
\end{array}
\end{equation}
The discrete residual operator, in the temporal interval $\Tcalt$, is
\begin{equation}\label{eqn:residual_operator}
    \res_{\tau}(\Rcal_{\tau},\Rcal_{\tau-1};\prmh)
    = \Rcal_{\tau}-\Rcal_{\tau-1}-\dt \sum_{i=1}^{s}b_i k_i=0.
\end{equation}
We consider the linearization of the residual operator \eqref{eqn:residual_operator} at $\left(R_{\tau},R_{\tau-1}\right)$, where
$R_{\tau}$ is the approximate reduced solution at time $t^{\tau}$, obtained from \eqref{eq:PRK-ZVh} as
$R_{\tau}=U_{\tau}Z_{\tau}$; thereby
\begin{equation}\label{eqn:linearization_residual}
\begin{array}{lll}
    \res_{\tau}(\Rcal_{\tau},\Rcal_{\tau-1};\prmh)
    = & \res_{\tau}(R_{\tau},R_{\tau-1};\prmh)
     + \dfrac{\partial \res_{\tau}}{\partial \Rcal_{\tau}}
    \biggr\rvert_{\left(R_{\tau},R_{\tau-1}\right)}
    \left( \Rcal_{\tau}-R_{\tau}\right)\\[0.5em]
    & + \dfrac{\partial \res_{\tau}}{\partial \Rcal_{\tau-1}}
    \biggr\rvert_{\left(R_{\tau},R_{\tau-1}\right)}
    \left( \Rcal_{\tau-1}-R_{\tau-1}\right)
    + \mathcal{O}\left(\left\| \Rcal_{\tau}-R_{\tau}\right\|^2
    + \left\| \Rcal_{\tau-1}-R_{\tau-1}\right\|^2\right).
\end{array}
\end{equation}
Similar procedures have been adopted in the formulation of the piecewise linear methods for the approximation of nonlinear operators, providing accurate approximations in case of low-order nonlinearities.
From the residual operator, an approximation of the local error $\Rcal_{\tau}-R_{\tau}$ is given by
the matrix-valued quantity $\mathbf{E}_{\tau}$
defined as
\begin{equation}\label{eqn:error_equation}
\mathbf{E}_{\tau} :=
-\bigg(\dfrac{\partial \res_{\tau}}{\partial \Rcal_{\tau}}
\biggr\rvert_{\left(R_{\tau},R_{\tau-1}\right)}\bigg)^{-1}  \bigg(\res_{\tau}(R_{\tau},R_{\tau-1};\prmh)+
\dfrac{\partial \res_{\tau}}{\partial \Rcal_{\tau-1}}
\biggr\rvert_{\left(R_{\tau},R_{\tau-1}\right)}
\bl{\mathbf{E}_{\tau-1}}\bigg),
\end{equation}
\bl{with $\mathbf{E}_0:=\Rcal(t_{0})-U_0Z_0$.}
The quantity defined by \eqref{eqn:error_equation} is the first order approximation of the error between the reduced and the full model solution. In particular, it quantifies the discrepancy due to the local approximation \eqref{eq:Rrb}.
Even if the linearization error is negligible, the computational cost related to the assembly of the entire full-order residual $\rho$ and its Jacobian, together with the solution of a linear system for any instance of the $\Np$ parameters $\prmh$,
makes the indicator unappealing if used in the context of highly efficient reduced approximations.
In \cite{meyer2003efficient}, a hierarchical approach has been proposed to alleviate the aforementioned computational bottleneck
but it relies on the offline phase to capture the dominant modes of the exact error.
Instead, in this work, we solve \eqref{eqn:error_equation} on a subset \bl{$\prmhsub$}
of the $\Np$ vector-valued parameters $\prmh$ of cardinality $\widetilde{\Np}\ll \Np$, and only \bl{each $\freqE$ time steps} during the simulation.
To further reduce the computational cost, we compute \eqref{eqn:error_equation} on a coarse mesh in the parameter domain, whenever possible, and then $\mathbf{E}_{\tau}$ is recovered on the original mesh via spline interpolation.
Although the assembly and solution of the sparse linear system in \eqref{eqn:error_equation} has, for example, arithmetic complexity $\mathcal{O}(N^{\frac{1}{2}})$ \cite{george1981computer} for problems originating from the discretization of two-dimensional PDEs,
this sampling strategy allows to reduce the computational cost required by the error estimator as compared to the evolution of the reduced basis and the coefficients, as discussed in \Cref{sec:numerical_tests}. 

\subsection{Criterion for rank update}\label{sec:criterion_rank_update}

Let $\mathbf{E}_{\tau}\in\R{\Nf}{\Np}$ be the error indicator matrix obtained
in \eqref{eqn:error_equation}.
To decide when to activate the rank update algorithm, we take into account that, for advection-dominated and hyperbolic problems discretized using spectral methods, the error accumulates, and the effect of unresolved modes on the resolved dynamic contributes to this accumulation \cite{couplet2003intermodal}. Moreover, it has been noticed \cite{spantini2013preconditioning} that, for many problems of practical interest, the modes associated with initially negligible singular values might become relevant over time, potentially causing a loss of accuracy if a reduced manifold of fixed dimension is employed.

Let us define $t^{\tau}$ as the current time, $t^{*}$ as the last time at which the dimension of the reduced basis $U$ was updated and let $\lambda_{\tau}$ be the number of past updates at time $t^{\tau}$.
At the beginning of the simulation $t^{*}=t^{0}$ and $\lambda_{0}=0$. The rank update is performed if the ratio between the norms of error indicators at $t^{\tau}$ and $t^{*}$ satisfies the criterion
\begin{equation}\label{eq:ratio_tmp}
    \dfrac{\norm{\mathbf{E}_{\tau}}}{\norm{\mathbf{E}_{*}}} > r c^{\lambda_{\tau}}\,,
\end{equation}
where $r, c\in\mathbb{R}$ are control parameters \bl{larger than 1}.
The ratio of the norms of the error indicator gives a qualitatively indication of how the error is increasing in time and \eqref{eq:ratio_tmp} fixes a maximum acceptable growing slope. 
Deciding what represents an acceptable slope is a problem-dependent task
but the numerical results in Section \ref{sec:numerical_tests} show little sensitivity of the algorithm with respect to $r$ and $c$. Moreover, the variable $\lambda_{\tau}$ induces a frequent rank-update when $n_{\tau}$ is small and vice versa when $n_{\tau}$ is large, hence controlling both the efficiency and the accuracy of the updating algorithm. \bl{We postpone to future investigations greedy strategies for the selection of optimal control parameters}.
Note that other (combinations of) criteria are possible: one alternative is to check that
the norm of the error indicator remains below a fixed threshold;
another possibility is to control the norm of some approximate gradient of the error indicator, etc. By numerically testing these various criteria, we
observe that, at least in the numerical simulations performed,
the criterion \eqref{eq:ratio_tmp} based on the ratio of error indicators is
reliable and robust and gives the largest flexibility.

\subsection{Update of the reduced state}\label{sec:rank-update}
If criterion \eqref{eq:ratio_tmp} is satisfied, the rank adaptive algorithm updates the current reduced solution to a new state having a different rank.
Specifically, assume that, in the time interval $\Tcal_{\tm}$, we have solved the
discrete reduced problem \eqref{eq:PRK-ZVh} to
obtain the reduced solution $R_{\tm}=U_{\tm} Z_{\tm}$ in $\Mcal_{\Nrh_{\tm}}$.

As a first step, we derive an updated basis
$U\in\Ucalt$
from $U_{\tm}\in\Ucal_{\tau-1}$,
with $\Nrht=\Nrh_{\tau-1}+1$.
To this aim, we enlarge
$U_{\tm}$ with two extra columns derived from
an approximation of the error, analogously to a greedy strategy.
%
In greater detail, with the algorithm described in \Cref{sec:err-indicator},
we derive the error matrix $\mathbf{E}_{\tau}$
associated with the reduced solution at the current time.
Via a thin SVD, we extract the left singular vector associated with the principal component of the error matrix, and we normalize it in the $2$-norm
to obtain the vector $e\in\r{\Nf}$.
We finally enlarge the basis $U_{\tm}$ with the two columns $[e\,|\,\J{\Nf}^\top e]\in \R{\Nf}{2}$.
The rationale for this choice is
that we seek to increase the accuracy of the low-rank approximation by
adding to the reduced basis the direction
that is worst approximated by the current reduced space.
Numerical evidence of the improved quality of the updated basis in approximating the
full model solution is provided in \Cref{sec:SW1D}.

From the updated matrix $[U_{\tm}|\,e\,|\,\J{\Nf}^\top e]\in\R{\Nf}{\Nrt}$,
we construct an orthosymplectic basis in the sense of \Cref{def:ortsym},
by performing a QR-like decomposition using symplectic unitary transformations. 
In particular, we employ
a symplectic (modified) Gram-Schmidt algorithm \cite{Salam05},
with the possibility of adding
reorthogonalization
\cite{Giraud03}
to enhance the stability and robustness of the algorithm.

Once the updated reduced basis $U\in\Ucalt$ is computed,
we derive the matrix $Z\in\R{\Nrt}{\Np}$
by expanding the current reduced solution $R_{\tm}$
in the updated basis.
Therefore, the updated $Z$ satisfies
$UZ = R_{\tm}$,
which results in $Z = U^\top R_{\tm}$.

\begin{remark}
Since the updated reduced state coincides with the reduced solution $R_{\tm}$ at time $t^{\tm}$,
all invariants of \eqref{eq:HamSystemMatrix} preserved by the partitioned Runge--Kutta scheme \eqref{eq:PRK-ZVh} are conserved during the rank update.
\end{remark}
Observe that, even if the current reduced state $R_{\tm}$ is in $\Mcal_{\Nrmt}$, it does not
belong to the manifold $\Mcal_{\Nrt}$.
Indeed, one easily shows that $Z=U^\top R_{\tm}\in\R{\Nrt}{\Np}$
does not satisfy the full-rank condition,
    \begin{equation*}
    \begin{aligned}
        \rank{S(Z)}
        & =\rank{U^\top U_{\tm} [Z_{\tm} Z_{\tm}^\top +
        \J{\Nrt}^\top Z_{\tm} Z_{\tm}^\top \J{\Nrt}] U_{\tm}^\top U}\\
        &\leq \min\{\rank{U^\top U_{\tm}}, \rank{Z_{\tm} Z_{\tm}^\top +
        \J{\Nrt}^\top Z_{\tm} Z_{\tm}^\top \J{\Nrt}}\}
        \leq \Nrmt.
    \end{aligned}
    \end{equation*}
As shown in \Cref{lem:FinTM}, the fact that $Z\notin\Zcalt$ implies that
the velocity field $\Fcal$ in \eqref{eq:Fcal}, describing the evolution of the reduced basis, is not well-defined. Therefore, we need to introduce an approximate
velocity field for the solution of the reduced problem \eqref{eq:UZred}
in the temporal interval $\Tcal_{\tau}$
with initial conditions $(U,Z)\in\Ucalt\times \R{\Nrt}{\Np}$.
We refer to \Cref{sec:evol-rank-deficient} for a discussion about this issue and
the description of the algorithm designed to
solve the rank-deficient reduced
dynamics ensuing from the rank update.

\begin{algorithm}
\caption{Rank update}\label{algo:rank-update}
\begin{algorithmic}[1]
 \Procedure{\textsc{Rank\_update}}{$U_{\tm}, Z_{\tm}, \mathbf{E}_*,\bl{\lambda_{\tm},\mathbf{E}_{\tm},r,c}$}
 \State \parbox[t]{.85\linewidth}{Compute the error indicator matrix $\mathbf{E}_{\bl{\tau}}\in\R{\Nf}{\bl{\widetilde{\Np}}}$ in  \eqref{eqn:error_equation}}
 \If{criterion \eqref{eq:ratio_tmp} is satisfied}
 \State Compute $Q\Sigma V^\top = \mathbf{E}_{\bl{\tau}}$ via thin SVD \label{line:thinSVD}
 \State Set $e\gets Q_1/\norm{Q_1}_2$ where $Q_1\in\r{\Nf}$ is the first column of the matrix $Q$
 \State Construct the enlarged basis $\overline{U}\gets [U_{\tm}|\,e\,|\,\J{\Nf}^\top e]\in\R{\Nf}{(2\Nrh_{\tm}+2)}$
 \State Compute $U$ via symplectic orthogonalization of $\overline{U}$ with symplectic Gram-Schmidt
 \State Compute the coefficients $Z\gets U^\top U_{\tm}Z_{\tm}$ \label{line:Z}
 \State \bl{Update reference error indicator matrix $\mathbf{E}_{*}\gets \mathbf{E}_{\tau}$}
 \State Set $\Nrht=\Nrh_{\tm}+1$ \bl{and $\lambda_{\tau}=\lambda_{\tm}+1$}
 \Else
 \State $U\gets U_{\tm}$, $Z\gets Z_{\tm}$, $\Nrht=\Nrh_{\tm}$ \bl{and $\lambda_{\tau}=\lambda_{\tm}$}
 \EndIf
 \State \Return updated factors $(U,Z)\in\Ucalt\times \R{\Nrt}{\Np}$, \bl{$\mathbf{E}_{*},\mathbf{E}_{\tau}\in\mathbb{R}^{2N\times \bl{\widetilde{\Np}}}$ and $\lambda_{\tau}$}
 \EndProcedure
\end{algorithmic}
\end{algorithm}


\subsection{Approximation properties of the rank-adaptive scheme}
To gauge the local approximation properties of the rank-adaptive scheme for the solution of the reduced dynamical system \eqref{eq:UZred},
we consider the temporal interval $\Tcalt$ where the first rank update is performed.
In other words, assume that $R_{\tm}=U_{\tm}Z_{\tm}$, with
$(U_{\tm},Z_{\tm})\in\Ucal_{\tm}\times\Zcal_{\tm}$,
is the numerical
approximation of the solution $R(t^{\tm})\in\Mcal_{\Nr_{\tm}}$ of the reduced dynamical system \eqref{eq:dynn} at time $t^{\tm}$ with $\Nrh_{\tm}=\Nrh_{\tau-2}=\ldots=\Nrh_1$.
After the rank update at time $t^{\tm}$, the reduced state $R$ satisfies the local evolution problem
\begin{equation}\label{eq:R}
	\left\{
	\begin{array}{ll}
		\dot{R}(t) =\Pcal^{\reg}_{R}\Xcal_{\HamN}(R(t),\prm_h),\qquad\quad\mbox{for }\; t\in\Tcalt,\\
		R(t^{\tm}) = R_{\tm}=U_{\tm}^{\Nrht}Z_{\tm}^{\Nrht}, &
	\end{array}\right.	
\end{equation}
where $(U_{\tm}^{\Nrht},Z_{\tm}^{\Nrht})\in\Ucalt\times\R{\Nrt}{\Np}$ are the rank-updated factors, and
\begin{equation*}
\Pcal^{\reg}_{R}\Xcal_{\HamN}:=
(\Idm_{\Nf}-UU^\top)(\Xcal_{\HamN}Z^\top + \J{\Nf}\Xcal_{\HamN}Z^\top\J{\Nrt}^\top)\Sreg(Z)^{-1}Z+ UU^\top \Xcal_{\HamN},\qquad\forall\, R=UZ\in\mathbb{R}^{\Nf}{\Np}.
\end{equation*}
We make the assumption that the reduced problem \eqref{eq:dynn} is well-posed.
Let $\Rcal(t)\in\Vcal^{\Np}_{\Nf}$ be the full model solution of problem
\eqref{eq:HamSystemMatrix} in the temporal interval $\Tcalt$ with given initial condition
$\Rcal(t^{\tm})$.
The error between the approximate reduced solution of \eqref{eq:R}
and the full model solution
at time $t^{\tau}\in\Tcal$ is given by
\begin{equation*}
R_{\tau}-\Rcal(t^{\tau}) = \big(R_{\tau}-R(t^{\tau})\big)
+ \big(R(t^{\tau})-\Rcal(t^{\tau})\big).
\end{equation*}
The quantity $e^{\tau}_{\textrm{A}}:=R_{\tau}-R(t^{\tau})$ is the approximation error associated with the partitioned Runge--Kutta discretization scheme,
and can be treated using standard convergence analysis techniques, in light of the fact that the retraction map is Lipschitz continuous in the Frobenius norm, as shown in \cite[Proposition 5.7]{P19}.
The term $e_{\textrm{RA}}(t):=R(t)-\Rcal(t)$, for any $t\in\Tcalt$,
is associated with the rank update and can be bounded as
\begin{equation*}
    \begin{aligned}
       d_t\norm{e_{\textrm{RA}}}
       & \leq \norm{\Pcal^{\reg}_{R}\Xcal_{\HamN}(R) - \Xcal_{\HamN}(\Rcal)}
       \leq
       \norm{\Pcal^{\reg}_{R}\Xcal_{\HamN}(R) - \Xcal_{\HamN}(R)} +
       \norm{\Xcal_{\HamN}(R)-\Xcal_{\HamN}(\Rcal)}\\
       & \leq L_{\Xcal_{\HamN}}\norm{e_{\textrm{RA}}} +
       \norm{(\Idm_{\Nf}-\Pcal^{\reg}_{R})\Xcal_{\HamN}(R)}, 
    \end{aligned}
\end{equation*}
where $L_{\Xcal_{\HamN}}$ is the Lipschitz continuity constant of $\Xcal_{\HamN}$.
Gronwall's inequality \cite{Gro19} gives, for all $t\in\Tcalt$,
\begin{equation}\label{eq:err}
\norm{e_{\textrm{RA}}(t)} \leq \norm{e_{\textrm{RA}}(t_0)}\,e^{L_{\Xcal_{\HamN}} t}+
		\int_{t^{\tm}}^{t^{\tau}} e^{L_{\Xcal_{\HamN}}(t-s)}\norm{(\Idm_{\Nf}-\Pcal^{\reg}_{R})\Xcal_{\HamN}(R)}\,ds.
\end{equation}
Observe that the estimate \eqref{eq:err} depends on the distance between the Hamiltonian vector field at the reduced state and its image under the map $\Pcal^{\reg}_{R}$
that approximates the orthogonal projection operator on the tangent space of $\Mcal_{\Nrt}$.
Although a rigorous bound for this term is not available, 
we expect that it can be controlled arbitrary well by increasing the size of the reduced basis, as will also be demonstrated
in \Cref{sec:numerical_tests}.
Moreover, the estimate \eqref{eq:err} on the whole temporal interval $\Tcal$ depends exponentially on the final time $T$. A linear dependence on $T$ can be obtained only in special cases, for example when $\nabla_{\Rcal} \HamN$ is uniformly negative monotone.


\section{Computational complexity of the rank-adaptive algorithm}\label{sec:cost}
In this \bl{s}ection we discuss the computational cost required
for the numerical solution of the reduced problem \eqref{eq:UZred}
with the rank-adaptive algorithm introduced in \Cref{sec:rank-adaptivity}.

In each temporal interval $\Tcalt$, the algorithm
consists of two main steps: the evolution step, which
entails the repeated evaluation of
the velocity fields $\Fcal$ and $\Gcal$ in \eqref{eq:PRK-ZVh} at each stage of the Runge--Kutta temporal integrator,
and the rank update step, which requires the evaluation of the error indicator
and the update of the approximate reduced solution at the current time step.

The rank update strategy introduced in \Cref{sec:rank-adaptivity},
and summarized in \Cref{algo:rank-update},
has an arithmetic complexity of $O(\Nfh\Np^2)+ O(\Nfh\Nrht^2) + O(\Nfh\Np\Nrht)$,
and the computational bottleneck is the computation of the error indicator.
As suggested in \Cref{sec:err-indicator}, sub-sampling techniques \bl{and mesh coarsening} can be employed to overcome this limitation.
%
The evolution step consists in solving the discrete reduced system \eqref{eq:PRK-ZVh} in each temporal interval.
To understand the computational complexity of this step, we neglect the
number of nonlinear iterations required by the implicit temporal integrators
for the evolution of the coefficients $Z$.
The solution of \eqref{eq:PRK-ZVh} requires the evaluation of
four operators: the velocity fields $\Gcal$
and $\Fcal$, the retraction $\Rcal$ and its inverse tangent map $f_{\tau}$.
The algorithms proposed in \cite[Section 5.3.1]{P19} for the computation of $\Rcal$
and $f_{\tau}$ have arithmetic complexity $O(\Nfh\Nrht^2)$.
We denote with $C_{\HamN}=C_{\HamN}(\Nfh,\Nrht,\Np)$
the computational cost to evaluate the gradient of the reduced Hamiltonian
at the reduced solution.
Finally, the velocity field $\Fcal$ is computed via \Cref{algo:reg}
with a computational complexity of
$O(\Nfh\Nrht\Np)+O(\Nfh\Nrht^2)+ O(\Np\Nrht^2)+O(\Nrht^3)$,
while $C_{\Ham}$ is the cost to evaluate $Y$.
It follows that the rank-adaptive algorithm for the solution of the reduced system \eqref{eq:PRK-ZV} with a partitioned Runge--Kutta scheme has a computational complexity
being at most linear in the dimension of the full model $\Nfh$,
provided the computational cost $C_{\HamN}$ to evaluate the Hamiltonian vector field
at the reduced solution has a comparable cost.
Concerning the latter, observe that the assembly of the reduced state $R$ from the factors $U$ and $Z$ and the matrix-vector multiplication $U^\top\nabla_R\HamN(R;\prmh)$ require
$O(\Nfh\Np\Nrht)$ operations. Therefore,
the computational bottleneck of the algorithm is associated
with the evaluation of the Hamiltonian gradient at the reduced state $R$.

This problem is well-known in model order reduction and emerges whenever reduced models involve non-affine and nonlinear operators, \emph{cf.}  e.g. \cite[Chapters 10 and 11]{QMN16}.
Several hyper-reduction techniques have been proposed to mitigate or overcome this
limitation, resulting in approximations of nonlinear operators that can be evaluated at a cost
independent of the size of the full model. However,
we are not aware of any hyper-reduction method able to \emph{exactly} preserve
the Hamiltonian phase space structure during model reduction.
Furthermore, hyper-reduction methods entail an offline phase to
learn the low-rank structure of the nonlinear operators by means of snapshots of the
full model solution. 
Compared to traditional \emph{global} model order reduction,
in a dynamical reduced basis approach the constraints on the computational complexity of the reduced operators is less severe since we allow the dimension of the full model
to enter, albeit at most linearly, the computational cost of the operations involved.
This means that the dynamical model order reduction can accommodate Hamiltonian gradients
where each vector entry depends only on a few, say $k\ll\Nfh$, components of the reduced solution, with a resulting computational cost of $C_{\HamN}=O(\Nfh\Np\Nrht)+O(k\Nfh\Np)$.
This is the case when, for example, the dynamical system \eqref{eq:HamSystem}
ensues from a \emph{local} discretization of a partial differential equation in Hamiltonian form. Note that this assumption is also required for the effective application of
discrete empirical interpolation methods (DEIM) \cite{ChSo10}.

When dealing with low-order polynomial nonlinearities of the Hamiltonian vector field, we can use tensorial techniques to 
perform the most expensive operations only once and not at each instance of the parameter, as discussed in the following.


\subsection{Efficient treatment of polynomial nonlinearities}\label{sec:nonlinear}

Let us consider the explicit expression
of the cost $C_{\Ham}$ for different Hamiltonian functions $\HamN$.
If the Hamiltonian vector field $\Xcal_{\Ham}$
in \eqref{eq:HamSystemMatrix} 
is linear, then
\begin{equation*}
    \Gcal(U,Z;\prmh) = \J{\Nr} U^\top \nabla_R\HamN(R;\prmh) = \J{\Nr} U^\top A UZ,\qquad \forall\, R=UZ\in\Mcal_{\Nrt},
 \end{equation*}   
where $A\in\R{\Nf}{\Nf}$ is a given linear application,
associated with the spatial discretization of the Hamiltonian function $\HamN$. Standard matrix-matrix multiplication to compute $\Gcal$ has arithmetic complexity
$O(\Nfh\Nrht^2) + O(\Np\Nrht^2) + O(\Nrht k)$,
where $k$ is the number of nonzero
entries of the matrix $A$. The computational complexity of the algorithm
is therefore still linear in $\Nfh$ provided the matrix $A$ is sparse.
This is the case in applications we are interested in where
the Hamiltonian system \eqref{eq:HamSystemMatrix} ensues from a
local spatial approximation of a partial differential equation.
    
In case of low-order polynomial nonlinearities, we use the tensorial representation \cite{cstefuanescu2014comparison} of the nonlinear function and rearrange the order of computing.
    The gist of this approach is to exploit the structure of the polynomial nonlinearities to separate the quantities
    that depend on the dimension of the full model from the reduced variables, by manipulating the order of computation of the various factors.
    Consider the evolution equations for the coefficients $Z$ in \eqref{eq:Z} for a single value $\prm_j$ of the parameter $\prmh\in\Sprmh$. The corresponding reduced Hamiltonian vector
    can be expressed in the form
    \begin{equation}\label{eqn:polynomial_nonlinearity}
        \J{\Nr}\nabla_{Z_j}\HamN_U(Z_j;\prm_j)=U^{T} J_{2N}G^{\{q\}}\bigg (\mathop{\bigotimes}\limits_{i=1}^{q}A_{i}UZ_j\bigg)
        =\underbrace{U^{T} J_{2N}G^{\{q\}}\bigg (\mathop{\bigotimes}\limits_{i=1}^{q}A_{i}U\bigg)}_{\mathcal{G}_U}\underbrace{\bigg( \mathop{\bigotimes}\limits_{i=1}^{q} Z_j \bigg)}_{\mathcal{Z}},
    \end{equation}
    where $Z_j\in\Zcalt$ with $\Np=1$, $q\in\mathbb{N}$ is the polynomial degree of the nonlinearity, $A_i\in\mathbb{R}^{\Nf\times\Nf}$ are sparse discrete differential operators, $G^{\{q\}}$ represents the matricized $q$-order tensor and $\otimes$ denotes the Kronecker product.
    The last expression in \eqref{eqn:polynomial_nonlinearity}
    allows to separate the computations involving factors of size $\Nfh$ from the reduced coefficients $Z$, so that the matrix $\mathcal{G}_U\in\mathbb{R}^{\Nrt\times (\Nrt)^{q}}$ can be precomputed during the offline phase.
   
    In the case of the proposed dynamical reduced basis method, we employ the tensorial POD approach to reduce the computational complexity of the evaluation of $\Gcal$, the RHS of \eqref{eq:Z}, and its Jacobian needed
    in the implicit symplectic integrator at each time step of the numerical integrator. We start by noticing that a straightforward calculation of the second expression in \eqref{eqn:polynomial_nonlinearity} suggests $O(c\Nfh\Np\Nrht)+O(c\Np q k)+O(c \Nfh \Np q)$ operations, where the first term is due to the reduced basis ansatz and the Galerkin projection, the second term to the multiplication by the sparse matrices $A_i$ and the third term to the evaluation of a polynomial of degree $q$ for each entry of a $\Nf\times\Np$ matrix. The constant $c$ represents the number of iterations of the Newton solver and
    $k:=\max_i k_i$, where $k_i$ is the number of nonzero entries of $A_i$.
    Moreover, in each iteration we evaluate not only the nonlinear term but also its Jacobian, with an additional cost of $O(c\Nfh\Np(q-1))+O(c\Np k_{\mathcal{G}}\Nrht)+O(c\Nfh\Np\Nrht^2)$ operations, with $k_{\mathcal{G}}$ being the number of nonzero entries of the full-order Jacobian.
    These terms represent, respectively, the operations required to evalute the polynomial functions in the Jacobian, the assembly of the Jacobian matrix and its Galerkin projection onto the reduced basis.
    This high computational cost can again be mitigated by resorting to the second formula in \eqref{eqn:polynomial_nonlinearity}, where the term $\mathcal{G}_U$ is precomputed at each iteration, for each stage of the partitioned RK integrator \eqref{eq:PRK-ZVh}.
    To estimate the computational cost of the procedure we resort to the multi-index notation by introducing $\mathbf{n}:=\left(n_{\tau},\dots,n_{\tau}\right)\in\mathbb{R}^{n}$ and hence $\mathcal{G}_U\mathcal{Z}$ in \eqref{eqn:polynomial_nonlinearity} can be recast as
    \begin{equation}\label{eqn:matrix_tensorial_POD}
       \mathcal{G}_U\mathcal{Z} = \underbrace{ U^TJ_{\Nr} \sum_{\ell\leq 2\mathbf{n}}}_{\text{(III)}} \prod_{1<i\leq q} \overbrace{ \text{diag}\underbrace{\left( A_{i}U_{\ell} \right)}_{\text{(I)}} \underbrace{A_1 U_{\ell}}_{\text{(I)}}}^{\text{(II)}} Z_{j}^{\ell}.
    \end{equation}
    The arithmetic complexity of this step is $O(qk\Nrht)+O((q-1)\Nfh\Nrht^{q})+O(\Nfh\Nrht^{q+1})$, where the first term is due to the matrix multiplication of the $q$ matrices $A_iU$ in (I), the second term to the pointwise and diagonal matrices multiplications involved in the computations of (II) and the third term to the  multiplications by $U^TJ_{2N}$ in (III). We stress that the cost required to assemble $\mathcal{G}_U$ is independent of the number of parameters $\Np$ and the number of iterations of the nonlinear solver. Once $\mathcal{G}_U$ has been precomputed, the evaluation of the reduced RHS has a computational cost of $O(c\Np\Nrht^{q+1})$ \cite{cstefuanescu2014comparison}.
    The same splitting technique is exploited for each evaluation of the reduced Jacobian and most of the precomputed terms in \eqref{eqn:matrix_tensorial_POD} can be reused. The proposed treatment of polynomial nonlinearities results in an effective reduction of the computational cost in case of low-order polynomial nonlinearity $(q=2,3)$, a large set of vector-valued parameters $(\Np\gg 10)$ and a moderate number $\Nrht$ of basis vectors.


\section{Numerical tests}\label{sec:numerical_tests}

To assess the performance of the proposed adaptive dynamical structure preserving reduced basis method,
we consider finite-dimensional parametrized Hamiltonian dynamical systems arising from the spatial approximation of partial differential equations.
Let $\Omega\subset\r{d}$ be a continuous domain and
let $u:\Tcal\times \Omega\times\Sprm \rightarrow \r{m}$ belong to a Sobolev space $\Vcal$
endowed with the inner product $\big<\cdot,\cdot\big>$.
A parametric evolutionary PDE in Hamiltonian form can be written as
\begin{equation}\label{eq:HamPDE}
    \left\{
    \begin{aligned}
    & \bl{\dot{u}}(t,x;\prm) = \Jcal \dfrac{\delta\Hcal}{\delta u}(u;\prm), & \qquad\mbox{in}\;\Omega\times \Tcal,\\
    & u(0,x;\prm) = u^0(x;\prm), &\qquad\mbox{in}\;\Omega,
    \end{aligned}\right.
\end{equation}
with suitable boundary conditions prescribed at the boundary $\partial\Omega$.
Here, the dot denotes the derivative with respect to time, and $\delta$ denotes the variational derivative of the Hamiltonian $\Hcal$ defined as
\begin{equation*}
    \dfrac{d}{d\epsilon} \Hcal(u+\epsilon v;\prm) \bigg|_{\epsilon=0} =
    \bigg<\dfrac{\delta\Hcal}{\delta u},v\bigg>,\qquad\forall\, u,v\in\Vcal,
\end{equation*}
so that, for $\ell=1,\ldots,m$ and $u_{\ell,k}:=\partial_{x_k} u_{\ell}$, it holds
\begin{equation*}
    \dfrac{\delta\Hcal}{\delta u_{\ell}} = \dfrac{\partial H}{\partial u_{\ell}} 
    - \sum_{k=1}^d \dfrac{\partial}{\partial x_k}\left(\dfrac{\partial H}{\partial u_{\ell,k}}\right) + \ldots,\qquad
    \mbox{with}\qquad
    \Hcal(u;\prm) = \int_{\Omega} H(x,u,\partial_x u,\partial_{xx} u,\ldots;\prm)\,dx.
\end{equation*}
In the numerical tests, we consider, for any fixed value of the parameter $\prm_j\in\Sprmh$, numerical spatial approximations of
\eqref{eq:HamPDE} that yield a $\Nf$-dimensional Hamiltonian system in canonical form
\begin{equation}\label{eq:HamODE}
    \left\{
    \begin{aligned}
    & \bl{\dot{u}_h}(t;\prm_j) = \J{\Nf} \nabla\Hcal_h(u_h;\prm_j), &\qquad \mbox{in}\;\Tcal,\\
    & u_h(0;\prm_j) = u_h^0(\prm_j), &
    \end{aligned}\right.
\end{equation}
where $u_h$ belongs to a finite $\Nf$-dimensional subspace of $\Vcal$,
$\nabla_{u}$ is the gradient with respect to the state variable $u_h$
and $\Hcal_h:\r{\Nf}\rightarrow\r{}$ is such that $\Delta x_1\ldots\Delta x_d\Hcal_h$ is a suitable approximation of $\Hcal$.
Testing \eqref{eq:HamODE} for $\Np$ values $\prmh = \{\prm_j\}_{j=1}^{\Np}$
of the parameter, yields a matrix-valued ODE of the form
\eqref{eq:HamSystemMatrix}, where the $j$-th column of the unknown
matrix $\Rcal(t)\in\R{\Nf}{\Np}$ is equal to $u_h(t,\prm_j)$ for all $j=1,\ldots,\Np$.

We validate our adaptive dynamical reduced basis method on several representative Hamiltonian systems of the form \eqref{eq:HamODE}, of increasing complexity, and compare the quality of the adaptive dynamical approach with a reduced model with a global basis. \bl{The proposed approach, including all the steps introduced in the previous sections, is summarized in Algorithm \ref{algo:summary}}. For the global model, we consider the method proposed in \cite[Section 4.2]{peng2016symplectic}, where a reduced basis is built via a complex SVD of a suitable matrix of snapshots and the reduced model is derived via symplectic Galerkin projection onto the space spanned by the global basis.
We analyze and compare the accuracy, conservation properties and efficiency of the reduced models by monitoring the various quantities.
To assess the approximation properties of the reduced model,
we track the error, in the Frobenius norm, between the full model solution $\Rcal$ and
the reduced solution $R$ at any time $t\in\Tcal$, namely
\begin{equation}\label{eqn:error_metric}
E(t)=\left \| \Rcal(t) - R(t) \right \|.
\end{equation}
Moreover, we study the conservation of the Hamiltonian via the relative error in the $\ell^1$-norm in the parameter space $\Sprmh$, that is
\begin{equation}\label{eqn:relative_error_Hamiltonian}
    E_{\mathcal{H}_h}(t) = \sum_{i=1}^{\Np} \left | \dfrac{\mathcal{H}\left( U_{\tau}Z_{\tau}^i;\eta_i\right)-\mathcal{H}\left( U_{0}Z_{0}^i;\eta_i\right)}{\mathcal{H}\left( U_{0}Z_{0}^i;\eta_i\right)}\right |.
\end{equation}
Finally, we monitor the computational cost of the different reduction strategies. Throughout, the runtime is defined as the sum of the lengths of the offline and online phases in the case of the complex SVD (global method); while, for the dynamical approaches it is the time required to evolve basis and coefficients \eqref{eq:PRK-ZV} plus the time required to compute the error indicator and update the dimension of the approximating manifold, in the adaptive case.

The adaptive dynamical reduced basis method is numerically tested on two nonlinear problems, the shallow water and Schr\"odinger equations in one and two dimensions.
Finally, we consider a preliminary application to particle simulations of plasma physics problem with the reduction of the Vlasov equation with a forced external electric field,
modeling the evolution of charged particle beams.
All numerical simulations are performed using \textsc{Matlab} computing environment on computer nodes with Intel Xeon E5-2643 (3.40GHz). \bl{The code and the data supporting the findings of this study are available from the authors upon request.} 

\begin{algorithm}
\caption{\bl{Rank-adaptive reduced basis method}}\label{algo:summary}
\begin{algorithmic}[1]
\Procedure{\textsc{Rank-adaptive\_RBM}}{$\Rcal_0$, $\prmh$, $\prmhsub$, $\freqE$, $\Nrh_1$, $\reg$, $r$, $c$}
 \State Compute $U_0\in\Ucal_1$ via complex SVD of $\Rcal_0(\prmh)$ truncated at the $\Nrh_1$-th mode, and $Z_0 \gets U_0^\top \Rcal_0(\prm_h)$
 \State Initialize the error indicator matrix $\mathbf{E}_0 \gets \Rcal_0(\prmhsub)-U_0U_0^\top\Rcal_0(\prmhsub)\in \R{\Nf}{\widetilde{\Np}}$ and $\mathbf{E}_{*}\gets \mathbf{E}_{0}$
 \For{$\tau=1,\ldots,N_{\tau}$}
 \State \parbox[t]{.8\linewidth}{Calculate $(U_{\tau},Z_{\tau})\in\Ucalt\times \R{\Nrt}{\Np}$ using partitioned RK integrator \eqref{eq:PRK-ZVh}, starting from 
 $(U_{\tm},Z_{\tm})\in\Ucal_{\tau-1}\times \R{2n_{\tau-1}}{\Np}$:}
 \Statex \hspace{1.2cm}$\bigcdot$ Use the tensorial POD approach \eqref{eqn:polynomial_nonlinearity} to assemble the operator $\mathcal{G}$
 \Statex \hspace{1.2cm}$\bigcdot$ Use the retraction map given in \eqref{eq:retraction} to compute $\mathcal{R}_{U_{\tau-1}}$
 \Statex \hspace{1.2cm}$\bigcdot$ \parbox[t]{.75\linewidth}{Compute $f_{\tau}$ according to \eqref{eq:EvolTM}, using $\textsc{Regularization}$ (Algorithm \ref{algo:reg}), with parameter $\reg$ as input, to assemble $\mathcal{F}$}
 \vspace{0.15cm}
 \If{mod$(\tau,\freqE)=0$}
    \State \parbox[t]{.75\linewidth}{Compute the error indicator matrix and check the rank update criterion using \textsc{Rank\_update} (Algorithm~\ref{algo:rank-update}) as\\
    $(U_{\tau},Z_{\tau},\mathbf{E}_{*},\mathbf{E}_{\tau},\lambda_{\tau})=\textsc{Rank\_update}(U_{\tau},Z_{\tau},\mathbf{E}_{*},\mathbf{E}_{\tm},\lambda_{\tm},r,c)$}
 \EndIf
 \EndFor
 \EndProcedure
\end{algorithmic}
\end{algorithm}

\subsection{Shallow water equations}\label{sec:SW1D}

The shallow water equations (SWE) describe the kinematic behaviour of a thin inviscid single fluid layer flowing over a variable topography.
In the setting of irrotational flows and flat bottom topography, the fluid is described by a scalar potential $\phi$ and the canonical Hamiltonian formulation \eqref{eq:HamPDE} is recovered \cite{sultana2013hamiltonian}. 
The resulting time-dependent nonlinear system of PDEs is defined as
\begin{equation}\label{eq:SWE}
    \left\{
    \begin{aligned}
    & \dfrac{\partial h}{\partial t} + \nabla\cdot(h\nabla\phi)=0, &\qquad \mbox{in}\;\Omega\times \Tcal,\\
    & \dfrac{\partial \phi}{\partial t} + \dfrac12|\nabla\phi|^2 + h =0, &\qquad \mbox{in}\;\Omega\times \Tcal,\\
    & h(0,x;\prmh) = h^0(x;\prmh), &\qquad\mbox{in}\;\Omega,\\
    & \phi(0,x;\prmh) = \phi^0(x;\prmh), &\qquad\mbox{in}\;\Omega,
    \end{aligned}\right.
\end{equation}
 with spatial coordinates $x\in\Omega$, time $t\in\Tcal$, state variables $h,\phi:\Omega\times\Tcal\mapsto \mathbb{R}$,
 $\nabla \cdot$ and $\nabla$ divergence and gradient differential operators in $x$, respectively. The variable $\phi$ is the scalar potential of the fluid and $h$ represents the height of the free-surface, normalized by its mean value. The system is coupled with periodic boundary conditions for both the state variables.
The evolution problem \eqref{eq:SWE}
admits a canonical symplectic Hamiltonian form \eqref{eq:HamPDE}
with the Hamiltonian
\begin{equation}
    \HamN(h,\phi;\prm) = \dfrac12 \int_{\Omega} \big(
    h|\nabla\phi|^2+
    h^2 \big)\,dx.
\end{equation}
We consider numerical simulations in $d=1$ and $d=2$ dimensions
on rectangular spatial domains.
The domain $\Omega$ is partitioned using a Cartesian mesh in $M-1$ equispaced intervals in each dimension,
having mesh width $\Delta x$ and $\Delta y$, when $d=2$.
As degrees of freedom of the problem we consider the
nodal values of the height and potential, i.e.
$u_h(t;\prmh):=(h_h,\phi_h)=(h_1,\dots,h_N,\phi_1,\dots,\phi_N)$, for all $t\in\mathcal{T}$ and $\prmh\in\Sprmh$,
where $\Nfh:=M^d$, $h_m=h_{i,j}$ with $m:=(j-1)M+i$, and
$i,j=1,\ldots,M$. In 1D, $\Nfh=M$, and the index $j$ is dropped.

We consider second order accurate central finite difference schemes to discretize the differential operators in \eqref{eq:SWE}, and denote with $D_x$ and $D_y$ the discrete differential operators
acting in the $x$- and $y$-direction, respectively.
The semi-discrete formulation of \eqref{eq:SWE} represents a canonical Hamiltonian system with the gradient of the Hamiltonian function with respect to $u_h$ given by
\begin{equation}\label{eq:grad_Hamiltonian}
    \nabla\HamN_h(u_h;\prmh) = 
    \begin{pmatrix}
    \dfrac{1}{2}\left [ \left ( D_x \phi_h \right )^2 + \left ( D_y \phi_h \right )^2 \right ] + h_h \\
    -D_x\left ( h \odot D_x\phi_h \right )-D_y\left ( h \odot D_y\phi_h \right )
    \end{pmatrix},
\end{equation}
where $\odot$ is the Hadamard product between two vectors.
The discrete Hamiltonian is
\begin{equation}\label{eq:SW-1D_Ham}
    \HamN_h(u_h;\prmh) = \dfrac12 \sum_{i,j=1}^{M}
    \bigg(h_{i,j} \left[ \left( \dfrac{\phi_{i+1,j}-\phi_{i-1,j}}{2\Delta x} \right)^2 + \left( \dfrac{\phi_{i,j+1}-\phi_{i,j-1}}{2\Delta y} \right)^2 \right ] + h_{i,j}^2\bigg).
\end{equation}
In the one-dimensional case, the operator $D_y$ vanishes.
    
\subsubsection{One-dimensional shallow water equations (SWE-1D)}

For this example, we set $\Omega=\left [ -10, 10 \right ]$ and we consider the parameter domain $\Sprm = \left [ \frac{1}{10},\frac{1}{7} \right ] \times \left [ \frac{2}{10},\frac{15}{10} \right ]$. The discrete set of parameters $\Sprmh$ is obtained by uniformly sampling $\Sprm$ with $10$ samples per dimension, for a total of $\Np=100$ different configurations. Problem \eqref{eq:SWE} is completed with the initial condition
\begin{equation}\label{eq:init_cond_SWE1D}
    \begin{cases}
    h^{0}(x;\prmh) = 1+\alpha e^{-\beta x^2},\\
    \phi^{0}(x;\prmh) = 0,
    \end{cases}
\end{equation}
with $\prmh=(\alpha,\beta)$, where $\alpha$ controls the amplitude of the initial hump in the depth $h$ and $\beta$ describes its width.
We consider a partition of the spatial domain $\Omega$ into $\Nfh-1$ equispaced intervals
with \bl{$\Nfh=1001$}.
The full model solution $u_h(t;\prmh)$ is computed using a uniform step size $\Delta t = 10^{-3}$ in the time interval $\mathcal{T}=(0,T:=7]$. 
We use the implicit midpoint rule as time integrator because, being symplectic, it preserves the geometrical properties of the flow of the semi-discrete equation associated to \eqref{eq:grad_Hamiltonian}.
To study the reducibility properties of the problem, we explore the solution manifold
and collect the solutions to the high-fidelity model in different matrices.
The global snapshot matrix $\mathcal{S}\in\mathbb{R}^{\Nf\times(N_{\tau}\Np)}$
contains the snapshots associated with all sampled parameters $\prmh$ and time steps, while, for any $\tau=1,\ldots,N_{\tau}$, the matrix
$\mathcal{S}_{\tau}\in\mathbb{R}^{\Nf\times \Np}$ 
collects the full model solutions
at fixed time $t^{\tau}$. 
\begin{figure}[H]
\centering
\begin{tikzpicture}
    \begin{groupplot}[
      group style={group size=2 by 1,
                   horizontal sep=2cm},
      width=7cm, height=5cm
    ]
    \nextgroupplot[xlabel={index},
                   ylabel={singular values},
                   axis line style = thick,
                   grid=both,
                   minor tick num=2,
                   max space between ticks=20,
                   grid style = {gray,opacity=0.2},
                   ymode=log,
                   every axis plot/.append style={ultra thick},
                   xmin = 1,
                   ymin = 1e-14, ymax = 1,
                   legend style={at={(0.37,0.25)},anchor=north},
                   ylabel near ticks, 
                   xlabel style={font=\footnotesize},
                   ylabel style={font=\footnotesize},
                   x tick label style={font=\footnotesize},
                   y tick label style={font=\footnotesize},
                   legend style={font=\tiny}
                   ]
        \addplot[color=black] table[x=Index,y=Values] {figures/data/SW1D_singular_values/full_singular_values_SW1D.txt};
        \addplot[color=red] table[x=Index,y=Values] {figures/data/SW1D_singular_values/avg_singular_values_SW1D.txt};
        \legend{global,local};
        \node [text width=1em,anchor=north west] at (rel axis cs: 0.08,1.05) {\subcaption{\label{fig:singular_values_SW1D_a}}};
        \coordinate (spypoint) at (axis cs:2,0.01);
    \nextgroupplot[xlabel={time $\left [ s \right ]$},
                   ylabel={$\epsilon$-rank},
                   axis line style = thick,
                   grid=both,
                   minor tick num=2,
                   grid style = {gray,opacity=0.2},
                   every axis plot/.append style={ultra thick},
                   legend style={at={(1.16,1)},anchor=north},
                   xmin = 0, xmax = 7,
                   ymin = 0, ymax = 60,
                   xlabel style={font=\footnotesize},
                   ylabel style={font=\footnotesize},
                   x tick label style={font=\footnotesize},
                   y tick label style={font=\footnotesize},
                   legend style={font=\tiny}]
        \addplot+[] table[x=Time,y=Rank] {figures/data/SW1D_epsilon_rank/0.1_epsilon_rank_SW1D.txt};
        \addplot+[] table[x=Time,y=Rank] {figures/data/SW1D_epsilon_rank/0.001_epsilon_rank_SW1D.txt};
        \addplot+[] table[x=Time,y=Rank] {figures/data/SW1D_epsilon_rank/1e-05_epsilon_rank_SW1D.txt};
        \addplot+[] table[x=Time,y=Rank] {figures/data/SW1D_epsilon_rank/1e-07_epsilon_rank_SW1D.txt};
        \addplot+[] table[x=Time,y=Rank] {figures/data/SW1D_epsilon_rank/1e-09_epsilon_rank_SW1D.txt};
        \legend{$10^{-1}$,$10^{-3}$,$10^{-5}$,$10^{-7}$,$10^{-9}$};
        \node [text width=1em,anchor=north west] at (rel axis cs: 0.02,1.05) {\subcaption{\label{fig:singular_values_SW1D_b}}};
    \end{groupplot}
    \node[pin={[pin distance=3.25cm]368:{%
        \begin{tikzpicture}[baseline,trim axis right]
            \begin{axis}[
                    axis line style = thick,
                    grid=both,
                    minor tick num=2,
                    grid style = {gray,opacity=0.2},
                    every axis plot post/.append style={ultra thick},
                    ymode=log,
                    xmin=0,xmax=22,
                    ymin=0.0001,ymax=1,
                    width=4cm,
                    legend style={at={(1.4,1)},anchor=north},
                    legend cell align=left,
                    xlabel style={font=\footnotesize},
                    ylabel style={font=\footnotesize},
                    x tick label style={font=\footnotesize},
                    y tick label style={font=\footnotesize},
                    legend style={font=\tiny}
                ]
                \addplot[color=black] table[x=Index,y=Values] {figures/data/SW1D_singular_values/full_singular_values_SW1D.txt};
                \addplot[color=red] table[x=Index,y=Values] {figures/data/SW1D_singular_values/avg_singular_values_SW1D.txt};
            \end{axis}
        \end{tikzpicture}%
    }},draw,circle,minimum size=1cm] at (spypoint) {};
\end{tikzpicture}
\caption{SWE-1D: \bl{(a)} Singular values of the global snapshots matrix $\mathcal{S}$ and time average of the singular values of the local trajectories matrix $\mathcal{S}_{\tau}$. The singular values are normalized using the largest singular value for each case. \bl{(b)} $\epsilon$-rank of the local trajectories matrix $\mathcal{S}_{\tau}$ for different values of $\epsilon$.}
\label{fig:singular_values_SWE1D}
\end{figure}
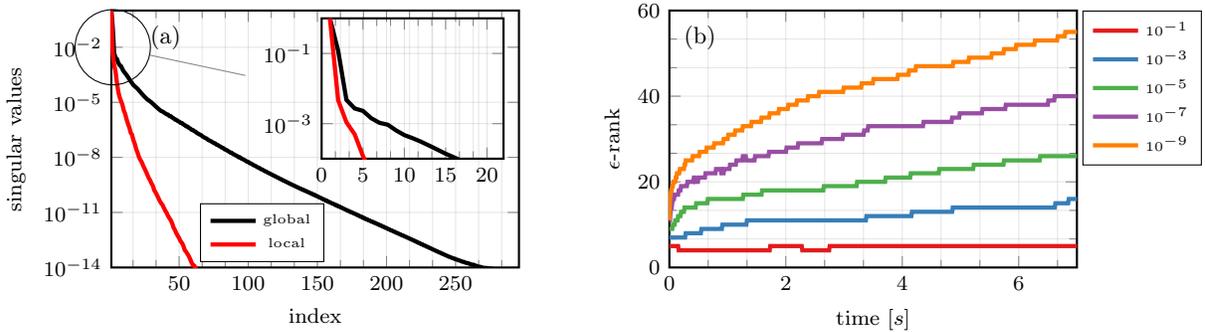
In Figure \ref{fig:singular_values_SW1D_a}, we compare the normalized singular values of $\Scal$ and $\Scal_{\tau}$, averaged over time for the latter. Although, in both cases, the exponential decay of the spectrum suggests the existence of reduced approximation spaces, the decay of the singular values of the averaged $\mathcal{S}_{\tau}$ is roughly $5$ times faster than that of $\mathcal{S}$. This difference suggests that a low-rank dynamical approach may be beneficial to reduce the computational cost and to increase the accuracy of the solution of the reduced model compared to a  method with a global basis. Furthermore, the evolution of the numerical rank of $\mathcal{S}_{\tau}$ over time, reported in Figure \ref{fig:singular_values_SW1D_b}, shows a rapid growth during the first steps, followed by a mild increase in the remaining part of the simulation. This is compatible with the observations, made in Section \ref{sec:criterion_rank_update}, about the behavior of the singular value spectrum for advection dominated problems.

In order to compare the performances of local and global model order reduction, we consider, as
global reduced method, the complex SVD approach \cite{peng2016symplectic}
with reduced dimension $\Nr\in\left \{ 10, 20, 30, 40, 60, 80 \right \}$. This is used to generate a symplectic reduced basis from the solution of the high-fidelity model \eqref{eq:SWE}
obtained every $10$ time steps and
by uniformly sampling $\Sprm$ with $4$ samples per dimension. The reduced system is solved using the implicit midpoint rule with the same time step $\Delta t$ used for the full order model. The quadratic operator, describing the evolution of \eqref{eq:SWE}, is reduced by using the approach described in Section \ref{sec:nonlinear} and the reduced operators are computed once during the offline stage.

Concerning the adaptive dynamical reduced model, we evaluate the initial condition \eqref{eq:init_cond_SWE1D} at all values $\prmh\in\Sprmh$ and compute the matrix $\mathcal{S}_1\in\mathbb{R}^{\Nf\times \Np}$ having as columns each of the evaluations. As initial condition for the reduced system \eqref{eq:UZred}, we use
\begin{equation}\label{eqn:initialization_local_low_rank}
    \begin{cases}
        U(0) = U_0,\\
        Z(0) = U_0^T\mathcal{S}_1,
    \end{cases}
\end{equation}
where $U_0\in\mathbb{R}^{\Nf\times \Nr_1}$ is obtained using the complex SVD applied to the snapshot matrix $\mathcal{S}_1$.
System \eqref{eq:UZred} is then evolved using the 2-stage partitioned Runge-Kutta method described in \eqref{lem:2_stage_PRK}.
For the following numerical experiments, we consider $\Nr_1\in\left \{ 6,8,10,12 \right \}$ as initial dimensions of the approximating reduced manifolds. As control parameters
for the rank update criterion of \Cref{algo:rank-update}, we fix the value $c=1.2$ and study examples with $r\in\left\{ 1.02, 1.05, 1.1, 1.2 \right \}$. Moreover, we examine the case in which the rank-updating algorithm is never triggered, i.e., the basis $U(t)$ evolves in time but its dimension is fixed ($\Nrht=\Nrh_1$ for all $\tau$). In the adaptive case, the error indicator $\mathbf{E}_{\tau}$ in \eqref{eqn:error_equation} is computed every $100$ iterations \bl{using a coarse mesh with $500$ equispaced intervals} on the subset \bl{$\widetilde{\eta}_h$} 
obtained by sampling $5$ parameters per dimension from $\Sprmh$.

In Figure \ref{fig:error_final_SWE1D}, we compare the global reduced model, the dynamical models for different values of $r$, and the high-fidelity model in terms of total runtime and accuracy at the final time $T$ by monitoring the error \eqref{eqn:error_metric}. The results show that, as we increase the dimension of the global reduced basis, the 
global reduced model provides accurate approximations but the runtime becomes larger than the one required to solve the high-fidelity problem. Hence, the global method loses the efficiency. 
The adaptive dynamical reduced approach outperforms the global reduced method by reaching comparable levels of accuracy at a computational time which is one order of magnitude smaller than the one required by the global reduction. Compared to the high-fidelity solver, the adaptive dynamical reduced method achieves an
accuracy of $E(T)=2.55\cdot10^{-5}$ with a speedup up of $42$, in the best-case scenario. 
For this numerical experiment, the effectiveness of the rank update algorithm is limited by the error introduced in the approximation of the initial condition via a reduced basis. While the error
is reduced from a factor of $4$ in the case of $\Nr_1=8$ to a factor of $20$ in the case of $\Nr_1=12$, compared to the non adaptive method, the accuracy is not significantly improved when $\Nr_1=6$. We note that, when the adaptive algorithm is effective, the additional computational cost associated with the evaluation of the error indicator and the evolution of a larger basis is balanced by a considerable error reduction.

\begin{figure}[H]
\centering
\begin{tikzpicture}[spy using outlines={rectangle, width=4.15cm, height=5cm, magnification=1.45, connect spies}]
    \begin{axis}[xlabel={runtime $\left[s\right]$},
                 ylabel={$E(T)$},
                 axis line style = thick,
                 grid=both,
                 minor tick num=2,
                 grid style = {gray,opacity=0.2},
                 xmode=log,
                 ymode=log,
                 ymax = 0.02, ymin = 0.000012,
                 xmax = 30000, xmin = 80, 
                 every axis plot/.append style={thick},
                 width = 9cm, height = 6cm,
                 legend style={at={(1.4,1)},anchor=north},
                 legend cell align=left,
                 ylabel near ticks, yticklabel pos=right,
                 xlabel style={font=\footnotesize},
                 ylabel style={font=\footnotesize},
                 x tick label style={font=\footnotesize},
                 y tick label style={font=\footnotesize},
                 legend style={font=\tiny}]
        \addplot+[mark=x,color=black,mark size=2pt] table[x=Timing,y=Error]
            {figures/data/SW1D_Pareto/error_final_local_reduced_model_no_indicator_SW1D.txt};
        \addplot+[mark=x,color=red,mark size=2pt,
            every node near coord/.append style={xshift=0.65cm},
            every node near coord/.append style={yshift=-0.2cm},
            nodes near coords, 
            point meta=explicit symbolic,
            every node near coord/.append style={font=\footnotesize}] table[x=Timing,y=Error, meta index=2]
            {figures/data/SW1D_Pareto/error_final_local_reduced_model_indicator_SW1D_1.02_1.2_2.txt};   
        \addplot+[mark=x,color=cyan,mark size=2pt] table[x=Timing,y=Error]
            {figures/data/SW1D_Pareto/error_final_local_reduced_model_indicator_SW1D_1.2_1.2_2.txt};
        \addplot+[mark=x,color=violet,mark size=2pt,
            every node near coord/.append style={xshift=0.65cm},
            every node near coord/.append style={yshift=-0.2cm},
            nodes near coords, 
            point meta=explicit symbolic,
            every node near coord/.append style={font=\footnotesize}] table[x=Timing,y=Error, meta index=2]
            {figures/data/SW1D_Pareto/error_final_global_reduced_model_SW1D.txt};
        \addplot+[color=black,dashed] table[x=Timing,y=DummyError]
            {figures/data/SW1D_Pareto/error_final_full_SW1D.txt}; 
        \coordinate (spypoint) at (axis cs:290,0.00025);
        \coordinate (magnifyglass) at (14,0.00015);
        \legend{{Non adaptive},
                {$r=1.02, \quad c=1.2$},
                {$r=1.2, \quad c=1.2$},
                {Global method},
                {Full model}};
    \end{axis}
    \spy [black] on (spypoint)
   in node[fill=white] at (magnifyglass);
\end{tikzpicture}
\caption{SWE-1D: Error \eqref{eqn:error_metric}, at time $T=7$, as a function of the runtime for the complex SVD method ({\color{violet}{\rule[.5ex]{1em}{1.2pt}}}),  the dynamical RB method ({\color{black}{\rule[.5ex]{1em}{1.2pt}}}) and the adaptive dynamical RB method for different values of the control parameters $r$ and $c$     ({\color{red}{\rule[.5ex]{1em}{1.2pt}}},{\color{cyan}{\rule[.5ex]{1em}{1.2pt}}}). 
For the sake of comparison, we report the runtime required by the high-fidelity solver ({\color{black}{\rule[.5ex]{0.4em}{1.2pt}}} {\color{black}{\rule[.5ex]{0.4em}{1.2pt}}}) to compute the numerical solutions for all values of the parameter $\prmh\in\Sprmh$.}
\label{fig:error_final_SWE1D}
\end{figure}
To better gauge the accuracy properties of the adaptive dynamical reduced basis method, we compare its error with the error given by the high-fidelity solver for the same initial condition.
The solution to the full model, with the projection of \eqref{eq:init_cond_SWE1D} onto the column space of $U_{0}$ as the initial condition,
is the target of the adaptive reduced procedure, which aims at correctly representing the high-fidelity solution space at every time step.
The importance of having a reduced space that accurately reproduces the initial condition can be inferred from
Figure \ref{fig:time_error_SWE1D_6}: the error associated with a poorly resolved initial condition dominates over the remaining sources of error, and adapting the dimension of the reduced basis is not beneficial in terms of accuracy. As noted above, increasing $\Nr_{1}$ not only improves the performance of the non adaptive reduced dynamical procedure but also boosts the potential gain, in terms of relative error reduction, of the adaptive method, as can be seen in Figure \ref{fig:time_error_SWE1D_10}.

In Figures \ref{fig:error_basis_time_SWE1D} we report the growth of the dimension of the reduced basis for different initial dimension $\Nr_1$. For the evolution of the error, we do not notice any significant difference as the parameter $r$ for the adaptive criterion \eqref{eq:ratio_tmp} varies.

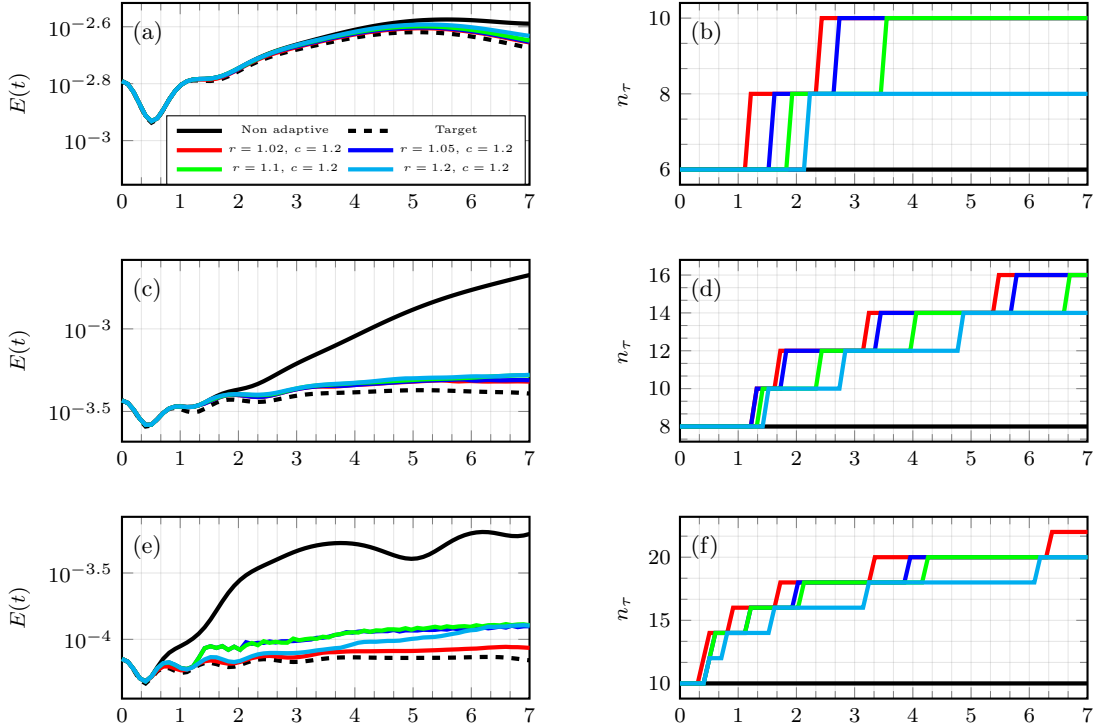
\begin{figure}[H]
\centering
\begin{tikzpicture}
    \begin{groupplot}[
      group style={group size=2 by 3,
                   horizontal sep=2cm},
      width=7cm, height=4cm
    ]
    \nextgroupplot[ylabel={$E(t)$},
                   axis line style = thick,
                   grid=both,
                   minor tick num=2,
                   max space between ticks=20,
                   grid style = {gray,opacity=0.2},
                   every axis plot/.append style={ultra thick},
                   xmin=0, xmax=7,
                   ymin=0.0007,
                   ymode=log,
                   xlabel style={font=\footnotesize},
                   ylabel style={font=\footnotesize},
                   x tick label style={font=\footnotesize},
                   y tick label style={font=\footnotesize},
                   legend style={font=\tiny},
                   legend style={at={(axis cs:0.76,0.0007)},anchor=south west},
                   legend columns = 2,
                   legend style={nodes={scale=0.83, transform shape}}]
        \addplot[color=black] table[x=Time,y=Error_6] {figures/data/SW1D_error_time_basis/error_local_reduced_model_no_indicator_SW1D.txt};
        \addplot+[color=black,dashed] table[x=Time,y=Error_6] {figures/data/SW1D_error_time_basis/error_full_model_corrupted_SW1D.txt};
        \addplot+[color=red] table[x=Timing,y=Error_1_02] {figures/data/SW1D_error_time_basis/error_local_reduced_model_indicator_SW1D_6.txt};
        \addplot+[color=blue] table[x=Timing,y=Error_1_05] {figures/data/SW1D_error_time_basis/error_local_reduced_model_indicator_SW1D_6.txt};
        \addplot+[color=green] table[x=Timing,y=Error_1_10] {figures/data/SW1D_error_time_basis/error_local_reduced_model_indicator_SW1D_6.txt};
        \addplot+[color=cyan] table[x=Timing,y=Error_1_20] {figures/data/SW1D_error_time_basis/error_local_reduced_model_indicator_SW1D_6.txt};
        \node [text width=1em,anchor=north west] at (rel axis cs: 0.01,1.05) {\subcaption{\label{fig:time_error_SWE1D_6}}};
        \legend{{Non adaptive},
                {Target},
                {$r=1.02, \, c=1.2$},
                {$r=1.05, \, c=1.2$},
                {$r=1.1, \, c=1.2$},
                {$r=1.2, \, c=1.2$}};
    \nextgroupplot[ylabel={$n_{\tau}$},
                   axis line style = thick,
                   grid=both,
                   minor tick num=2,
                   max space between ticks=20,
                   grid style = {gray,opacity=0.2},
                   every axis plot/.append style={ultra thick},
                   xmin=0, xmax=7,
                   xlabel style={font=\footnotesize},
                   ylabel style={font=\footnotesize},
                   x tick label style={font=\footnotesize},
                   y tick label style={font=\footnotesize},
                   legend style={font=\tiny}]
        \addplot[color=black] table[x=Time,y=Basis_6] {figures/data/SW1D_error_time_basis/basis_local_reduced_model_no_indicator_SW1D.txt};     \addplot+[color=red] table[x=Timing,y=Basis_1_02] {figures/data/SW1D_error_time_basis/basis_local_reduced_model_indicator_SW1D_6.txt};
        \addplot+[color=blue] table[x=Timing,y=Basis_1_05] {figures/data/SW1D_error_time_basis/basis_local_reduced_model_indicator_SW1D_6.txt}; 
        \addplot+[color=green] table[x=Timing,y=Basis_1_10] {figures/data/SW1D_error_time_basis/basis_local_reduced_model_indicator_SW1D_6.txt}; 
        \addplot+[color=cyan] table[x=Timing,y=Basis_1_20] {figures/data/SW1D_error_time_basis/basis_local_reduced_model_indicator_SW1D_6.txt};
        \node [text width=1em,anchor=north west] at (rel axis cs: 0.01,1.05) {\subcaption{\label{fig:time_basis_SWE1D_6}}};
    \nextgroupplot[ylabel={$E(t)$},
                   axis line style = thick,
                   grid=both,
                   minor tick num=2,
                   max space between ticks=20,
                   grid style = {gray,opacity=0.2},
                   every axis plot/.append style={ultra thick},
                   xmin=0, xmax=7,
                   ymode=log,
                   xlabel style={font=\footnotesize},
                   ylabel style={font=\footnotesize},
                   x tick label style={font=\footnotesize},
                   y tick label style={font=\footnotesize},
                   legend style={font=\tiny}]
        \addplot[color=black] table[x=Time,y=Error_8] {figures/data/SW1D_error_time_basis/error_local_reduced_model_no_indicator_SW1D.txt};
        \addplot+[color=black,dashed] table[x=Time,y=Error_8] {figures/data/SW1D_error_time_basis/error_full_model_corrupted_SW1D.txt};
        \addplot+[color=red] table[x=Timing,y=Error_1_02] {figures/data/SW1D_error_time_basis/error_local_reduced_model_indicator_SW1D_8.txt};
        \addplot+[color=blue] table[x=Timing,y=Error_1_05] {figures/data/SW1D_error_time_basis/error_local_reduced_model_indicator_SW1D_8.txt};
        \addplot+[color=green] table[x=Timing,y=Error_1_10] {figures/data/SW1D_error_time_basis/error_local_reduced_model_indicator_SW1D_8.txt};
        \addplot+[color=cyan] table[x=Timing,y=Error_1_20] {figures/data/SW1D_error_time_basis/error_local_reduced_model_indicator_SW1D_8.txt};
        \node [text width=1em,anchor=north west] at (rel axis cs: 0.01,1.05) {\subcaption{\label{fig:time_error_SWE1D_8}}};
    \nextgroupplot[ylabel={$n_{\tau}$},
                   axis line style = thick,
                   grid=both,
                   minor tick num=2,
                   max space between ticks=20,
                   grid style = {gray,opacity=0.2},
                   every axis plot/.append style={ultra thick},
                   xmin=0, xmax=7,
                   xlabel style={font=\footnotesize},
                   ylabel style={font=\footnotesize},
                   x tick label style={font=\footnotesize},
                   y tick label style={font=\footnotesize},
                   legend style={font=\tiny}]
        \addplot[color=black] table[x=Time,y=Basis_8] {figures/data/SW1D_error_time_basis/basis_local_reduced_model_no_indicator_SW1D.txt};     \addplot+[color=red] table[x=Timing,y=Basis_1_02] {figures/data/SW1D_error_time_basis/basis_local_reduced_model_indicator_SW1D_8.txt};
        \addplot+[color=blue] table[x=Timing,y=Basis_1_05] {figures/data/SW1D_error_time_basis/basis_local_reduced_model_indicator_SW1D_8.txt}; 
        \addplot+[color=green] table[x=Timing,y=Basis_1_10] {figures/data/SW1D_error_time_basis/basis_local_reduced_model_indicator_SW1D_8.txt}; 
        \addplot+[color=cyan] table[x=Timing,y=Basis_1_20] {figures/data/SW1D_error_time_basis/basis_local_reduced_model_indicator_SW1D_8.txt};
        \node [text width=1em,anchor=north west] at (rel axis cs: 0.01,1.05) {\subcaption{\label{fig:time_basis_SWE1D_8}}};
    \nextgroupplot[xlabel={\bl{time $\left [ s \right ]$}},
                   ylabel={$E(t)$},
                   axis line style = thick,
                   grid=both,
                   minor tick num=2,
                   max space between ticks=20,
                   grid style = {gray,opacity=0.2},
                   every axis plot/.append style={ultra thick},
                   xmin=0, xmax=7,
                   ymode=log,
                   xlabel style={font=\footnotesize},
                   ylabel style={font=\footnotesize},
                   x tick label style={font=\footnotesize},
                   y tick label style={font=\footnotesize},
                   legend style={font=\tiny}]
        \addplot[color=black] table[x=Time,y=Error_10] {figures/data/SW1D_error_time_basis/error_local_reduced_model_no_indicator_SW1D.txt};
        \addplot+[color=black,dashed] table[x=Time,y=Error_10] {figures/data/SW1D_error_time_basis/error_full_model_corrupted_SW1D.txt};
        \addplot+[color=red] table[x=Timing,y=Error_1_02] {figures/data/SW1D_error_time_basis/error_local_reduced_model_indicator_SW1D_10.txt};
        \addplot+[color=blue] table[x=Timing,y=Error_1_05] {figures/data/SW1D_error_time_basis/error_local_reduced_model_indicator_SW1D_10.txt};
        \addplot+[color=green] table[x=Timing,y=Error_1_10] {figures/data/SW1D_error_time_basis/error_local_reduced_model_indicator_SW1D_10.txt};
        \addplot+[color=cyan] table[x=Timing,y=Error_1_20] {figures/data/SW1D_error_time_basis/error_local_reduced_model_indicator_SW1D_10.txt};
        \node [text width=1em,anchor=north west] at (rel axis cs: 0.01,1.05) {\subcaption{\label{fig:time_error_SWE1D_10}}};
    \nextgroupplot[xlabel={\bl{time $\left [ s \right ]$}},
                   ylabel={$n_{\tau}$},
                   axis line style = thick,
                   grid=both,
                   minor tick num=2,
                   max space between ticks=20,
                   grid style = {gray,opacity=0.2},
                   every axis plot/.append style={ultra thick},
                   xmin=0, xmax=7,
                   xlabel style={font=\footnotesize},
                   ylabel style={font=\footnotesize},
                   x tick label style={font=\footnotesize},
                   y tick label style={font=\footnotesize},
                   legend style={font=\tiny}]
        \addplot[color=black] table[x=Time,y=Basis_10] {figures/data/SW1D_error_time_basis/basis_local_reduced_model_no_indicator_SW1D.txt};  
        \addplot+[color=red] table[x=Timing,y=Basis_1_02] {figures/data/SW1D_error_time_basis/basis_local_reduced_model_indicator_SW1D_10.txt};
        \addplot+[color=blue] table[x=Timing,y=Basis_1_05] {figures/data/SW1D_error_time_basis/basis_local_reduced_model_indicator_SW1D_10.txt}; 
        \addplot+[color=green] table[x=Timing,y=Basis_1_10] {figures/data/SW1D_error_time_basis/basis_local_reduced_model_indicator_SW1D_10.txt}; 
        \addplot+[color=cyan] table[x=Timing,y=Basis_1_20] {figures/data/SW1D_error_time_basis/basis_local_reduced_model_indicator_SW1D_10.txt}; 
        \node [text width=1em,anchor=north west] at (rel axis cs: 0.01,1.05) {\subcaption{\label{fig:time_basis_SWE1D_10}}};
     \end{groupplot}
\end{tikzpicture}
\caption{SWE-1D: On the left column, we report the evolution of the error $E(t)$ \eqref{eqn:error_metric} for the adaptive and non adaptive dynamical RB methods for different values of the control parameter $r$ and different dimensions $\Nr_1$ of the approximating manifold of the initial condition. The target error is obtained by solving the
full model
with initial condition obtained by projecting \eqref{eq:init_cond_SWE1D} onto a symplectic manifold of dimension $\Nr_1$. On the right column, we report the evolution of the dimension of the dynamical reduced basis over time. The adaptive algorithm is driven by the error indicator \eqref{eq:ratio_tmp}, while in the non adaptive setting, the dimension does not change with time.
We consider the cases
$\Nr_1=6$ (Figs.~\bl{(a)-(b)}),
$\Nr_1=8$ (Figs.~\bl{(c)-(d)}),
$\Nr_1=10$ (Figs.~\bl{(e)-(f)}).}
\label{fig:error_basis_time_SWE1D}
\end{figure}
Ideally, within each temporal interval, the reduced solution is close, in the Frobenius norm,
to the best rank-$\Nrt$ approximation of the full model solution.
To verify this property for the adaptive dynamical reduced basis
method,
we monitor the evolution of the error $E_{\perp}$ between the full model solution $\Rcal$, at the current time and for all $\prmh\in\Sprmh$, and its projection onto the space spanned by the current reduced basis evolved following \eqref{eq:U}, namely
$E_{\perp}(t)=\norm{\Rcal(t) - U(t) U(t)^{\top}\Rcal(t)}$. 

In Figure \ref{fig:SW1D_projection_error},
the projection error is shown for different values of $\Nr_1$ (Figures \ref{fig:projection_error_SWE1D_8} and \ref{fig:projection_error_SWE1D_10}) and the corresponding evolution of the reduced basis dimension is reported (Figures \ref{fig:basis_evolution_SWE1D_8} and \ref{fig:basis_evolution_SWE1D_10}).
We notice that, when the dimension of the basis $U$ is not adapted, the projection error tends to increase in time. This can be ascribed to the fact that the effective rank of the high-fidelity solution is growing and the reduced basis is no longer large enough to capture the rank-increasing solution.
Adapting $\Nr_{\tau}$ during the simulation results in a zero-growth scenario, with local negative peaks when the basis is enlarged.
This indicates that the strategy of enlarging the reduced manifold in the direction of the larger error (see \Cref{sec:rank-update}) yields a considerable improvement of the approximation.
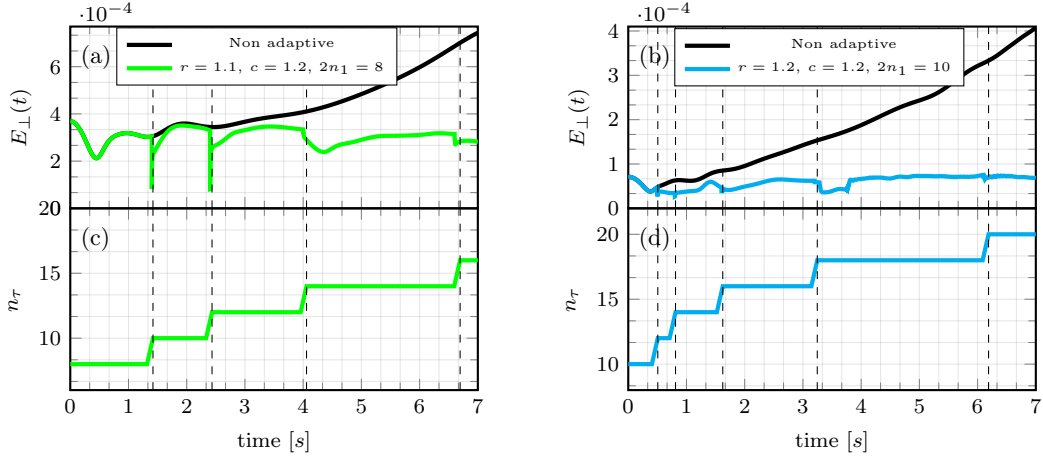
\begin{figure}[H]
\centering
\begin{tikzpicture}
    \begin{groupplot}[
      group style={group size=2 by 2,
                   horizontal sep=2cm},
      width=7cm, height=4cm
    ]
    \nextgroupplot[ylabel={$E_{\perp}(t)$},
                   axis line style = thick,
                   grid=both,
                   minor tick num=2,
                   max space between ticks=20,
                   grid style = {gray,opacity=0.2},
                   every axis plot/.append style={ultra thick},
                   xmin=0, xmax=7,
                   enlargelimits=false,
                   xticklabel=\empty,
                   legend style={at={(axis cs:0.79,0.00052)},anchor=south west},
                   xlabel style={font=\footnotesize},
                   ylabel style={font=\footnotesize},
                   x tick label style={font=\footnotesize},
                   y tick label style={font=\footnotesize},
                   legend style={font=\tiny}]
        \addplot[color=black] table[x=Time,y=Error] {figures/data/SW1D_projection_error/projection_error_8_9999999_SW1D.txt};
        \addplot+[color=green] table[x=Time,y=Error] {figures/data/SW1D_projection_error/projection_error_8_1.1_SW1D.txt};
        \addplot+[thin,color=black,ycomb,dashed] table[x=Time, y expr=0.000385*\thisrow{Spikes}] {figures/data/SW1D_projection_error/nbasis_8_1.1_SW1D.txt};
        \node [text width=1em,anchor=north west] at (rel axis cs: 0.01,1.05) {\subcaption{\label{fig:projection_error_SWE1D_8}}};
        \legend{{Non adaptive},
                {$r=1.1, \, c=1.2, \, \Nr_1=8$}};
    \nextgroupplot[ylabel={$E_{\perp}(t)$},
                   axis line style = thick,
                   grid=both,
                   minor tick num=2,
                   max space between ticks=20,
                   grid style = {gray,opacity=0.2},
                   every axis plot/.append style={ultra thick},
                   xmin=0, xmax=7,
                   enlargelimits=false,
                   xticklabel=\empty,
                   legend style={at={(axis cs:0.79,0.000275)},anchor=south west},
                   xlabel style={font=\footnotesize},
                   ylabel style={font=\footnotesize},
                   x tick label style={font=\footnotesize},
                   y tick label style={font=\footnotesize},
                   legend style={font=\tiny}]
        \addplot[color=black] table[x=Time,y=Error] {figures/data/SW1D_projection_error/projection_error_10_9999999_SW1D.txt};
        \addplot+[color=cyan] table[x=Time,y=Error] {figures/data/SW1D_projection_error/projection_error_10_1.2_SW1D.txt};
        \addplot+[thin,color=black,ycomb,dashed] table[x=Time, y expr=0.000205*\thisrow{Spikes}] {figures/data/SW1D_projection_error/nbasis_10_1.2_SW1D.txt};
        \node [text width=1em,anchor=north west] at (rel axis cs: 0.01,1.05) {\subcaption{\label{fig:projection_error_SWE1D_10}}};
        \legend{{Non adaptive},
                {$r=1.2, \, c=1.2, \, \Nr_1=10$}};
    \nextgroupplot[xlabel={time $\left [ s \right ]$},
                   ylabel={$n_{\tau}$},
                   axis line style = thick,
                   grid=both,
                   minor tick num=2,
                   max space between ticks=20,
                   grid style = {gray,opacity=0.2},
                   every axis plot/.append style={ultra thick},
                   xmin=0, xmax=7,
                   ymin=6, ymax=20,
                   xlabel style={font=\footnotesize},
                   ylabel style={font=\footnotesize},
                   x tick label style={font=\footnotesize},
                   y tick label style={font=\footnotesize},
                   legend style={font=\tiny},
                   yshift=1cm]
        \addplot[color=green] table[x=Time,y=Basis] {figures/data/SW1D_projection_error/nbasis_8_1.1_SW1D.txt};
        \addplot+[thin,color=black,ycomb,dashed] table[x=Time, y expr=12*\thisrow{Spikes}] {figures/data/SW1D_projection_error/nbasis_8_1.1_SW1D.txt};
        \node [text width=1em,anchor=north west] at (rel axis cs: 0.01,1.05) {\subcaption{\label{fig:basis_evolution_SWE1D_8}}};
    \nextgroupplot[xlabel={time $\left [ s \right ]$},
                   ylabel={$n_{\tau}$},
                   axis line style = thick,
                   grid=both,
                   minor tick num=2,
                   max space between ticks=20,
                   grid style = {gray,opacity=0.2},
                   every axis plot/.append style={ultra thick},
                   xmin=0, xmax=7,
                   ymin=8, ymax=22,
                   xlabel style={font=\footnotesize},
                   ylabel style={font=\footnotesize},
                   x tick label style={font=\footnotesize},
                   y tick label style={font=\footnotesize},
                   legend style={font=\tiny},
                   yshift=1cm]
        \addplot[color=cyan] table[x=Time,y=Basis] {figures/data/SW1D_projection_error/nbasis_10_1.2_SW1D.txt};
        \addplot+[thin,color=black,ycomb,dashed] table[x=Time, y expr=12*\thisrow{Spikes}] {figures/data/SW1D_projection_error/nbasis_10_1.2_SW1D.txt};
        \node [text width=1em,anchor=north west] at (rel axis cs: 0.01,1.05) {\subcaption{\label{fig:basis_evolution_SWE1D_10}}};
    \end{groupplot}
\end{tikzpicture}
\caption{SWE-1D: In Figs. \bl{(a)} and \bl{(b)}, we report the evolution of the projection error $E_{\perp}(t)$ for different values of the initial dimension $\Nr_1$ of the reduced manifold.
In Figs.\bl{(c)} and \bl{(d)}, we report the corresponding evolution of the dimension of the reduced manifolds.}
\label{fig:SW1D_projection_error}
\end{figure}    

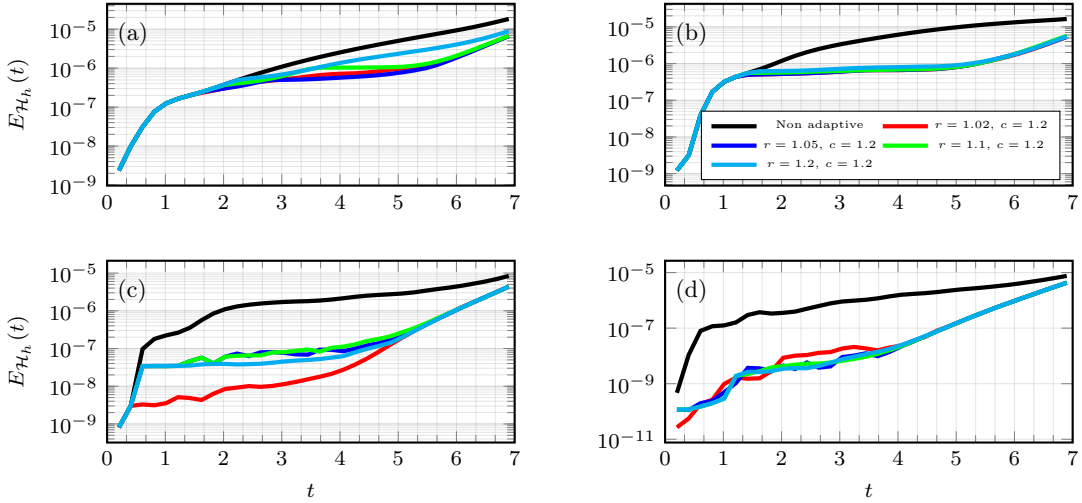
\begin{figure}[H]
\centering
\begin{tikzpicture}
    \begin{groupplot}[
      group style={group size=2 by 4,
                   horizontal sep=2cm},
      width=7cm, height=4cm
    ]
    \nextgroupplot[ylabel={$E_{\mathcal{H}_h}(t)$},
                   axis line style = thick,
                   grid=both,
                   minor tick num=2,
                   max space between ticks=20,
                   grid style = {gray,opacity=0.2},
                   every axis plot/.append style={ultra thick},
                   xmin=0, xmax=7,
                   ymode=log,
                   xlabel style={font=\footnotesize},
                   ylabel style={font=\footnotesize},
                   x tick label style={font=\footnotesize},
                   y tick label style={font=\footnotesize},
                   legend style={font=\tiny},
                   legend style={at={(axis cs:0.76,0.0007)},anchor=south west},
                   legend columns = 2,
                   legend style={nodes={scale=0.83, transform shape}}]
        \addplot+[color=black] table[x=Timing,y=Error] {figures/data/SW1D_Hamiltonian/error_Hamiltonian_local_reduced_model_no_indicator_SW1D_6.txt};
        \addplot+[color=red] table[x=Timing,y=Error_1_02] {figures/data/SW1D_Hamiltonian/error_Hamiltonian_local_reduced_model_indicator_SW1D_6.txt};
        \addplot+[color=blue] table[x=Timing,y=Error_1_05] {figures/data/SW1D_Hamiltonian/error_Hamiltonian_local_reduced_model_indicator_SW1D_6.txt};
        \addplot+[color=green] table[x=Timing,y=Error_1_10] {figures/data/SW1D_Hamiltonian/error_Hamiltonian_local_reduced_model_indicator_SW1D_6.txt};
        \addplot+[color=cyan] table[x=Timing,y=Error_1_20] {figures/data/SW1D_Hamiltonian/error_Hamiltonian_local_reduced_model_indicator_SW1D_6.txt};
        \node [text width=1em,anchor=north west] at (rel axis cs: 0.01,1.05) {\subcaption{\label{fig:error_Hamiltonian_SW1D_6}}};
    \nextgroupplot[axis line style = thick,
                   grid=both,
                   minor tick num=2,
                   max space between ticks=20,
                   grid style = {gray,opacity=0.2},
                   every axis plot/.append style={ultra thick},
                   xmin=0, xmax=7,
                   ymode=log,
                   xlabel style={font=\footnotesize},
                   ylabel style={font=\footnotesize},
                   x tick label style={font=\footnotesize},
                   y tick label style={font=\footnotesize},
                   legend style={font=\tiny},
                   legend pos = south east,
                   legend columns = 2,
                    legend style={nodes={scale=0.83, transform shape}}]
        \addplot+[color=black] table[x=Timing,y=Error] {figures/data/SW1D_Hamiltonian/error_Hamiltonian_local_reduced_model_no_indicator_SW1D_8.txt};
        \addplot+[color=red] table[x=Timing,y=Error_1_02] {figures/data/SW1D_Hamiltonian/error_Hamiltonian_local_reduced_model_indicator_SW1D_8.txt};
        \addplot+[color=blue] table[x=Timing,y=Error_1_05] {figures/data/SW1D_Hamiltonian/error_Hamiltonian_local_reduced_model_indicator_SW1D_8.txt};
        \addplot+[color=green] table[x=Timing,y=Error_1_10] {figures/data/SW1D_Hamiltonian/error_Hamiltonian_local_reduced_model_indicator_SW1D_8.txt};
        \addplot+[color=cyan] table[x=Timing,y=Error_1_20] {figures/data/SW1D_Hamiltonian/error_Hamiltonian_local_reduced_model_indicator_SW1D_8.txt};
        \node [text width=1em,anchor=north west] at (rel axis cs: 0.01,1.05) {\subcaption{\label{fig:error_Hamiltonian_SW1D_8}}};
        \legend{{Non adaptive},
                {$r=1.02, \, c=1.2$},
                {$r=1.05, \, c=1.2$},
                {$r=1.1, \, c=1.2$},
                {$r=1.2, \, c=1.2$}};
    \nextgroupplot[ylabel={$E_{\mathcal{H}_h}(t)$},
                   xlabel={\bl{time $\left [ s \right ]$}},
                   axis line style = thick,
                   grid=both,
                   minor tick num=2,
                   max space between ticks=20,
                   grid style = {gray,opacity=0.2},
                   every axis plot/.append style={ultra thick},
                   xmin=0, xmax=7,
                   ymode=log,
                   xlabel style={font=\footnotesize},
                   ylabel style={font=\footnotesize},
                   x tick label style={font=\footnotesize},
                   y tick label style={font=\footnotesize},
                   legend style={font=\tiny}]
        \addplot+[color=black] table[x=Timing,y=Error] {figures/data/SW1D_Hamiltonian/error_Hamiltonian_local_reduced_model_no_indicator_SW1D_10.txt};
        \addplot+[color=red] table[x=Timing,y=Error_1_02] {figures/data/SW1D_Hamiltonian/error_Hamiltonian_local_reduced_model_indicator_SW1D_10.txt};
        \addplot+[color=blue] table[x=Timing,y=Error_1_05] {figures/data/SW1D_Hamiltonian/error_Hamiltonian_local_reduced_model_indicator_SW1D_10.txt};
        \addplot+[color=green] table[x=Timing,y=Error_1_10] {figures/data/SW1D_Hamiltonian/error_Hamiltonian_local_reduced_model_indicator_SW1D_10.txt};
        \addplot+[color=cyan] table[x=Timing,y=Error_1_20] {figures/data/SW1D_Hamiltonian/error_Hamiltonian_local_reduced_model_indicator_SW1D_10.txt};
        \node [text width=1em,anchor=north west] at (rel axis cs: 0.01,1.05) {\subcaption{\label{fig:error_Hamiltonian_SW1D_10}}};
    \nextgroupplot[xlabel={\bl{time $\left [ s \right ]$}},
                   axis line style = thick,
                   grid=both,
                   minor tick num=2,
                   max space between ticks=20,
                   grid style = {gray,opacity=0.2},
                   every axis plot/.append style={ultra thick},
                   xmin=0, xmax=7,
                   ymode=log,
                   xlabel style={font=\footnotesize},
                   ylabel style={font=\footnotesize},
                   x tick label style={font=\footnotesize},
                   y tick label style={font=\footnotesize},
                   legend style={font=\tiny}]
        \addplot+[color=black] table[x=Timing,y=Error] {figures/data/SW1D_Hamiltonian/error_Hamiltonian_local_reduced_model_no_indicator_SW1D_12.txt};
        \addplot+[color=red] table[x=Timing,y=Error_1_02] {figures/data/SW1D_Hamiltonian/error_Hamiltonian_local_reduced_model_indicator_SW1D_12.txt};
        \addplot+[color=blue] table[x=Timing,y=Error_1_05] {figures/data/SW1D_Hamiltonian/error_Hamiltonian_local_reduced_model_indicator_SW1D_12.txt};
        \addplot+[color=green] table[x=Timing,y=Error_1_10] {figures/data/SW1D_Hamiltonian/error_Hamiltonian_local_reduced_model_indicator_SW1D_12.txt};
        \addplot+[color=cyan] table[x=Timing,y=Error_1_20] {figures/data/SW1D_Hamiltonian/error_Hamiltonian_local_reduced_model_indicator_SW1D_12.txt};
        \node [text width=1em,anchor=north west] at (rel axis cs: 0.01,1.05) {\subcaption{\label{fig:error_Hamiltonian_SW1D_12}}};
    \end{groupplot}
\end{tikzpicture}
\caption{SWE-1D: Relative error \eqref{eqn:relative_error_Hamiltonian} in the conservation of the discrete Hamiltonian \eqref{eq:SW-1D_Ham}
for the dynamical reduced basis method with initial reduced
dimensions
$\Nr_1=6$ \bl{(a)},
$\Nr_1=8$ \bl{(b)},
$\Nr_1=10$ \bl{(c)} and
$\Nr_1=12$ \bl{(d)}.}
\label{fig:error_Hamiltonian_SWE1D}
\end{figure}
In Figure \ref{fig:error_Hamiltonian_SWE1D} we show the relative error in the conservation of the Hamiltonian for different dimensions of the reduced manifold,
and values of the control parameters $r$ and $c$.
As the Hamiltonian \eqref{eq:SW-1D_Ham} is a cubic quantity, we do not expect exact conservation associated with the proposed partitioned Runge--Kutta temporal integrators. However, the preservation of the symplectic structure both in the reduction and in the discretization yields a good control on the Hamiltonian error, as it can be observed in Figure \ref{fig:error_Hamiltonian_SWE1D}.

\subsubsection{Two-dimensional shallow water equations (SWE-2D)}
We set $\Omega=[-4,4]^2$ as the spatial domain and $\Sprm=\left [ \frac{1}{5}, \frac{1}{2}\right ]\times\left[ \frac{11}{10},\frac{17}{10} \right]$ as the domain of parameters. We consider $10$ uniformly spaced values of the parameter for each dimension of $\Sprm$ to define the discrete subset $\Sprmh$. As initial condition, we consider
\begin{equation}\label{eq:init_cond_SWE2D}
    \begin{cases}
    h^{0}(x,y;\prmh) = 1+\alpha e^{-\beta (x^2+y^2)},\\
    \phi^{0}(x,y;\prmh) = 0,
    \end{cases}
\end{equation}
where $\prmh=(\alpha,\beta)$ represents the natural extension to the two-dimensional setting of the parameter used in the previous example.
The domain $\Omega$ is partitioned using \bl{$M=51$} points per dimension, so that the resulting mesh width is \bl{$\Delta x= \Delta y = 16\cdot 10^{-2}$}.
The time domain $\mathcal{T}=\left [ 0, T:=20 \right ]$ is split into $N_{\tau}=10000$ uniform intervals of length $\Delta t = 2\cdot 10^{-3}$. The symplectic implicit midpoint is employed as time integrator in the high-fidelity solver, while
the reduced dynamics \eqref{eq:UZred} is integrated using the 2-stage partitioned RK method.
The spatial and temporal domains considered for this numerical experiment are taken so that the solution of the high-fidelity model is characterized by circular waves that interact and overlap because of the periodic boundary conditions, as shown in Figure \ref{fig:surf_comparison_SW2D}.

The increased complexity of the two-dimensional dynamics is reflected in the behaviour of the spectrum of the matrix snapshots. In Figure \ref{fig:singular_values_SW2D_a}, we show the normalized singular values of the global snapshot matrix $\mathcal{S}\in\mathbb{R}^{\Nf\times(N_{\tau} \Np)}$ and the average of the $N_{\tau}$ local-in-time snapshot matrices $\mathcal{S}_{\tau}\in\mathbb{R}^{\Nf\times \Np}$. The decay of the singular values of the local trajectories is one order of magnitude faster than of the global (in time) snapshots, suggesting that there exists an underlying \emph{local} low-rank structure that can be exploited to improve the efficiency of the reduced model. The evolution of the numerical rank of $\mathcal{S}_{\tau}$, reported in Figure \ref{fig:singular_values_SW2D_b}, indicates that, while the matrix-valued initial condition is exactly represented using an extremely small basis, the full model solution at times $t\geq 2$ requires a relatively large basis to be properly approximated, and hence adapting the dimension of the reduced manifold becomes crucial.
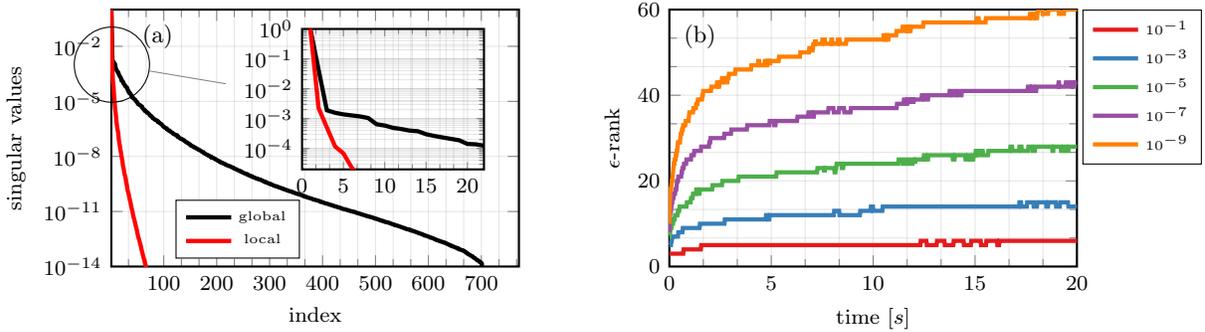
\begin{figure}[H]
\centering
\begin{tikzpicture}
    \begin{groupplot}[
      group style={group size=2 by 1,
                   horizontal sep=2cm},
      width=7cm, height=5cm
    ]
    \nextgroupplot[xlabel={index},
                   ylabel={singular values},
                   axis line style = thick,
                   grid=both,
                   minor tick num=2,
                   max space between ticks=20,
                   grid style = {gray,opacity=0.2},
                   ymode=log,
                   every axis plot/.append style={ultra thick},
                   xmin = 1,
                   ymin = 1e-14, ymax = 1,
                   ylabel near ticks, 
                   legend style={at={(0.31,0.26)},anchor=north},
                   xlabel style={font=\footnotesize},
                   ylabel style={font=\footnotesize},
                   x tick label style={font=\footnotesize},
                   y tick label style={font=\footnotesize},
                   legend style={font=\tiny}
                   ]
        \addplot[color=black] table[x=Index,y=Values] {figures/data/SW2D_singular_values/full_singular_values_SW2D.txt};
        \addplot[color=red] table[x=Index,y=Values] {figures/data/SW2D_singular_values/avg_singular_values_SW2D.txt};
        \legend{global,local};
        \node [text width=1em,anchor=north west] at (rel axis cs: 0.06,1.05) {\subcaption{\label{fig:singular_values_SW2D_a}}};
        \coordinate (spypoint) at (axis cs:2,0.001);
    \nextgroupplot[xlabel={time $\left [ s \right ]$},
                   ylabel={$\epsilon$-rank},
                   axis line style = thick,
                   grid=both,
                   minor tick num=2,
                   grid style = {gray,opacity=0.2},
                   every axis plot/.append style={ultra thick},
                   xmin = 0, xmax = 20,
                   ymin = 0, ymax = 60,
                   legend style={at={(1.16,1)},anchor=north},
                   xlabel style={font=\footnotesize},
                   ylabel style={font=\footnotesize},
                   x tick label style={font=\footnotesize},
                   y tick label style={font=\footnotesize},
                   legend style={font=\tiny}]
        \addplot+[] table[x=Time,y=Rank] {figures/data/SW2D_epsilon_rank/0.1_epsilon_rank_SW2D.txt};
        \addplot+[] table[x=Time,y=Rank] {figures/data/SW2D_epsilon_rank/0.001_epsilon_rank_SW2D.txt};
        \addplot+[] table[x=Time,y=Rank] {figures/data/SW2D_epsilon_rank/1e-05_epsilon_rank_SW2D.txt};
        \addplot+[] table[x=Time,y=Rank] {figures/data/SW2D_epsilon_rank/1e-07_epsilon_rank_SW2D.txt};
        \addplot+[] table[x=Time,y=Rank] {figures/data/SW2D_epsilon_rank/1e-09_epsilon_rank_SW2D.txt};
        \legend{$10^{-1}$,$10^{-3}$,$10^{-5}$,$10^{-7}$,$10^{-9}$};
        \node [text width=1em,anchor=north west] at (rel axis cs: 0.02,1.05) {\subcaption{\label{fig:singular_values_SW2D_b}}};
    \end{groupplot}
    \node[pin={[pin distance=3.2cm]373:{%
        \begin{tikzpicture}[baseline,trim axis right]
            \begin{axis}[
                    axis line style = thick,
                    grid=both,
                    minor tick num=2,
                    grid style = {gray,opacity=0.2},
                    every axis plot post/.append style={ultra thick},
                    tiny,
                    ymode=log,
                    xmin=0,xmax=22,
                    ymin=0.00002,ymax=1,
                    width=4cm,
                    legend style={at={(1.4,1)},anchor=north},
                    legend cell align=left,
                    xlabel style={font=\footnotesize},
                    ylabel style={font=\footnotesize},
                    x tick label style={font=\footnotesize},
                    y tick label style={font=\footnotesize},
                    legend style={font=\tiny}
                ]
                \addplot[color=black] table[x=Index,y=Values] {figures/data/SW2D_singular_values/full_singular_values_SW2D.txt};
                \addplot[color=red] table[x=Index,y=Values] {figures/data/SW2D_singular_values/avg_singular_values_SW2D.txt};
            \end{axis}
        \end{tikzpicture}%
    }},draw,circle,minimum size=1cm] at (spypoint) {};
\end{tikzpicture}
\caption{SWE-2D: \bl{(a)} Singular values of the global snapshots matrix $\mathcal{S}$ and time average of the singular values of the local trajectories matrix $\mathcal{S}_{\tau}$. The singular values are normalized using the largest singular value for each case. \bl{(b)} $\epsilon$-rank of the local trajectories matrix $\mathcal{S}_{\tau}$ for different values of $\epsilon$.
}
\end{figure}
We employ the complex SVD method to build a global reduced order model, using the same sampling rates in time and parameter space as in the 1D test case. With none of the dimensions considered, i.e., $2n\in\{10,20,40,60,80,120\}$, we obtain results that are both accurate (error smaller than $10^{-1}$) and computationally less expensive
than solving the high-fidelity model.
Hence, for this two-dimensional test,
we only compare the performances of the adaptive and the non-adaptive dynamical reduced basis method in terms of accuracy and computational time. 
As initial condition for the reduced dynamics \eqref{eq:UZred} we consider
the initialization \eqref{eqn:initialization_local_low_rank} where $\mathcal{S}_1$ is given by \eqref{eq:init_cond_SWE2D}.
Moreover, for the adaptive method, we compute the error indicator
every $10$ iterations and on a subset \bl{$\widetilde{\eta}_h$}  of $25$ uniformly sampled parameters. Different combinations of the initial reduced manifold dimension $\Nr_{1}=\{4,6,8\}$, and control parameters $r=\{1.1,1.2,1.3\}$ and $c=\{1.1,1.2,1.3\}$, are considered to study their impact on the accuracy of the method.

Figure \ref{fig:surf_comparison_SW2D} shows the high-fidelity solution for $(\alpha,\beta)=(\frac{1}{3},\frac{17}{10})$ with its adaptive reduced approximation at different times. The results are qualitatively equivalent.
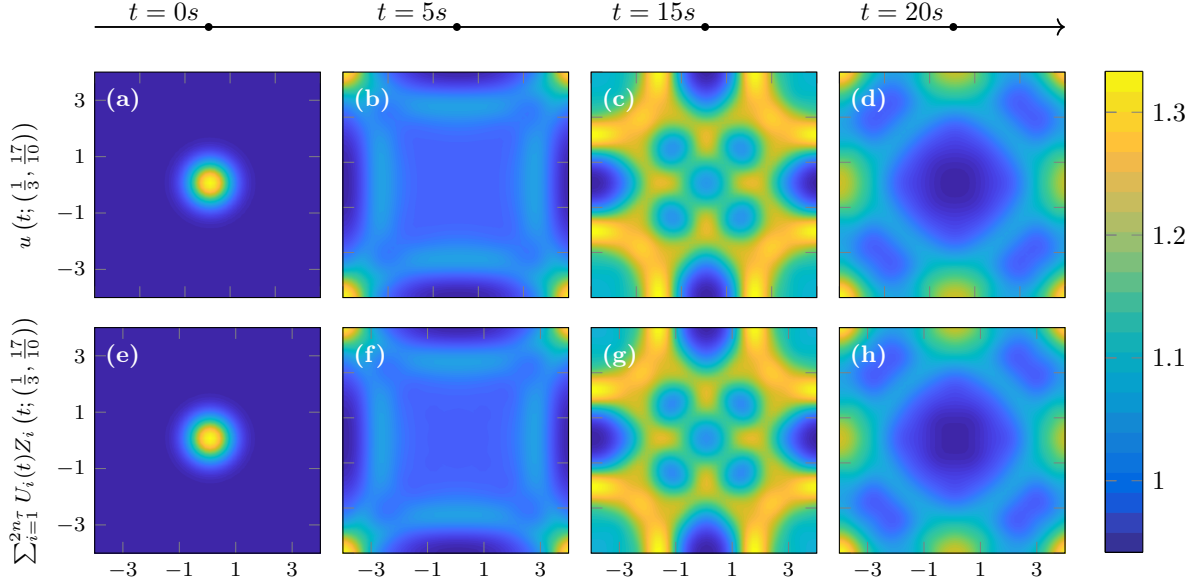
\begin{figure}[H]
\centering
\begin{tikzpicture}

    \pgfplotsset{
        colormap={parula}{
            rgb=(0.208100000000000,0.166300000000000,0.529200000000000)
            rgb=(0.211623809523810,0.189780952380952,0.577676190476191)
            rgb=(0.212252380952381,0.213771428571429,0.626971428571429)
            rgb=(0.208100000000000,0.238600000000000,0.677085714285714)
            rgb=(0.195904761904762,0.264457142857143,0.727900000000000)
            rgb=(0.170728571428571,0.291938095238095,0.779247619047619)
            rgb=(0.125271428571429,0.324242857142857,0.830271428571429)
            rgb=(0.0591333333333334,0.359833333333333,0.868333333333333)
            rgb=(0.0116952380952381,0.387509523809524,0.881957142857143)
            rgb=(0.00595714285714286,0.408614285714286,0.882842857142857)
            rgb=(0.0165142857142857,0.426600000000000,0.878633333333333)
            rgb=(0.0328523809523810,0.443042857142857,0.871957142857143)
            rgb=(0.0498142857142857,0.458571428571429,0.864057142857143)
            rgb=(0.0629333333333333,0.473690476190476,0.855438095238095)
            rgb=(0.0722666666666667,0.488666666666667,0.846700000000000)
            rgb=(0.0779428571428572,0.503985714285714,0.838371428571429)
            rgb=(0.0793476190476190,0.520023809523810,0.831180952380952)
            rgb=(0.0749428571428571,0.537542857142857,0.826271428571429)
            rgb=(0.0640571428571428,0.556985714285714,0.823957142857143)
            rgb=(0.0487714285714286,0.577223809523810,0.822828571428572)
            rgb=(0.0343428571428572,0.596580952380952,0.819852380952381)
            rgb=(0.0265000000000000,0.613700000000000,0.813500000000000)
            rgb=(0.0238904761904762,0.628661904761905,0.803761904761905)
            rgb=(0.0230904761904762,0.641785714285714,0.791266666666667)
            rgb=(0.0227714285714286,0.653485714285714,0.776757142857143)
            rgb=(0.0266619047619048,0.664195238095238,0.760719047619048)
            rgb=(0.0383714285714286,0.674271428571429,0.743552380952381)
            rgb=(0.0589714285714286,0.683757142857143,0.725385714285714)
            rgb=(0.0843000000000000,0.692833333333333,0.706166666666667)
            rgb=(0.113295238095238,0.701500000000000,0.685857142857143)
            rgb=(0.145271428571429,0.709757142857143,0.664628571428572)
            rgb=(0.180133333333333,0.717657142857143,0.642433333333333)
            rgb=(0.217828571428571,0.725042857142857,0.619261904761905)
            rgb=(0.258642857142857,0.731714285714286,0.595428571428571)
            rgb=(0.302171428571429,0.737604761904762,0.571185714285714)
            rgb=(0.348166666666667,0.742433333333333,0.547266666666667)
            rgb=(0.395257142857143,0.745900000000000,0.524442857142857)
            rgb=(0.442009523809524,0.748080952380952,0.503314285714286)
            rgb=(0.487123809523809,0.749061904761905,0.483976190476191)
            rgb=(0.530028571428571,0.749114285714286,0.466114285714286)
            rgb=(0.570857142857143,0.748519047619048,0.449390476190476)
            rgb=(0.609852380952381,0.747314285714286,0.433685714285714)
            rgb=(0.647300000000000,0.745600000000000,0.418800000000000)
            rgb=(0.683419047619048,0.743476190476191,0.404433333333333)
            rgb=(0.718409523809524,0.741133333333333,0.390476190476190)
            rgb=(0.752485714285714,0.738400000000000,0.376814285714286)
            rgb=(0.785842857142857,0.735566666666667,0.363271428571429)
            rgb=(0.818504761904762,0.732733333333333,0.349790476190476)
            rgb=(0.850657142857143,0.729900000000000,0.336028571428571)
            rgb=(0.882433333333333,0.727433333333333,0.321700000000000)
            rgb=(0.913933333333333,0.725785714285714,0.306276190476191)
            rgb=(0.944957142857143,0.726114285714286,0.288642857142857)
            rgb=(0.973895238095238,0.731395238095238,0.266647619047619)
            rgb=(0.993771428571429,0.745457142857143,0.240347619047619)
            rgb=(0.999042857142857,0.765314285714286,0.216414285714286)
            rgb=(0.995533333333333,0.786057142857143,0.196652380952381)
            rgb=(0.988000000000000,0.806600000000000,0.179366666666667)
            rgb=(0.978857142857143,0.827142857142857,0.163314285714286)
            rgb=(0.969700000000000,0.848138095238095,0.147452380952381)
            rgb=(0.962585714285714,0.870514285714286,0.130900000000000)
            rgb=(0.958871428571429,0.894900000000000,0.113242857142857)
            rgb=(0.959823809523810,0.921833333333333,0.0948380952380953)
            rgb=(0.966100000000000,0.951442857142857,0.0755333333333333)
            rgb=(0.976300000000000,0.983100000000000,0.0538000000000000)
        },
    }
    
    \begin{groupplot}[
      group style={group size=4 by 2,
                   horizontal sep=2cm},
      width=3cm, height=3cm
    ]
    \captionsetup{labelfont={color=white,bf}}
    \nextgroupplot[ylabel={{\footnotesize $u\left (t;\left( \frac{1}{3}, \frac{17}{10}\right)\right)$}},
                   scale only axis,
                   enlargelimits=false,
                   axis on top,
                   axis equal image,
                   xticklabels={,,},
                   xlabel style={font=\footnotesize},
                   ylabel style={font=\footnotesize},
                   x tick label style={font=\footnotesize},
                   y tick label style={font=\footnotesize},
                   legend style={font=\tiny},
                   ytick={100/8,3/8*100,5/8*100,7/8*100},
                   yticklabels={$-3$,$-1$,$1$,$3$}]
        \addplot graphics[xmin=0,xmax=100,ymin=0,ymax=100] {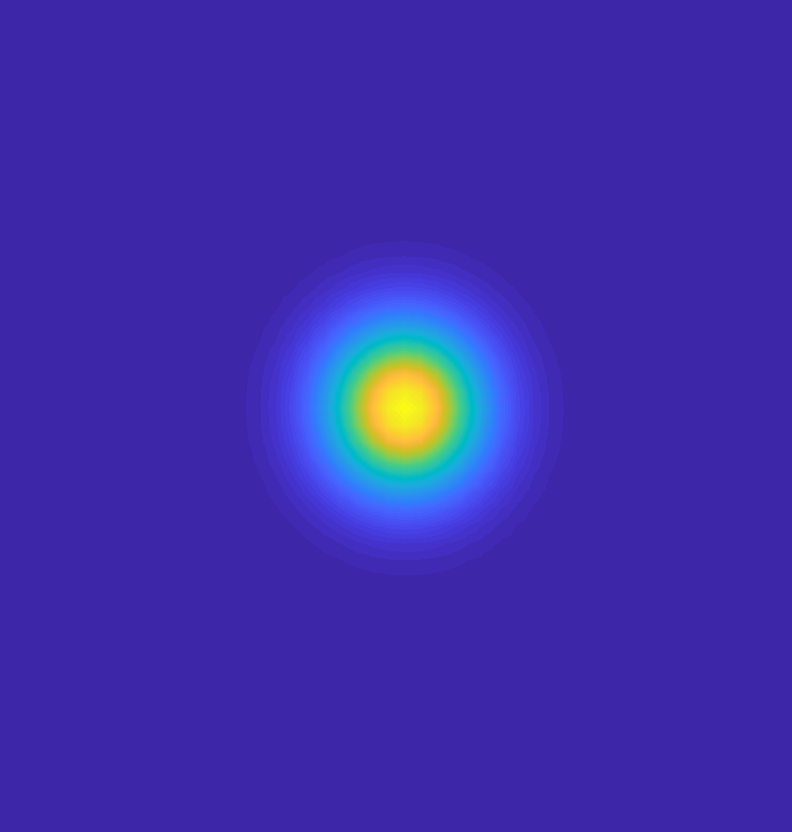};
        \coordinate (start) at (axis cs:0,120);
        \coordinate (first) at (axis cs:50.5,120);
        \coordinate (first1) at (axis cs:50.5,125);
        \node [text width=1em,anchor=north west] at (rel axis cs: 0.02,1.05) {\subcaption{\label{fig:surf_full_SW2D_1}}};
    \nextgroupplot[scale only axis,
                   enlargelimits=false,
                   axis on top,
                   axis equal image,
                   yticklabels={,,},
                   xticklabels={,,},
                   xlabel style={font=\footnotesize},
                   ylabel style={font=\footnotesize},
                   x tick label style={font=\footnotesize},
                   y tick label style={font=\footnotesize},
                   legend style={font=\tiny},
                   xshift=-1.7cm]
        \addplot graphics[xmin=0,xmax=100,ymin=0,ymax=100] {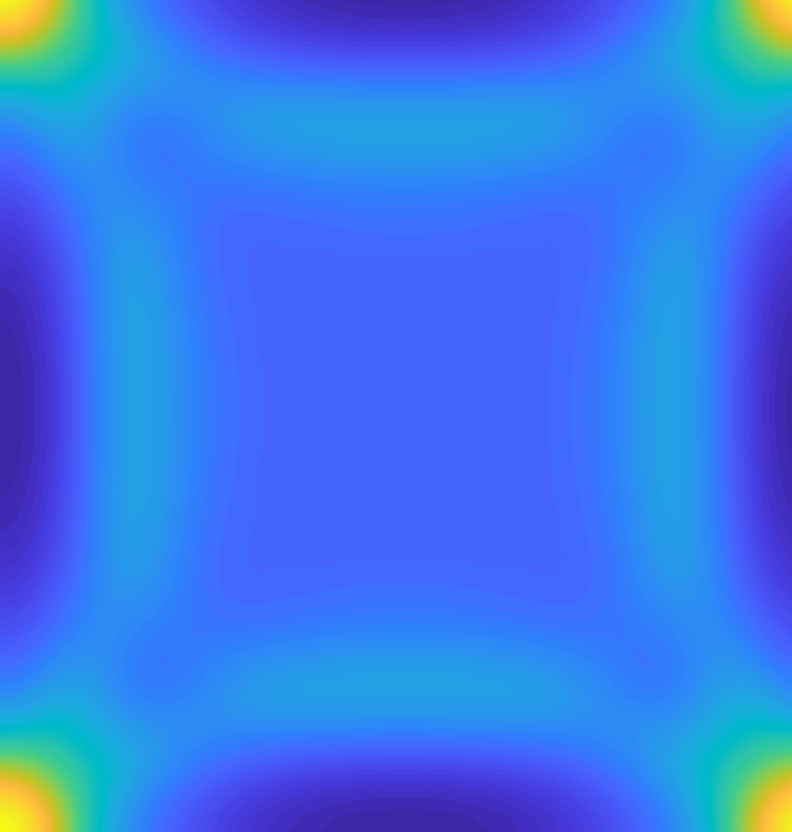};
        \coordinate (second) at (axis cs:50.5,120);
        \coordinate (second2) at (axis cs:50.5,125);
        \node [text width=1em,anchor=north west] at (rel axis cs: 0.02,1.05) {\subcaption{\label{fig:surf_full_SW2D_2}}};
    \nextgroupplot[scale only axis,
                   enlargelimits=false,
                   axis on top,
                   axis equal image,
                   yticklabels={,,},
                   xticklabels={,,},
                   xlabel style={font=\footnotesize},
                   ylabel style={font=\footnotesize},
                   x tick label style={font=\footnotesize},
                   y tick label style={font=\footnotesize},
                   legend style={font=\tiny},
                   xshift=-1.7cm]
        \addplot graphics[xmin=0,xmax=100,ymin=0,ymax=100] {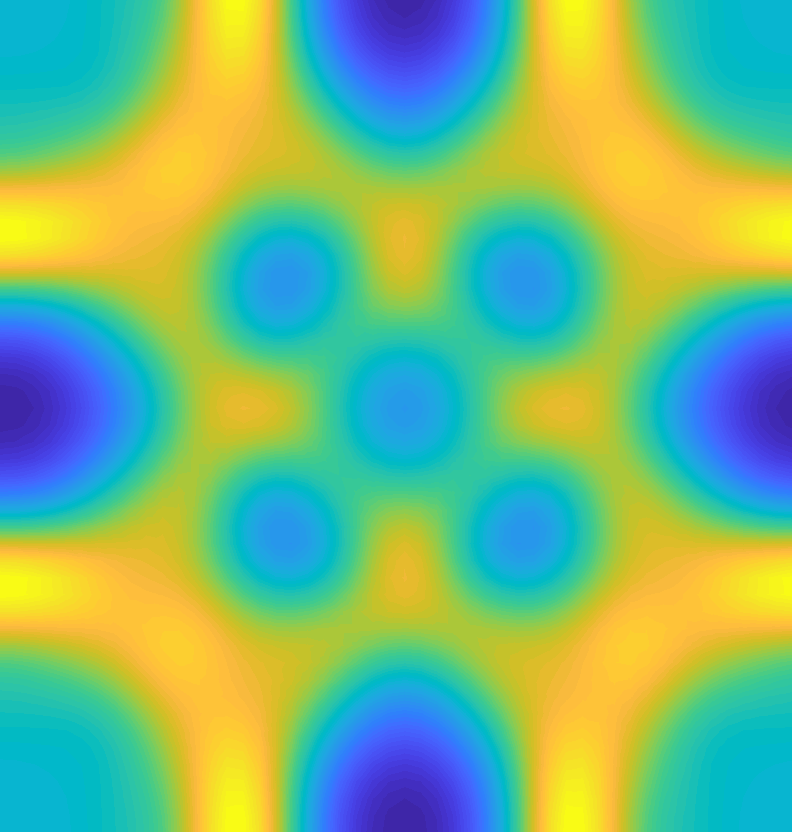};
        \coordinate (third) at (axis cs:50.5,120);
        \coordinate (third2) at (axis cs:50.5,125);
        \node [text width=1em,anchor=north west] at (rel axis cs: 0.02,1.05) {\subcaption{\label{fig:surf_full_SW2D_3}}};
    \nextgroupplot[scale only axis,
                   enlargelimits=false,
                   axis on top,
                   axis equal image,
                   yticklabels={,,},
                   xticklabels={,,},
                   xlabel style={font=\footnotesize},
                   ylabel style={font=\footnotesize},
                   x tick label style={font=\footnotesize},
                   y tick label style={font=\footnotesize},
                   legend style={font=\tiny},
                   xshift=-1.7cm,                   
                   colorbar sampled,
                   colorbar style={height=6.4cm,anchor=colorbar_pos},
                   point meta min=0.9417,
                   point meta max=1.3333]
        \addplot graphics[xmin=0,xmax=100,ymin=0,ymax=100] {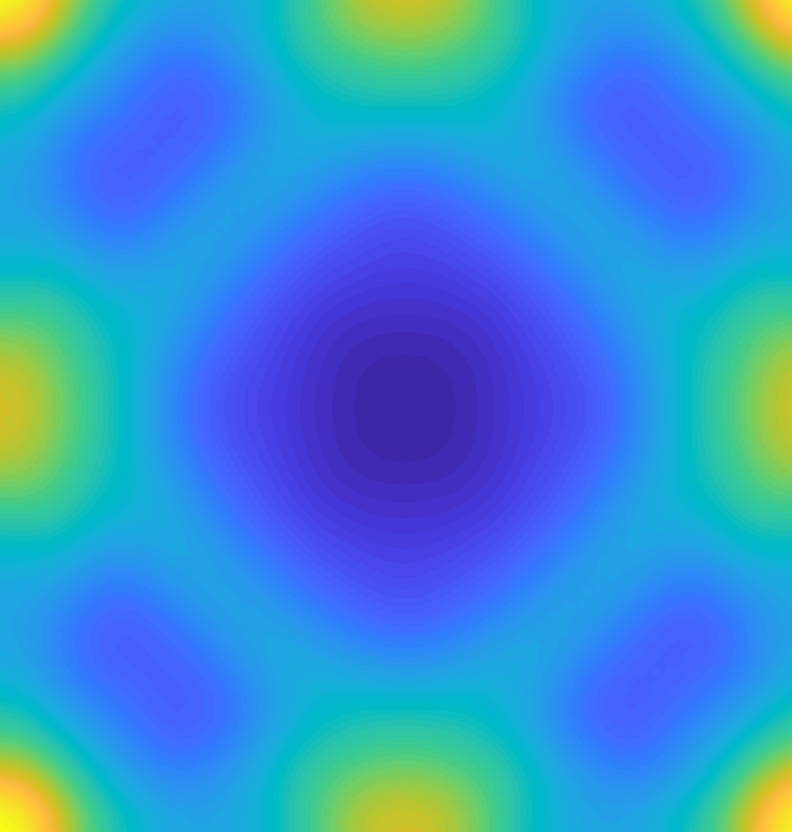};
        \coordinate (end) at (axis cs:100,120);
        \coordinate (colorbar_pos) at (axis cs:110,313);
        \coordinate (fourth) at (axis cs:50.5,120);
        \coordinate (fourth2) at (axis cs:50.5,125);
        \node [text width=1em,anchor=north west] at (rel axis cs: 0.02,1.05) {\subcaption{\label{fig:surf_full_SW2D_4}}};
    \nextgroupplot[ylabel={{\footnotesize{$\sum_{i=1}^{\Nr_{\tau}}U_i(t) Z_i\left (t;\left( \frac{1}{3}, \frac{17}{10}\right)\right)$}}},
                   scale only axis,
                   enlargelimits=false,
                   axis on top,
                   axis equal image,
                   xlabel style={font=\footnotesize},
                   ylabel style={font=\footnotesize},
                   x tick label style={font=\footnotesize},
                   y tick label style={font=\footnotesize},
                   legend style={font=\tiny},
                   yshift=0.6cm,
                   xtick={100/8,3/8*100,5/8*100,7/8*100},
                   xticklabels={$-3$,$-1$,$1$,$3$},
                   ytick={100/8,3/8*100,5/8*100,7/8*100},
                   yticklabels={$-3$,$-1$,$1$,$3$}]
        \addplot graphics[xmin=0,xmax=100,ymin=0,ymax=100] {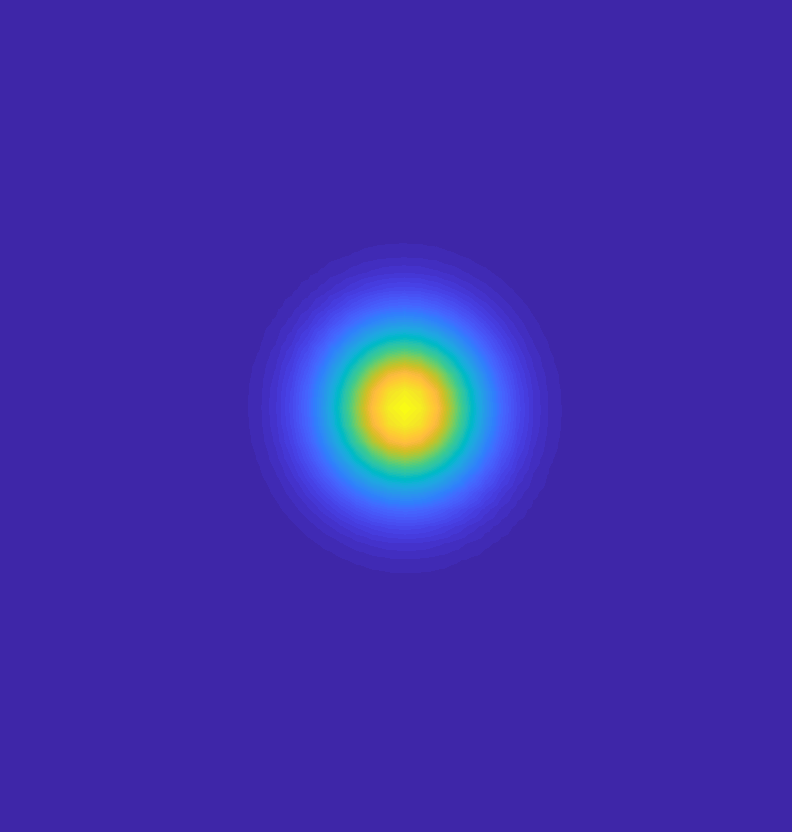};
        \node [text width=1em,anchor=north west] at (rel axis cs: 0.02,1.05) {\subcaption{\label{fig:surf_red_SW2D_1}}};
    \nextgroupplot[scale only axis,
                   enlargelimits=false,
                   axis on top,
                   axis equal image,
                   yticklabels={,,},
                   xlabel style={font=\footnotesize},
                   ylabel style={font=\footnotesize},
                   x tick label style={font=\footnotesize},
                   y tick label style={font=\footnotesize},
                   legend style={font=\tiny},
                   yshift=0.6cm,
                   xtick={100/8,3/8*100,5/8*100,7/8*100},
                   xticklabels={$-3$,$-1$,$1$,$3$}]
        \addplot graphics[xmin=0,xmax=100,ymin=0,ymax=100] {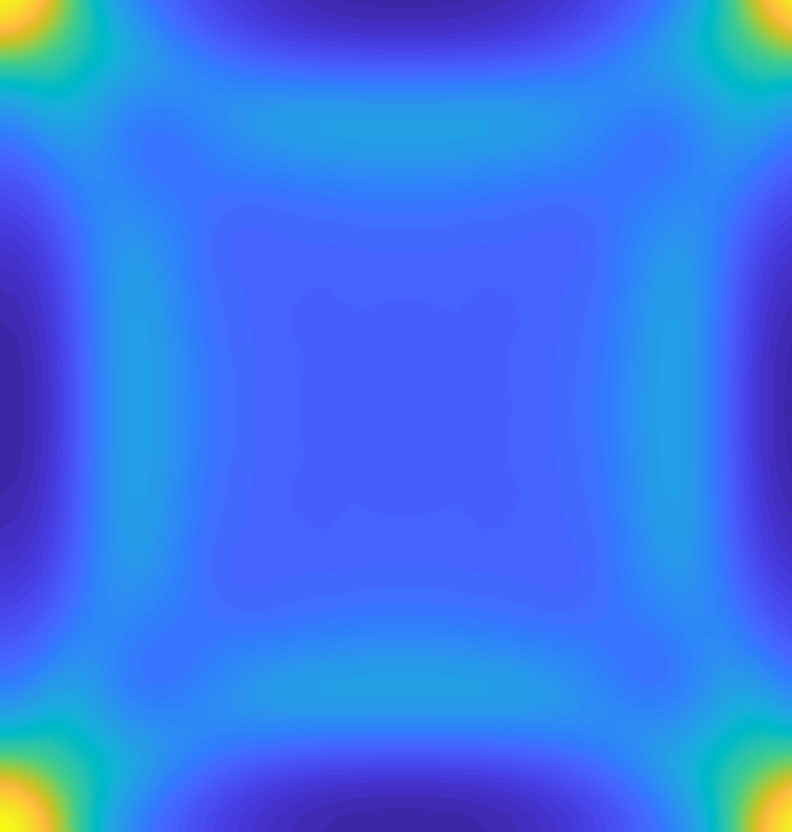};
        \node [text width=1em,anchor=north west] at (rel axis cs: 0.02,1.05) {\subcaption{\label{fig:surf_red_SW2D_2}}};
    \nextgroupplot[scale only axis,
                   enlargelimits=false,
                   axis on top,
                   axis equal image,
                   yticklabels={,,},
                   xlabel style={font=\footnotesize},
                   ylabel style={font=\footnotesize},
                   x tick label style={font=\footnotesize},
                   y tick label style={font=\footnotesize},
                   legend style={font=\tiny},
                   yshift=0.6cm,
                   xtick={100/8,3/8*100,5/8*100,7/8*100},
                   xticklabels={$-3$,$-1$,$1$,$3$}]
        \addplot graphics[xmin=0,xmax=100,ymin=0,ymax=100] {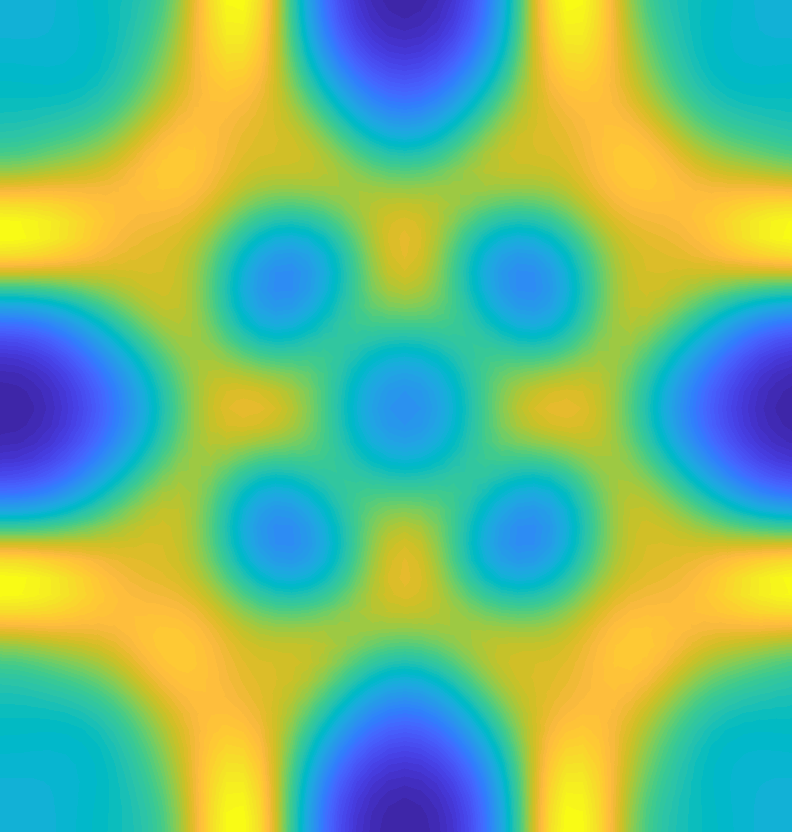};
        \node [text width=1em,anchor=north west] at (rel axis cs: 0.02,1.05) {\subcaption{\label{fig:surf_red_SW2D_3}}};
    \nextgroupplot[scale only axis,
                   enlargelimits=false,
                   axis on top,
                   axis equal image,
                   yticklabels={,,},
                   xlabel style={font=\footnotesize},
                   ylabel style={font=\footnotesize},
                   x tick label style={font=\footnotesize},
                   y tick label style={font=\footnotesize},
                   legend style={font=\tiny},
                   yshift=0.6cm,
                   xtick={100/8,3/8*100,5/8*100,7/8*100},
                   xticklabels={$-3$,$-1$,$1$,$3$}]
        \addplot graphics[xmin=0,xmax=100,ymin=0,ymax=100] {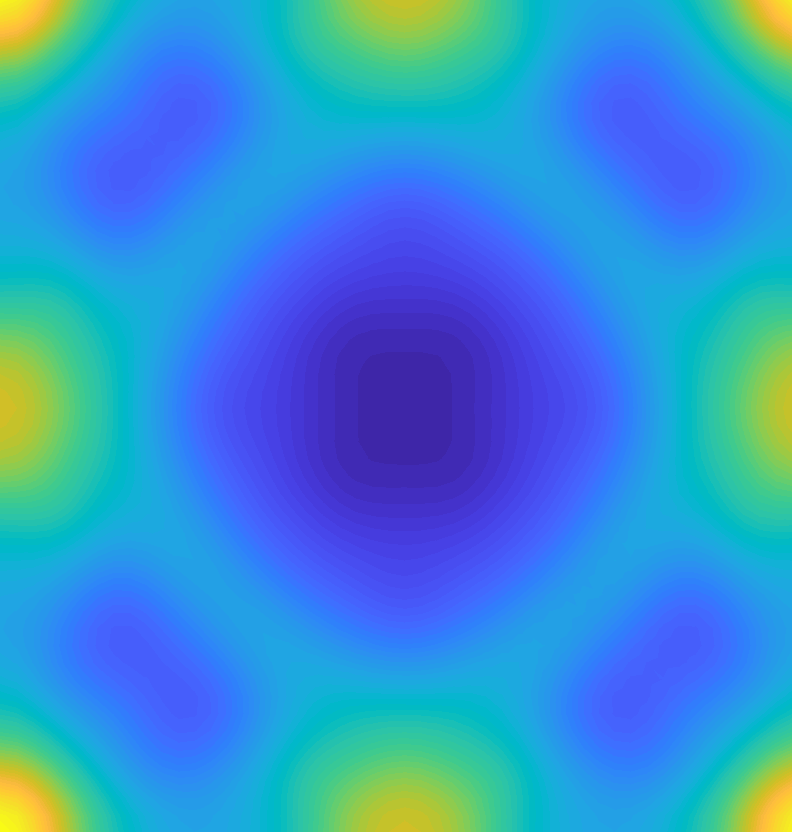};
        \node [text width=1em,anchor=north west] at (rel axis cs: 0.02,1.05) {\subcaption{\label{fig:surf_red_SW2D_4}}};
    \end{groupplot}
    \draw [->,line width=0.25mm] (start) -- (end);
    \node[circle,fill=black,inner sep=0pt,minimum size=3pt] (0) at (first) {};
    \node[inner sep=0pt,minimum size=3pt,label=above:{$t=0s$}] (0) at (first1) {};
    \node[circle,fill=black,inner sep=0pt,minimum size=3pt] (5) at (second) {};
    \node[inner sep=0pt,minimum size=3pt,label=above:{$t=5s$}] (5) at (second2) {};
    \node[circle,fill=black,inner sep=0pt,minimum size=3pt] (15) at (third) {};
    \node[inner sep=0pt,minimum size=3pt,label=above:{$t=15s$}] (15) at (third2) {};
    \node[circle,fill=black,inner sep=0pt,minimum size=3pt] (20) at (fourth) {};
    \node[inner sep=0pt,minimum size=3pt,label=above:{$t=20s$}] (20) at (fourth2) {};
\end{tikzpicture}
\caption{SWE-2D: High-fidelity solution
(Figs.~\bl{(a)-(d)})
and adaptive dynamical reduced solution
(Figs.~\bl{(e)-(h)})
for the parameter $\left(\alpha,\beta\right)=\left(\frac{1}{3},\frac{17}{10}\right)$ and $t=0,5,15$ and $20s$. In the adaptive reduced approach, we set $r=1.1$, $c=1.3$ and $\Nr_1=6$.}
\label{fig:surf_comparison_SW2D}
\end{figure}
Figure \ref{fig:error_runtime_SW2D} reports the error $E(T)$  vs. the runtime required to compute the solution for all $\prmh\in\Sprmh$ by means of the adaptive and non-adaptive dynamical reduced methods, for different values of $\Nr_1$, $r$ and $c$. 
Observe that the runtime of the high-fidelity solver is $3.29\cdot 10^{5}s$.
The results show that both reduction  methods are able to accurately approximate the high-fidelity solution, with speed-ups of $261$ for the non-adaptive approach and $113$ for the adaptive approach. The exceptional efficiency of the dynamical reduced approach in this context is a \bl{result} of the combination of three main factors: the low degree polynomial nonlinearity, the large number of degrees of freedom needed to represent the high-fidelity solution and the compact dimension of the local reduced manifold. Despite the small computational overhead for the adaptive method due to the error estimation, the basis update and the larger approximating spaces used, the adaptive algorithm leads to approximations that are one ($\Nr_1=4$) to two ($\Nr_1=10$) orders of magnitude more accurate than the approximations obtained by the non adaptive method.
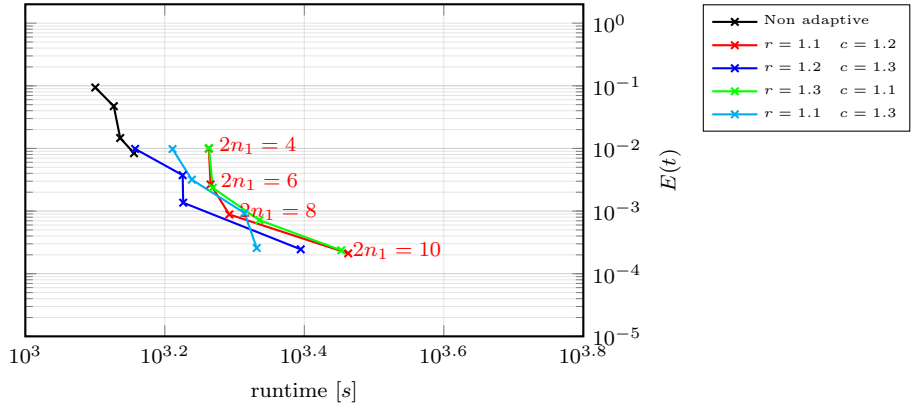
\begin{figure}[H]
\centering
\begin{tikzpicture}[spy using outlines={rectangle, width=4.4cm, height=5cm, magnification=1.9, connect spies}]
    \begin{axis}[xlabel={runtime $\left[s\right]$},
                 ylabel={$E(t)$},
                 axis line style = thick,
                 grid=both,
                 minor tick num=2,
                 grid style = {gray,opacity=0.2},
                 xmode=log,
                 ymode=log,
                 xmin = 1000, xmax = 10^(3.8),
                 ymax = 2, ymin = 0.00001,
                 every axis plot/.append style={thick},
                 width = 9cm, height = 6cm,
                 legend style={at={(1.4,1)},anchor=north},
                 legend cell align=left,
                 ylabel near ticks, yticklabel pos=right,
                 xlabel style={font=\footnotesize},
                 ylabel style={font=\footnotesize},
                 x tick label style={font=\footnotesize},
                 y tick label style={font=\footnotesize},
                 legend style={font=\tiny}]
        \addplot+[mark=x,color=black,size=2pt] table[x=Timing,y=Error]
            {figures/data/SW2D_Pareto/error_final_local_reduced_model_no_indicator_SW2D.txt};
        \addplot+[mark=x,color=red,size=2pt,
            every node near coord/.append style={xshift=0.65cm},
            every node near coord/.append style={yshift=-0.2cm},
            nodes near coords, 
            point meta=explicit symbolic,
            every node near coord/.append style={font=\footnotesize}] table[x=Timing,y=Error, meta index=2]
            {figures/data/SW2D_Pareto/error_final_local_reduced_model_indicator_SW2D_1.1_1.2.txt};  
        \addplot+[mark=x,color=blue,size=2pt] table[x=Timing,y=Error]
            {figures/data/SW2D_Pareto/error_final_local_reduced_model_indicator_SW2D_1.2_1.3.txt}; 
        \addplot+[mark=x,color=green,size=2pt] table[x=Timing,y=Error]
            {figures/data/SW2D_Pareto/error_final_local_reduced_model_indicator_SW2D_1.3_1.1.txt};  
        \addplot+[mark=x,color=cyan,size=2pt] table[x=Timing,y=Error]
            {figures/data/SW2D_Pareto/error_final_local_reduced_model_indicator_SW2D_1.1_1.3.txt};
        \legend{Non adaptive,
                $r=1.1 \quad c=1.2$,
                $r=1.2 \quad c=1.3$,
                $r=1.3 \quad c=1.1$,
                $r=1.1 \quad c=1.3$,
                Full Model};
    \end{axis}
\end{tikzpicture}
\caption{SWE-2D: Error \eqref{eqn:error_metric}, at time $T=20$, as a function of the runtime for the dynamical RB method ({\color{black}{\rule[.5ex]{1em}{1.2pt}}}) and the adaptive dynamical RB method for different values of the control parameters $r$ and $c$     ({\color{red}{\rule[.5ex]{1em}{1.2pt}}},{\color{blue}{\rule[.5ex]{1em}{1.2pt}}},{\color{green}{\rule[.5ex]{1em}{1.2pt}}},{\color{cyan}{\rule[.5ex]{1em}{1.2pt}}}) for the simulation of all the sampled parameters in $\Sprmh$.
For comparison, the high-fidelity model runtime is $3.3\cdot 10^{5}s$.}
\label{fig:error_runtime_SW2D}
\end{figure}
The results presented in Figures \ref{fig:error_basis_time_SWE2D} on the evolution of the error $E(t)$ for $2n_1=\{4,6,8\}$, corroborate the conclusions, already drawn from the 1D test case, regarding the effect of a poorly approximated initial condition on the performances of the adapting procedure. The evolution of the basis dimension is reported in Figures \ref{fig:time_basis_SWE2D_4}, \ref{fig:time_basis_SWE2D_6} and \ref{fig:time_basis_SWE2D_8} for different values of $r$, $c$ and $2n_{1}$.
\begin{figure}[H]
\centering
\begin{tikzpicture}
    \begin{groupplot}[
      group style={group size=2 by 3,
                   horizontal sep=2cm},
      width=7cm, height=4cm
    ]
    \nextgroupplot[ylabel={$E(t)$},
                   axis line style = thick,
                   grid=both,
                   minor tick num=2,
                   max space between ticks=20,
                   grid style = {gray,opacity=0.2},
                   every axis plot/.append style={ultra thick},
                   xmin=0, xmax=20,
                   ymode=log,
                   xlabel style={font=\footnotesize},
                   ylabel style={font=\footnotesize},
                   x tick label style={font=\footnotesize},
                   y tick label style={font=\footnotesize},
                   legend style={font=\tiny}]
        \addplot[color=black] table[x=Time,y=Error_4] {figures/data/SW2D_error_time_basis/error_local_reduced_model_no_indicator_SW2D.txt};
        \addplot+[color=red] table[x=Timing,y=Error_1.1_1.2] {figures/data/SW2D_error_time_basis/error_local_reduced_model_indicator_SW2D_4.txt};
        \addplot+[color=blue] table[x=Timing,y=Error_1.2_1.3] {figures/data/SW2D_error_time_basis/error_local_reduced_model_indicator_SW2D_4.txt}; 
        \addplot+[color=green] table[x=Timing,y=Error_1.3_1.1] {figures/data/SW2D_error_time_basis/error_local_reduced_model_indicator_SW2D_4.txt}; 
        \addplot+[color=cyan] table[x=Timing,y=Error_1.1_1.3] {figures/data/SW2D_error_time_basis/error_local_reduced_model_indicator_SW2D_4.txt};
        \addplot+[color=black,dashed] table[x=Time,y=Error_4] {figures/data/SW2D_error_time_basis/error_full_model_corrupted_SW2D.txt};
        \node [text width=1em,anchor=north west] at (rel axis cs: 0.01,1.05) {\subcaption{\label{fig:time_error_SWE2D_4}}};
    \nextgroupplot[ylabel={$\Nr_{\tau}$},
                   axis line style = thick,
                   grid=both,
                   minor tick num=2,
                   max space between ticks=20,
                   grid style = {gray,opacity=0.2},
                   every axis plot/.append style={ultra thick},
                   xmin=0, xmax=7,
                   xlabel style={font=\footnotesize},
                   ylabel style={font=\footnotesize},
                   x tick label style={font=\footnotesize},
                   y tick label style={font=\footnotesize},
                   legend style={font=\tiny}]
        \addplot[color=black] table[x=Time,y=Basis_4] {figures/data/SW2D_error_time_basis/basis_local_reduced_model_no_indicator_SW2D.txt};
        \addplot+[color=red] table[x=Timing,y=Basis_1.1_1.2] {figures/data/SW2D_error_time_basis/basis_local_reduced_model_indicator_SW2D_4.txt};
        \addplot+[color=blue] table[x=Timing,y=Basis_1.2_1.3] {figures/data/SW2D_error_time_basis/basis_local_reduced_model_indicator_SW2D_4.txt}; 
        \addplot+[color=green] table[x=Timing,y=Basis_1.3_1.1] {figures/data/SW2D_error_time_basis/basis_local_reduced_model_indicator_SW2D_4.txt}; 
        \addplot+[color=cyan] table[x=Timing,y=Basis_1.1_1.3] {figures/data/SW2D_error_time_basis/basis_local_reduced_model_indicator_SW2D_4.txt};
        \node [text width=1em,anchor=north west] at (rel axis cs: 0.01,1.05) {\subcaption{\label{fig:time_basis_SWE2D_4}}};
    \nextgroupplot[ylabel={$E(t)$},
                   axis line style = thick,
                   grid=both,
                   minor tick num=2,
                   max space between ticks=20,
                   grid style = {gray,opacity=0.2},
                   every axis plot/.append style={ultra thick},
                   xmin=0, xmax=20,
                   ymax = 4,
                   ymode=log,
                   xlabel style={font=\footnotesize},
                   ylabel style={font=\footnotesize},
                   x tick label style={font=\footnotesize},
                   y tick label style={font=\footnotesize},
                   legend style={font=\tiny},
                   legend style={at={(axis cs:2.3,0.06)},anchor=south west},
                   legend columns = 2,
                   legend style={nodes={scale=0.83, transform shape}}]
        \addplot[color=black] table[x=Time,y=Error_6] {figures/data/SW2D_error_time_basis/error_local_reduced_model_no_indicator_SW2D.txt};
        \addplot+[color=black,dashed] table[x=Time,y=Error_6] {figures/data/SW2D_error_time_basis/error_full_model_corrupted_SW2D.txt};
        \addplot+[color=red] table[x=Timing,y=Error_1.1_1.2] {figures/data/SW2D_error_time_basis/error_local_reduced_model_indicator_SW2D_6.txt};
        \addplot+[color=blue] table[x=Timing,y=Error_1.2_1.3] {figures/data/SW2D_error_time_basis/error_local_reduced_model_indicator_SW2D_6.txt}; 
        \addplot+[color=green] table[x=Timing,y=Error_1.3_1.1] {figures/data/SW2D_error_time_basis/error_local_reduced_model_indicator_SW2D_6.txt}; 
        \addplot+[color=cyan] table[x=Timing,y=Error_1.1_1.3] {figures/data/SW2D_error_time_basis/error_local_reduced_model_indicator_SW2D_6.txt};
        \node [text width=1em,anchor=north west] at (rel axis cs: 0.01,1.05) {\subcaption{\label{fig:time_error_SWE2D_6}}};
        \legend{{Non adaptive},
                {Target},
                {$r=1.1, \, c=1.2$},
                {$r=1.2, \, c=1.3$},
                {$r=1.3, \, c=1.1$},
                {$r=1.1, \, c=1.3$}};
    \nextgroupplot[ylabel={$\Nr_{\tau}$},
                   axis line style = thick,
                   grid=both,
                   minor tick num=2,
                   max space between ticks=20,
                   grid style = {gray,opacity=0.2},
                   every axis plot/.append style={ultra thick},
                   xmin=0, xmax=20,
                   xlabel style={font=\footnotesize},
                   ylabel style={font=\footnotesize},
                   x tick label style={font=\footnotesize},
                   y tick label style={font=\footnotesize},
                   legend style={font=\tiny}]
        \addplot[color=black] table[x=Time,y=Basis_6] {figures/data/SW2D_error_time_basis/basis_local_reduced_model_no_indicator_SW2D.txt};
        \addplot+[color=red] table[x=Timing,y=Basis_1.1_1.2] {figures/data/SW2D_error_time_basis/basis_local_reduced_model_indicator_SW2D_6.txt};
        \addplot+[color=blue] table[x=Timing,y=Basis_1.2_1.3] {figures/data/SW2D_error_time_basis/basis_local_reduced_model_indicator_SW2D_6.txt}; 
        \addplot+[color=green] table[x=Timing,y=Basis_1.3_1.1] {figures/data/SW2D_error_time_basis/basis_local_reduced_model_indicator_SW2D_6.txt}; 
        \addplot+[color=cyan] table[x=Timing,y=Basis_1.1_1.3] {figures/data/SW2D_error_time_basis/basis_local_reduced_model_indicator_SW2D_6.txt};
        \node [text width=1em,anchor=north west] at (rel axis cs: 0.01,1.05) {\subcaption{\label{fig:time_basis_SWE2D_6}}};
    \nextgroupplot[xlabel={time $\left [ s \right ]$},
                   ylabel={$E(t)$},
                   axis line style = thick,
                   grid=both,
                   minor tick num=2,
                   max space between ticks=20,
                   grid style = {gray,opacity=0.2},
                   every axis plot/.append style={ultra thick},
                   xmin=0, xmax=20,
                   ymode=log,
                   xlabel style={font=\footnotesize},
                   ylabel style={font=\footnotesize},
                   x tick label style={font=\footnotesize},
                   y tick label style={font=\footnotesize},
                   legend style={font=\tiny}]
        \addplot[color=black] table[x=Time,y=Error_8] {figures/data/SW2D_error_time_basis/error_local_reduced_model_no_indicator_SW2D.txt};
        \addplot+[color=red] table[x=Timing,y=Error_1.1_1.2] {figures/data/SW2D_error_time_basis/error_local_reduced_model_indicator_SW2D_8.txt};
        \addplot+[color=blue] table[x=Timing,y=Error_1.2_1.3] {figures/data/SW2D_error_time_basis/error_local_reduced_model_indicator_SW2D_8.txt}; 
        \addplot+[color=green] table[x=Timing,y=Error_1.3_1.1] {figures/data/SW2D_error_time_basis/error_local_reduced_model_indicator_SW2D_8.txt}; 
        \addplot+[color=cyan] table[x=Timing,y=Error_1.1_1.3] {figures/data/SW2D_error_time_basis/error_local_reduced_model_indicator_SW2D_8.txt};
        \addplot+[color=black,dashed] table[x=Time,y=Error_8] {figures/data/SW2D_error_time_basis/error_full_model_corrupted_SW2D.txt};
        \node [text width=1em,anchor=north west] at (rel axis cs: 0.01,1.05) {\subcaption{\label{fig:time_error_SWE2D_8}}};
    \nextgroupplot[xlabel={time $\left [ s \right ]$},
                   ylabel={$\Nr_{\tau}$},
                   axis line style = thick,
                   grid=both,
                   minor tick num=2,
                   max space between ticks=20,
                   grid style = {gray,opacity=0.2},
                   every axis plot/.append style={ultra thick},
                   xmin=0, xmax=20,
                   xlabel style={font=\footnotesize},
                   ylabel style={font=\footnotesize},
                   x tick label style={font=\footnotesize},
                   y tick label style={font=\footnotesize},
                   legend style={font=\tiny}]
        \addplot[color=black] table[x=Time,y=Basis_8] {figures/data/SW2D_error_time_basis/basis_local_reduced_model_no_indicator_SW2D.txt};
        \addplot+[color=red] table[x=Timing,y=Basis_1.1_1.2] {figures/data/SW2D_error_time_basis/basis_local_reduced_model_indicator_SW2D_8.txt};
        \addplot+[color=blue] table[x=Timing,y=Basis_1.2_1.3] {figures/data/SW2D_error_time_basis/basis_local_reduced_model_indicator_SW2D_8.txt}; 
        \addplot+[color=green] table[x=Timing,y=Basis_1.3_1.1] {figures/data/SW2D_error_time_basis/basis_local_reduced_model_indicator_SW2D_8.txt}; 
        \addplot+[color=cyan] table[x=Timing,y=Basis_1.1_1.3] {figures/data/SW2D_error_time_basis/basis_local_reduced_model_indicator_SW2D_8.txt};
        \node [text width=1em,anchor=north west] at (rel axis cs: 0.01,1.05) {\subcaption{\label{fig:time_basis_SWE2D_8}}};
    \end{groupplot}
\end{tikzpicture}
\caption{SWE-2D: On the left column, we report the evolution of the error $E(t)$ \eqref{eqn:error_metric} for the adaptive and non adaptive dynamical RB methods for different values of the control parameters $r$ and $c$, and for different dimensions $\Nr_1$ of the initial reduced manifold. The target error is obtained by solving the full model with initial condition obtained by projecting \eqref{eq:init_cond_SWE2D} onto a symplectic manifold of dimension $\Nr_1$. On the right column, we report the evolution of the dimension of the dynamical reduced basis over time. The adaptive algorithm is driven by the error indicator \eqref{eq:ratio_tmp}, while in the non adaptive setting, the dimension does not change with time. We consider the cases
$\Nr_1=4$ (Figs.~\bl{(a)-(b)}),
$\Nr_1=6$ (Figs.~\bl{(c)-(d)}) and
$\Nr_1=8$ (Figs.~\bl{(e)-(f)}).}
\label{fig:error_basis_time_SWE2D}
\end{figure}
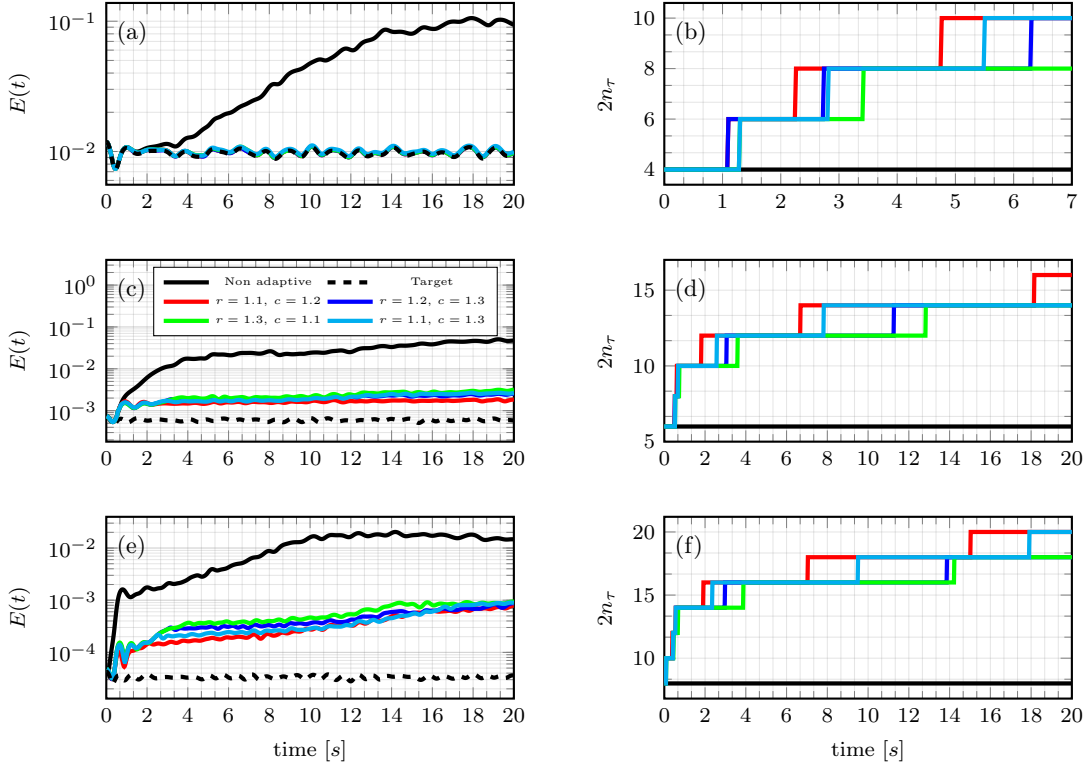


\subsection{\bl{Two-dimensional} nonlinear Schr\"odinger equation}
The nonlinear Schr\"odinger equation (NLS) is used to model, among others, the propagation of light in nonlinear optical fibers and planar waveguides and to describe the Bose--Einstein condensates in a macroscopic gaseous superfluid wave-matter state at ultra-cold temperature. 
\bl{In} the 2D setting, we test the adaptive strategy in the case of a Fourier mode cascade, where, starting from an initial condition represented by few low Fourier modes, the energy exchange to higher modes quickly complicates the dynamic of the problem \cite{caputo2011fourier}. 
More specifically, in the spatial
domain $\Omega$,
we consider the cubic Schr\"odinger equation
\begin{equation}\label{eq:schroedinger}
    \left\{
    \begin{aligned}
    & i \dfrac{\partial u}{\partial t} + \bl{\Delta u} + |u|^2 u=0, & \mbox{in}\;\Omega\times \Tcal,\\
    & u(t_0,x;\prm) = u^0(x;\prm), &\mbox{in}\;\Omega,
    \end{aligned}\right.
\end{equation}
with periodic boundary conditions, and \bl{vector-valued} parameter $\prm$.
By writing the complex-valued solution $u$ in terms of its
real and imaginary parts as
$u=q+ i v$, \eqref{eq:schroedinger}
can be written as a Hamiltonian system in canonical symplectic form with
Hamiltonian
\begin{equation*}
    \HamN(\bl{u};\prm) = \dfrac12 \int_{\Omega} \bigg[
    \bigg(\dfrac{\partial q}{\partial x}\bigg)^2+
    \bigg(\dfrac{\partial v}{\partial x}\bigg)^2+
    \bl{\bigg(\dfrac{\partial q}{\partial y}\bigg)^2+}    
    \bl{\bigg(\dfrac{\partial v}{\partial y}\bigg)^2}-
    \dfrac{1}{2} (q^2+v^2)^2
    \bigg]\,dx dy.
\end{equation*}

Let us consider the spatial domain $\Omega=[-2\pi, 2\pi]^2$ and the set of parameters $\Gamma = \left [ 0.97,1.03 \right ]^2$. We seek the numerical solution to 
\eqref{eq:schroedinger}, for $\Np=64$ uniformly sampled parameters $\eta_h:=(\alpha,\beta)\in\Gamma_h$ entering the initial condition
\begin{equation}\label{eqn:initial_condition_NLS2D}
    u^{0}(x,y;\prmh) = \left ( 1+\alpha\sin{x}\right )\left ( 2+\beta\sin{y}\right ). 
\end{equation}
This problem is characterized by an energy exchange between Fourier modes. Although this process is local, it is not well understood how the energy exchange mechanism is influenced by the problem dimension and parameters.
In particular, although the values of $\alpha$ and $\beta$ have a limited impact on the low-rank structure of the initial condition \eqref{eqn:initial_condition_NLS2D}, the explicit effect of their variation on the energy exchange process is not known. 
We use a centered finite difference scheme to discretize the Laplacian operator.
The domain $\Omega$ is partitioned using \bl{$M=101$} nodes per dimension, for a total of $N=10000$ points \bl{and with $\Delta x=\Delta y=4\pi\cdot 10^{-2}$}. Let $u_h(t;\prmh)$, for all $t\in\Tcal$ and $\prmh\in\Sprmh$, be the vector collecting 
the degrees of freedom associated with the nodal approximation of $u$. The semi-discrete problem is canonically Hamiltonian with the discrete Hamiltonian function
\begin{equation*}
\begin{aligned}
    \HamN_h(u_h;\prmh) = 
    \dfrac12 \bl{\sum_{i,j=1}^{M}}
    \bigg[ &
    \bigg(\dfrac{q_{i+1,j}-q_{i,j}}{\Delta x}\bigg)^2 +
    \bigg(\dfrac{v_{i+1,j}-v_{i,j}}{\Delta x}\bigg)^2 +\\
    & \bigg(\dfrac{q_{i,j+1}-q_{i,j}}{\Delta y}\bigg)^2 +
    \bigg(\dfrac{v_{i,j+1}-v_{i,j}}{\Delta y}\bigg)^2
    -\dfrac{1}{2}(q_{i,j}^2+v_{i,j}^2)^2
    \bigg],
    \end{aligned}
\end{equation*}
with periodic boundary conditions for $q_{i,j}$ and $v_{i,j}$.
We consider $N_{\tau}=12000$ time steps in the interval $\Tcal=\left (0,T:=3 \right ]$ so that $\Delta t = 2.5\cdot 10^{-4}$. As in the previous examples, the implicit midpoint rule is used as the numerical integrator in the high-fidelity solver.
The reduced dynamics \eqref{eq:UZred} is integrated using the 2-stage partitioned RK method.

To assess the reducibility of the problem, we collect in $\mathcal{S}\in\mathbb{R}^{\Nf\times(N_{\tau}\Np)}$
the snapshots associated with all parameters $\eta_h$ and times $t^{\tau}$, and in $\mathcal{S}_{\tau}\in\mathbb{R}^{\Nf\times \Np}$
the snapshots associated with all parameters $\eta_h$ at fixed time $t^{\tau}$, with $\tau=1,\dots,N_{\tau}$.
The slow decay of the singular values of $\mathcal{S}$, reported in Figure \ref{fig:singular_values_NLS2D_a}, suggests that a global reduced basis approach is not viable for model order reduction.
The growing complexity of the high-fidelity solution, associated with different values of $\alpha$ and $\beta$, is reflected by the growth of the numerical rank shown in Figure \ref{fig:singular_values_NLS2D_b}. Hence, despite the exponential decay of the singular values of $\mathcal{S}_{\tau}$, Figure \ref{fig:singular_values_NLS2D_b} indicates that this test represents a challenging problem even for the adaptive algorithm and a balance between accuracy and computational cost is necessary while adapting the dimension of the reduced manifold.
\begin{figure}[H]
\centering
\begin{tikzpicture}
    \begin{groupplot}[
      group style={group size=2 by 1,
                   horizontal sep=2cm},
      width=7cm, height=5cm
    ]
    \nextgroupplot[xlabel={index},
                   ylabel={singular values},
                   axis line style = thick,
                   grid=both,
                   minor tick num=2,
                   max space between ticks=20,
                   grid style = {gray,opacity=0.2},
                   ymode=log,
                   ylabel near ticks,
                   every axis plot/.append style={ultra thick},
                   xmin = 1, xmax=1199,
                   ymin = 1e-11, ymax = 1,
                   xlabel style={font=\footnotesize},
                   ylabel style={font=\footnotesize},
                   x tick label style={font=\footnotesize},
                   y tick label style={font=\footnotesize},
                   legend style={font=\tiny}
                   ]
        \addplot[color=black] table[x=Index,y=Values] {figures/data/NLS2D_singular_values/full_singular_values_NLS2D.txt};
        \addplot[color=red] table[x=Index,y=Values] {figures/data/NLS2D_singular_values/avg_singular_values_NLS2D.txt};
        \legend{global,local};
        \node [text width=1em,anchor=north west] at (rel axis cs: 0.08,1.05) {\subcaption{\label{fig:singular_values_NLS2D_a}}};
        \coordinate (spypoint) at (axis cs:2,0.01);
    \nextgroupplot[xlabel={time $\left [ s \right ]$},
                   ylabel={$\epsilon$-rank},
                   axis line style = thick,
                   grid=both,
                   minor tick num=2,
                   grid style = {gray,opacity=0.2},
                   every axis plot/.append style={ultra thick},
                   xmin = 0, xmax = 3,
                   ymin = 0, ymax = 60,
                   xlabel style={font=\footnotesize},
                   ylabel style={font=\footnotesize},
                   x tick label style={font=\footnotesize},
                   y tick label style={font=\footnotesize},
                   legend style={font=\tiny},
                   legend style={at={(1.2,1)},anchor=north}]
        \addplot+[] table[x=Time,y=Rank] {figures/data/NLS2D_epsilon_rank/1_epsilon_rank_NLS2D.txt};
        \addplot+[] table[x=Time,y=Rank] {figures/data/NLS2D_epsilon_rank/3_epsilon_rank_NLS2D.txt};
        \addplot+[] table[x=Time,y=Rank] {figures/data/NLS2D_epsilon_rank/5_epsilon_rank_NLS2D.txt};
        \addplot+[] table[x=Time,y=Rank] {figures/data/NLS2D_epsilon_rank/7_epsilon_rank_NLS2D.txt};
        \legend{$10^{-1}$,$10^{-3}$,$10^{-5}$,$10^{-7}$};
        \node [text width=1em,anchor=north west] at (rel axis cs: 0.16,1.05) {\subcaption{\label{fig:singular_values_NLS2D_b}}};
    \end{groupplot}
\end{tikzpicture}
\caption{NLS-2D: \bl{(a)} Singular values of the global snapshots matrix $\mathcal{S}$ and of the time average of the local trajectories matrix $\mathcal{S}_{\tau}$. The singular values are normalized using the largest singular value for each case. \bl{(b)} $\epsilon$-rank of the local trajectories matrix $\mathcal{S}_{\tau}$ for different values of $\epsilon$.}
\end{figure}
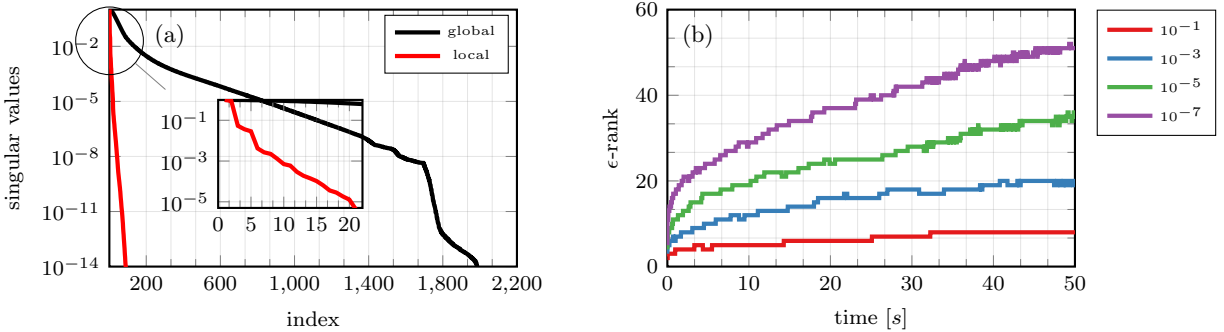
We consider several combinations of $r\in\left \{ 1.1, 1.2 \right \}$ and $c\in\left \{ 1.05,1.1,1.2 \right \}$ and different initial dimensions of the reduced manifold $\Nr_1\in\{6,8\}$. The error indicator is computed every $10$ time steps on a subset $\Gamma_I\subset\Gamma_h$ of $16$ uniformly sampled parameters. Both adaptive and non-adaptive reduced models are initialized using \eqref{eqn:initialization_local_low_rank}, with $U_0$ obtained via a complex SVD of the snapshots matrix $\mathcal{S}_1$ of the initial condition \eqref{eqn:initial_condition_NLS2D}. \bl{Figure \ref{fig:O_JO_error_NLS2D} confirms that the evolving basis $U$ generated by the dynamical reduced basis method satisfies the orthogonality and symplecticity constraints to machine precision.}

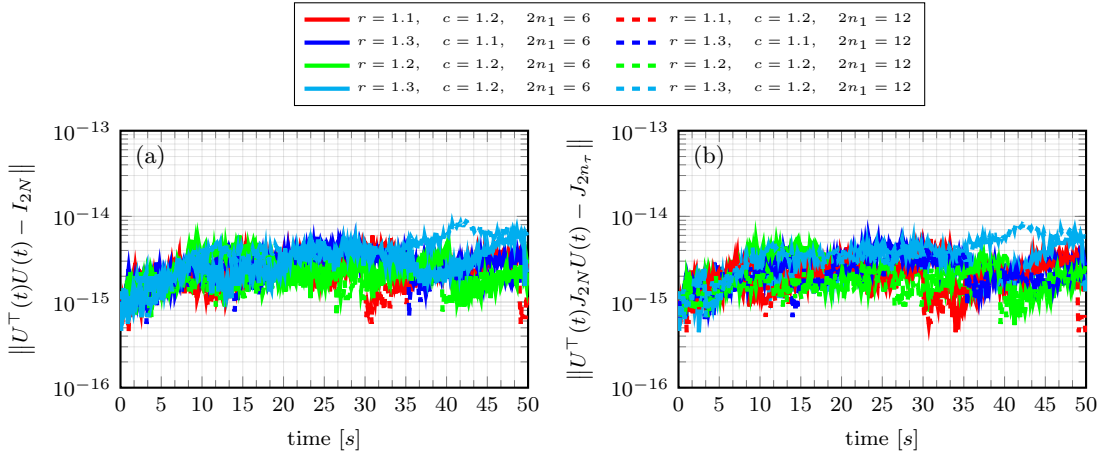
\begin{figure}[H]
\centering
\begin{tikzpicture}
    \begin{groupplot}[
      group style={group size=2 by 3,
                   horizontal sep=2cm},
      width=7cm, height=5cm
    ]
    \nextgroupplot[ylabel={$\left \| U^{\top}(t)U(t)-I_{2N} \right \|$},
                   xlabel={time $\left [ s \right ]$},
                   axis line style = thick,
                   grid=both,
                   minor tick num=2,
                   max space between ticks=20,
                   grid style = {gray,opacity=0.2},
                   xmin=0, xmax=3,
                   ymax = 10^(-14),
                   ymin = 10^(-16),
                   ymode=log,
                   xlabel style={font=\footnotesize},
                   ylabel style={font=\footnotesize},
                   x tick label style={font=\footnotesize},
                   y tick label style={font=\footnotesize},
                   legend style={font=\tiny},
                   legend columns = 2,
                   legend style={at={(1.2,1.5)},anchor=north}]
        \addplot+[color=red] table[x=Time,y=O_1.2_1.1] {figures/data/NLS2D_O_JO_orth/error_orth_NLS2D_6.txt};
        \addplot+[color=red,dashed] table[x=Time,y=O_1.2_1.1] {figures/data/NLS2D_O_JO_orth/error_orth_NLS2D_8.txt};
        \addplot+[color=blue] table[x=Time,y=O_1.2_1.05] {figures/data/NLS2D_O_JO_orth/error_orth_NLS2D_6.txt};
        \addplot+[color=blue,dashed] table[x=Time,y=O_1.2_1.05] {figures/data/NLS2D_O_JO_orth/error_orth_NLS2D_8.txt};
        \addplot+[color=green] table[x=Time,y=O_1.1_1.05] {figures/data/NLS2D_O_JO_orth/error_orth_NLS2D_6.txt};
        \addplot+[color=green,dashed] table[x=Time,y=O_1.1_1.05] {figures/data/NLS2D_O_JO_orth/error_orth_NLS2D_8.txt};
        \node [text width=1em,anchor=north west] at (rel axis cs: 0.02,1.05) {\subcaption{\label{fig:O_error_NLS2D}}};
        \legend{{$r=1.2, \quad c=1.1, \quad \Nr_{1}=6 \quad$},
                {$r=1.2, \quad c=1.1, \quad \Nr_{1}=8$},
                {$r=1.2, \quad c=1.05, \quad \Nr_{1}=6 \quad$},
                {$r=1.2, \quad c=1.05, \quad \Nr_{1}=8$},
                {$r=1.1, \quad c=1.05, \quad \Nr_{1}=6 \quad$},
                {$r=1.1, \quad c=1.05, \quad \Nr_{1}=8$},
                };
    \nextgroupplot[ylabel={$\left \| U^\top(t)\J{\Nf}U(t)-\J{\Nrt} \right \|$},
                   xlabel={time $\left [ s \right ]$},
                   axis line style = thick,
                   grid=both,
                   minor tick num=2,
                   max space between ticks=20,
                   grid style = {gray,opacity=0.2},
                   xmin=0, xmax=3,
                   ymax = 10^(-14),
                   ymin = 10^(-16),
                   ymode=log,
                   xlabel style={font=\footnotesize},
                   ylabel style={font=\footnotesize},
                   x tick label style={font=\footnotesize},
                   y tick label style={font=\footnotesize},
                   legend style={font=\tiny}]
        \addplot+[color=red] table[x=Time,y=JO_1.2_1.1] {figures/data/NLS2D_O_JO_orth/error_Jorth_NLS2D_6.txt};
        \addplot+[color=red,dashed] table[x=Time,y=JO_1.2_1.1] {figures/data/NLS2D_O_JO_orth/error_Jorth_NLS2D_8.txt};
        \addplot+[color=blue] table[x=Time,y=JO_1.2_1.05] {figures/data/NLS2D_O_JO_orth/error_Jorth_NLS2D_6.txt};
        \addplot+[color=blue,dashed] table[x=Time,y=JO_1.2_1.05] {figures/data/NLS2D_O_JO_orth/error_Jorth_NLS2D_8.txt};
        \addplot+[color=green] table[x=Time,y=JO_1.1_1.05] {figures/data/NLS2D_O_JO_orth/error_Jorth_NLS2D_6.txt};
        \addplot+[color=green,dashed] table[x=Time,y=JO_1.1_1.05] {figures/data/NLS2D_O_JO_orth/error_Jorth_NLS2D_8.txt};
        \node [text width=1em,anchor=north west] at (rel axis cs: 0.02,1.05) {\subcaption{\label{fig:JO_error_NLS2D}}};
    \end{groupplot}
\end{tikzpicture}
\caption{\bl{NLS-2D: Evolution of the error in the orthogonality (a) and symplecticity (b)
of the reduced basis obtained with the adaptive dynamical RB method for different choices of the control parameters $r$, $c$ and initial dimension of the reduced manifold $\Nr_1$.}}
\label{fig:O_JO_error_NLS2D}
\end{figure}

In line with the fact that the full model solution has a gradually increasing rank (Figure \ref{fig:singular_values_NLS2D_b}), adapting the dimension of the basis improves the accuracy of the approximation, as shown in Figure \ref{fig:error_time_NLS2D}.

In terms of the computational cost of the adaptive dynamical model, we record a speedup of at least $58$ times with respect to the high-fidelity model, whose runtime is $6.2 \cdot 10^{5}s$.
These results can be explained as for the 2D shallow water test: in the presence of polynomial nonlinearities the strategy proposed in \Cref{sec:nonlinear} allows computational costs that scale linearly with $\Nfh$ instead of $\Nfh^{\frac{1}{2}}$ and $\Nfh^{\frac{2}{3}}$ for problems ensuing from semi-discrete formulations of PDEs in $2$D and $3$D, respectively.

In Figure \ref{fig:error_time_NLS2D}, we observe that, although increasing in time, the error associated with the adaptive reduced dynamical model has a smaller slope than the error of the non-adaptive method.

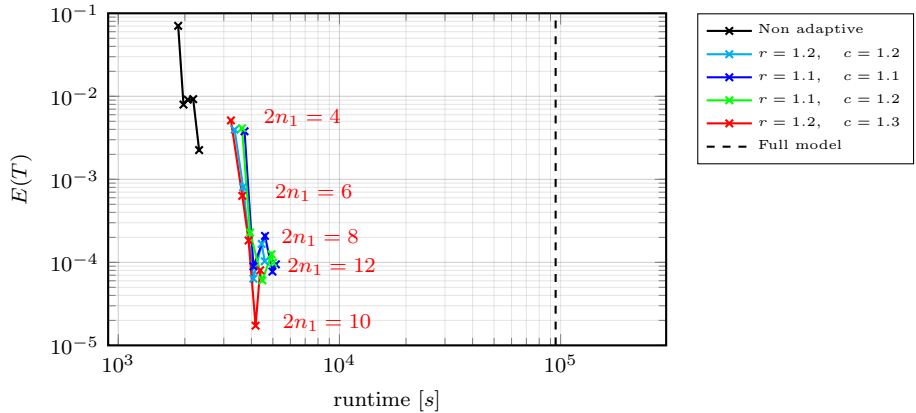
\begin{figure}[H]
\centering
\begin{tikzpicture}
    \begin{groupplot}[
      group style={group size=2 by 2,
                   horizontal sep=2cm},
      width=7cm, height=4cm
    ]
    \nextgroupplot[ylabel={$E(t)$},
                   axis line style = thick,
                   grid=both,
                   minor tick num=2,
                   max space between ticks=20,
                   grid style = {gray,opacity=0.2},
                   every axis plot/.append style={ultra thick},
                   xmin=0, xmax=3,
                   ymode=log,
                   xlabel style={font=\footnotesize},
                   ylabel style={font=\footnotesize},
                   x tick label style={font=\footnotesize},
                   y tick label style={font=\footnotesize},
                   legend style={font=\tiny}]
        \addplot[color=black] table[x=Time,y=Error_6] {figures/data/NLS2D_error_time_basis/error_local_reduced_model_no_indicator_NLS2D.txt};
        \addplot+[color=red] table[x=Timing,y=Error_1.2_1.05] {figures/data/NLS2D_error_time_basis/error_local_reduced_model_indicator_NLS2D_6.txt};
        \addplot+[color=green] table[x=Timing,y=Error_1.2_1.1] {figures/data/NLS2D_error_time_basis/error_local_reduced_model_indicator_NLS2D_6.txt}; 
        \addplot+[color=cyan] table[x=Timing,y=Error_1.1_1.05] {figures/data/NLS2D_error_time_basis/error_local_reduced_model_indicator_NLS2D_6.txt};
        \node [text width=1em,anchor=north west] at (rel axis cs: 0.01,1.05) {\subcaption{\label{fig:time_error_NLS2D_6}}};
    \nextgroupplot[ylabel={$\Nr_{\tau}$},
                   axis line style = thick,
                   grid=both,
                   minor tick num=2,
                   max space between ticks=20,
                   grid style = {gray,opacity=0.2},
                   every axis plot/.append style={ultra thick},
                   xmin=0, xmax=3,
                   xlabel style={font=\footnotesize},
                   ylabel style={font=\footnotesize},
                   x tick label style={font=\footnotesize},
                   y tick label style={font=\footnotesize},
                   legend style={font=\tiny}]
        \addplot[color=black] table[x=Time,y=Basis_6] {figures/data/NLS2D_error_time_basis/basis_local_reduced_model_no_indicator_NLS2D.txt};
        \addplot+[color=red] table[x=Timing,y=Basis_1.2_1.05] {figures/data/NLS2D_error_time_basis/basis_local_reduced_model_indicator_NLS2D_6.txt};
        \addplot+[color=green] table[x=Timing,y=Basis_1.2_1.1] {figures/data/NLS2D_error_time_basis/basis_local_reduced_model_indicator_NLS2D_6.txt}; 
        \addplot+[color=cyan] table[x=Timing,y=Basis_1.1_1.05] {figures/data/NLS2D_error_time_basis/basis_local_reduced_model_indicator_NLS2D_6.txt};
        \node [text width=1em,anchor=north west] at (rel axis cs: 0.01,1.05) {\subcaption{\label{fig:time_basis_NLS2D_6}}};
    \nextgroupplot[xlabel={time $\left [ s \right ]$},
                ylabel={$E(t)$},
                   axis line style = thick,
                   grid=both,
                   minor tick num=2,
                   max space between ticks=20,
                   grid style = {gray,opacity=0.2},
                   every axis plot/.append style={ultra thick},
                   xmin=0, xmax=3,
                   ymode=log,
                   ymax = 10^(1),
                   legend columns = 2,
                   xlabel style={font=\footnotesize},
                   ylabel style={font=\footnotesize},
                   x tick label style={font=\footnotesize},
                   y tick label style={font=\footnotesize},
                   legend style={font=\tiny},
                    legend style={nodes={scale=0.83, transform shape}}]
        \addplot[color=black] table[x=Time,y=Error_8] {figures/data/NLS2D_error_time_basis/error_local_reduced_model_no_indicator_NLS2D.txt};
        \addplot+[color=red] table[x=Timing,y=Error_1.2_1.05] {figures/data/NLS2D_error_time_basis/error_local_reduced_model_indicator_NLS2D_8.txt};
        \addplot+[color=green] table[x=Timing,y=Error_1.2_1.1] {figures/data/NLS2D_error_time_basis/error_local_reduced_model_indicator_NLS2D_8.txt}; 
        \addplot+[color=cyan] table[x=Timing,y=Error_1.1_1.05] {figures/data/NLS2D_error_time_basis/error_local_reduced_model_indicator_NLS2D_8.txt};
        \node [text width=1em,anchor=north west] at (rel axis cs: 0.01,1.05) {\subcaption{\label{fig:time_error_NLS2D_8}}};
        \legend{{Non adaptive},
                {$r=1.2, \, c=1.05$},
                {$r=1.2, \, c=1.1$},
                {$r=1.1, \, c=1.05$},
                {Target}};
    \nextgroupplot[xlabel={time $\left [ s \right ]$},
                   ylabel={$\Nr_{\tau}$},
                   axis line style = thick,
                   grid=both,
                   minor tick num=2,
                   max space between ticks=20,
                   grid style = {gray,opacity=0.2},
                   every axis plot/.append style={ultra thick},
                   xmin=0, xmax=3,
                   xlabel style={font=\footnotesize},
                   ylabel style={font=\footnotesize},
                   x tick label style={font=\footnotesize},
                   y tick label style={font=\footnotesize},
                   legend style={font=\tiny}]
        \addplot[color=black] table[x=Time,y=Basis_8] {figures/data/NLS2D_error_time_basis/basis_local_reduced_model_no_indicator_NLS2D.txt};
        \addplot+[color=red] table[x=Timing,y=Basis_1.2_1.05] {figures/data/NLS2D_error_time_basis/basis_local_reduced_model_indicator_NLS2D_8.txt};
        \addplot+[color=green] table[x=Timing,y=Basis_1.2_1.1] {figures/data/NLS2D_error_time_basis/basis_local_reduced_model_indicator_NLS2D_8.txt}; 
        \addplot+[color=cyan] table[x=Timing,y=Basis_1.1_1.05] {figures/data/NLS2D_error_time_basis/basis_local_reduced_model_indicator_NLS2D_8.txt};
        \node [text width=1em,anchor=north west] at (rel axis cs: 0.01,1.05) {\subcaption{\label{fig:time_basis_NLS2D_8}}};
    \end{groupplot}
\end{tikzpicture}
\caption{NLS-2D: On the left column, we report the evolution of the error $E(t)$ \eqref{eqn:error_metric} for the adaptive and non adaptive dynamical RB methods for different values of the control parameters $r$ and $c$, and for different dimensions $\Nr_1$ of the initial reduced manifold.
On the right column, we report the evolution of the dimension of the dynamical reduced basis over time. We consider the cases
$\Nr_1=6$ (Figs.~\bl{(a)-(b)}) and
$\Nr_1=8$ (Figs.~\bl{(c)-(d)}).}
\label{fig:error_time_NLS2D}
\end{figure}

\subsection{Vlasov--Poisson plasma model with forced electric field}
The Vlasov--Poisson system describes the dynamics of a collisionless magnetized plasma
under the action of a self-consistent electric field.
The evolution of the plasma at any time $t\in \Tcal\subset\mathbb{R}$ is described in terms of
the distribution function
$f^s(t,x,v)$ ($s$ denotes the particle species) in the Cartesian phase space domain
$(x,v)\in\Omega:=\Omega_x\times\Omega_v\subset\mathbb{R}^2$.
In this work, we consider the one-species ($s=1$) paraxial approximation of the Vlasov-Poisson equation, used in the study of long and thin beams of particles \cite{hirstoaga2019design}. More specifically, we assume that the beam has reached a stationary state, the longitudinal length of the beam is the predominant spatial scale and the velocity along the longitudinal direction is constant. Moreover, we look at the case in which the effects of the self-consistent electric field $E$ are negligible compared to the ones caused by an external electric field that we denote by $\Xi$. The external electric field is assumed to be independent of time and periodic with respect to the longitudinal dimension. Using the scaling argument proposed in \cite{frenod2009long} and the aforementioned assumptions, the problem is: For $f_0\in V_{|_{t=0}}$, find $f\in C^1(\Tcal;L^2(\Omega))\cap C^0(\Tcal;V)$ such that
\begin{equation}\label{eq:VP_paraxial}
\begin{aligned}
	\partial_t f+\dfrac{1}{\bl{\nu}}v\,\partial_x f +  \Xi\,\partial_v f=0, &\qquad\mbox{in}\;\Omega\times\Tcal,\\
	f(0,x,v)=f_0,														  &\qquad\mbox{in}\;\Omega,
\end{aligned}
\end{equation}
where the electric field $\Xi$ is prescribed at all $t\in\Tcal$, $x\in\Omega_x$,
the parameter $\bl{\nu}\in\mathbb{R}$ represents a spatial scaling and
the Vlasov equation has been normalized so that mass and charge are set to $m=q=1$.
In \eqref{eq:VP_paraxial}, since we are considering stationary states, the variable $t$ can be interpreted as the longitudinal coordinate and $\mathcal{T}$ as the longitudinal spatial domain. 

For the semi-discrete approximation of the Vlasov equation in \eqref{eq:VP_paraxial}
we consider a particle method:
The distribution function $f$ is approximated by the superposition of $P\in\mathbb{N}$ computational macro-particles each having a weight $\omega_\ell$, so that
\begin{equation*}
    f(t,x,v)\approx f_h(t,x,v) = \sum_{\ell=1}^P \omega_\ell\, S(x-X_\ell(t)) S(v-V_\ell(t)),
\end{equation*}
where $X(t)$ and $V(t)$ are the vector of the position and velocity of the macro-particles, respectively,
and $S$ is a compactly supported shape function,
here chosen to be the Dirac delta.
The idea of particle methods is to derive the time evolution of the approximate distribution function $f_h$ by advancing the macro-particles along the characteristics of the Vlasov equation.
Particle methods, like particle-in-cell (PIC), are widely use in the numerical simulation of plasma problems. However, the slow convergence requires the use of many particles to achieve sufficient accuracy and therefore PIC methods are expensive.
Model order reduction, in the number of macro-particles, of these semi-discrete schemes
can be crucial and potentially extremely beneficial.

The particle approximation of problem \eqref{eq:VP_paraxial}
yields a Hamiltonian system where the unknowns are the vectors of position $X$
and velocity $V$ of the particles with the discrete Hamiltonian reads 
\begin{equation}\label{eq:VPHamN}
    \begin{aligned}
        \HamN_h(f_h) = \sum_{\ell=1}^P \dfrac{1}{2\bl{\nu}}\, \omega_\ell V_\ell(t)^2
        \bl{+}\phi(X_{\ell}(t))
         = \dfrac{1}{2\bl{\nu}} V(t)^\top W_p V(t) \bl{+}\phi(X(t)).
    \end{aligned}
\end{equation}
Here $\phi$ denotes the potential, defined as $\Xi(x) = -\partial_x \phi(x)$, for all $x\in\Omega_x$, $W_p := \mathrm{diag}(\omega_1,\ldots,\omega_P)$,
and $\mathrm{diag}(d)$ denotes the diagonal matrix with diagonal elements given by the vector $d$.

For this test we consider $\Nfh=1000$ particles with uniform unitary weight, $\omega_i=1$, for all $i=1,\dots,\Nfh$. The external electric field is given as $\Xi(t,x)=-x^3$ for all $t\in\Tcal$ and $x\in\Omega_x$. The entries of the initial position $X(0)$ and velocity $V(0)$ vectors are independently sampled from the perturbed Maxwellian
\begin{equation}\label{eqn:VPInitial}
 f(0,x,v)= \left( \dfrac{1}{\sqrt{2\pi}\alpha}e^{-0.5v^2\alpha^{-2}}\right)
 \left( 1+\beta\cos\left( 4\pi\dfrac{x+0.8}{1.6}\right)\right),
\end{equation}
using the inversion sampling technique on the spatial domain $\Omega=[-0.8,0.8]$. The
vector-valued parameter $\prm=(\alpha,\beta,\bl{\nu})$ takes values in the
set $\Gamma_h$, derived via uniform samples of the parameter domain $\Gamma=[0.07,0.09]\times[0.02,0.03]\times[0.4,0.8]$ with $\Np=125$ values.
The full model solution is computed in the interval $\mathcal{T}=[0,20]$, split into $N_{\tau}=20000$ time steps,
using the symplectic midpoint rule. In this setting, particles oscillate along the longitudinal dimension with different transverse velocities and an approximate period of $2\pi\bl{\nu}$, with a bulk of slow particles in the center of the beam spreading thin filaments of faster particles, as shown in Figure \ref{fig:init_cond_final_V1D1V}.

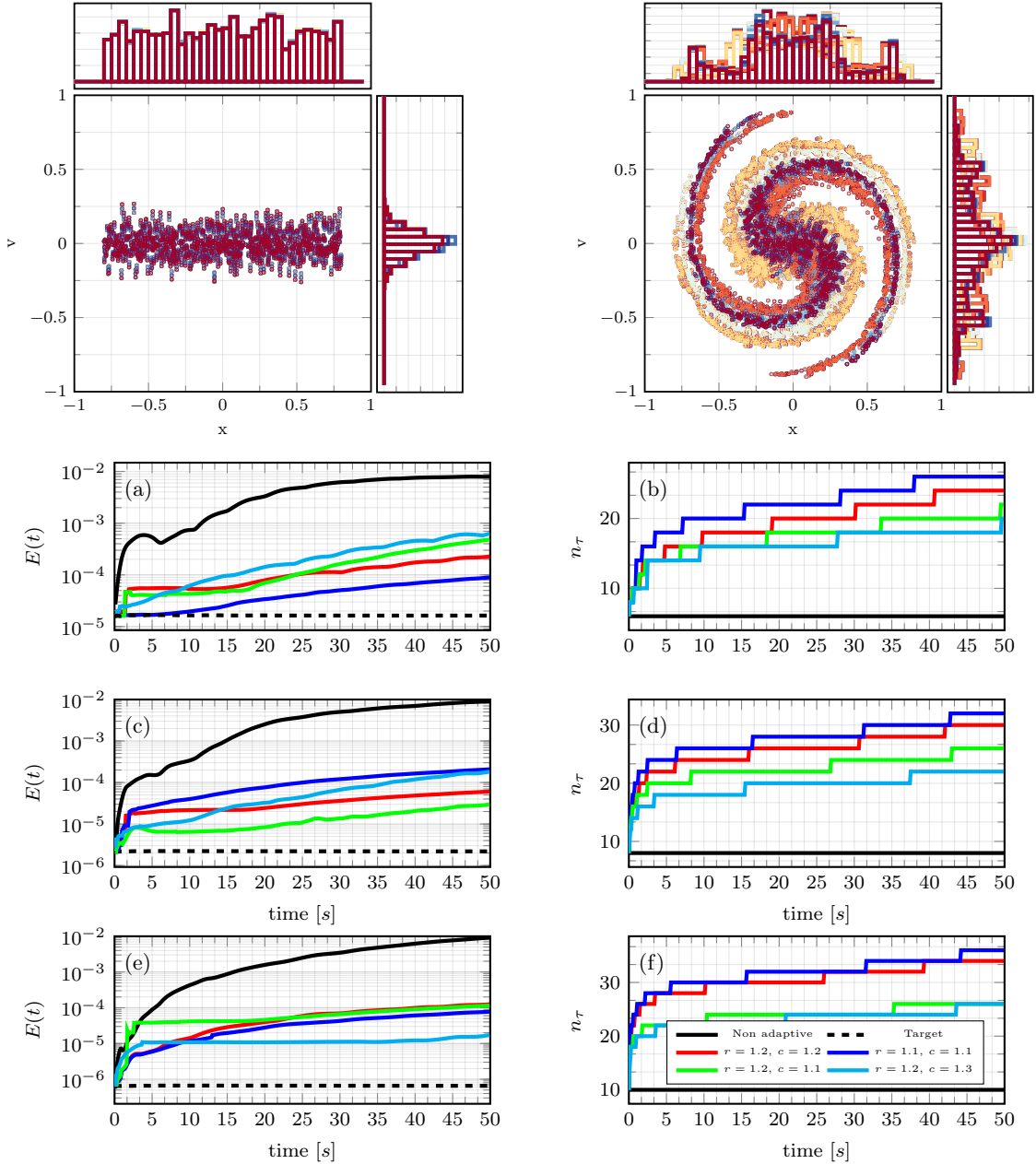
\begin{figure}[H]
\centering
\includegraphics[scale=0.87]{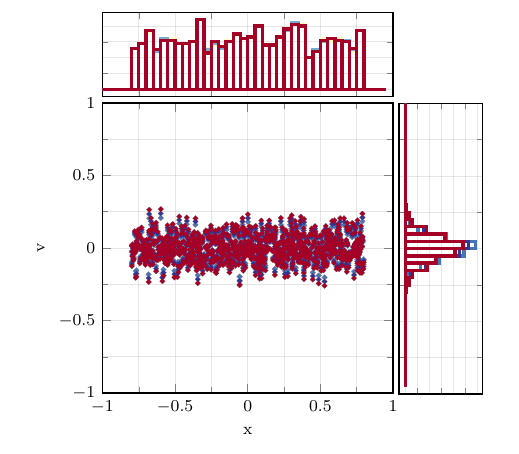}
\hspace{0.35cm}
\includegraphics[scale=0.87]{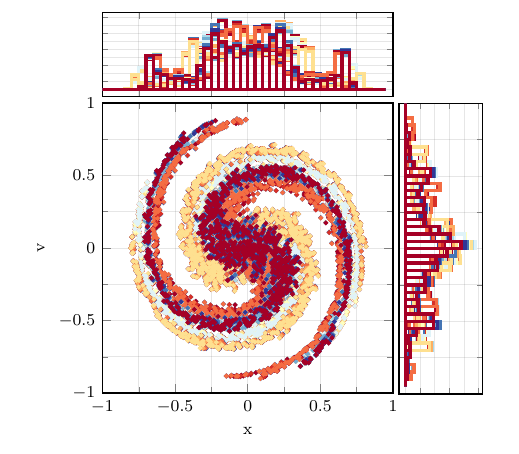}
\begin{tikzpicture}
\node[inner sep=0pt] at (-0.5,0) {(a)};
\node[inner sep=0pt] at (7,0) {(b)};
\end{tikzpicture}
\caption{Vlasov 1D1V: Particle distribution associated with the initial condition (\bl{Fig. (a)}) and the high-fidelity solution at $t=T$ (\bl{Fig. (b)}) for all the parameter values $\prmh\in\Sprmh$.}
\label{fig:init_cond_final_V1D1V}
\end{figure}
The reducibility of the problem is studied by computing the normalized singular values of the global snapshots matrix $\mathcal{S}\in\mathbb{R}^{2N\times(N_{\tau}p)}$ and the time average of the normalized singular values of the matrices $\mathcal{S}_{\tau}\in\mathbb{R}^{2N\times p}$ of snapshots at fixed time $t^{\tau}$ for all $\tau=1,\dots,N_{\tau}$, collecting the high-fidelity solutions corresponding to all the sampled parameters $\prmh$.
The spectra, reported in Figure \ref{fig:singular_values_V1D1V_a}, suggest that the decay of the global singular values is fast enough to hint at a global low-rank structure of the problem. However, for this test case, an adaptive dynamical approach is expected to be beneficial in capturing the increasing rank of the solution (Figure \ref{fig:singular_values_V1D1V_b}) with a smaller local reduced basis.
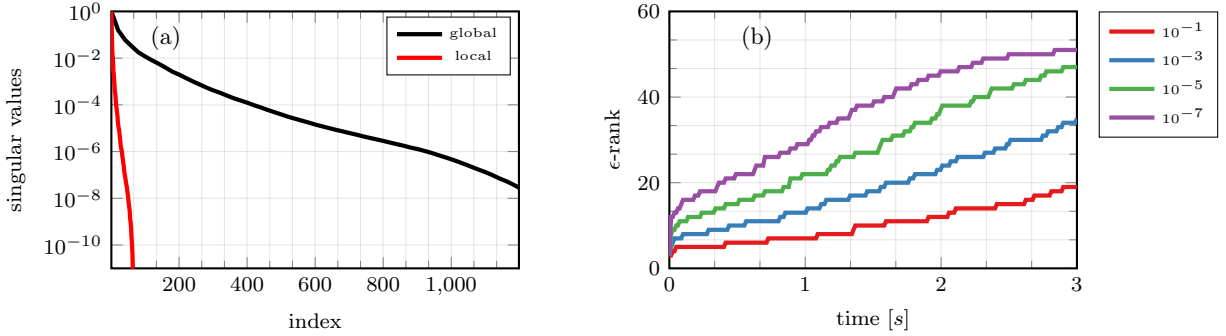
\begin{figure}[H]
\centering
\begin{tikzpicture}
    \begin{groupplot}[
      group style={group size=2 by 1,
                   horizontal sep=2cm},
      width=7cm, height=5cm
    ]
    \nextgroupplot[xlabel={index},
                   ylabel={singular values},
                   axis line style = thick,
                   grid=both,
                   minor tick num=2,
                   max space between ticks=20,
                   grid style = {gray,opacity=0.2},
                   ymode=log,
                   ylabel near ticks, 
                   every axis plot/.append style={ultra thick},
                   legend style={at={(0.8,0.36)},anchor=north},
                   xmin = 1,
                   ymin = 1e-14, ymax = 1,
                   xlabel style={font=\footnotesize},
                   ylabel style={font=\footnotesize},
                   x tick label style={font=\footnotesize},
                   y tick label style={font=\footnotesize},
                   legend style={font=\tiny}
                   ]
        \addplot[color=black] table[x=Index,y=Values] {figures/data/V1D1V_singular_values/full_singular_values_V1D1V.txt};
        \addplot[color=red] table[x=Index,y=Values] {figures/data/V1D1V_singular_values/avg_singular_values_V1D1V.txt};
        \legend{global,local};
        \node [text width=1em,anchor=north west] at (rel axis cs: 0.09,1.05) {\subcaption{\label{fig:singular_values_V1D1V_a}}};
        \coordinate (spypoint) at (axis cs:2,0.005);
    \nextgroupplot[xlabel={time $\left [ s \right ]$},
                   ylabel={$\epsilon$-rank},
                   axis line style = thick,
                   grid=both,
                   minor tick num=2,
                   grid style = {gray,opacity=0.2},
                   every axis plot/.append style={ultra thick},
                   xmin = 0, xmax = 7,
                   ymin = 0, ymax = 80,
                   legend style={at={(1.16,1)},anchor=north},
                   xlabel style={font=\footnotesize},
                   ylabel style={font=\footnotesize},
                   x tick label style={font=\footnotesize},
                   y tick label style={font=\footnotesize},
                   legend style={font=\tiny}]
        \addplot+[] table[x=Time,y=Rank] {figures/data/V1D1V_epsilon_rank/0.1_epsilon_rank_V1D1V.txt};
        \addplot+[] table[x=Time,y=Rank] {figures/data/V1D1V_epsilon_rank/0.001_epsilon_rank_V1D1V.txt};
        \addplot+[] table[x=Time,y=Rank] {figures/data/V1D1V_epsilon_rank/1e-05_epsilon_rank_V1D1V.txt};
        \addplot+[] table[x=Time,y=Rank] {figures/data/V1D1V_epsilon_rank/1e-07_epsilon_rank_V1D1V.txt};
        \addplot+[] table[x=Time,y=Rank] {figures/data/V1D1V_epsilon_rank/1e-09_epsilon_rank_V1D1V.txt};
        \legend{$10^{-1}$,$10^{-3}$,$10^{-5}$,$10^{-7}$,$10^{-9}$};
        \node [text width=1em,anchor=north west] at (rel axis cs: 0.02,1.05) {\subcaption{\label{fig:singular_values_V1D1V_b}}};
    \end{groupplot}
    \node[pin={[pin distance=3.2cm]372:{%
        \begin{tikzpicture}[baseline,trim axis right]
            \begin{axis}[
                    axis line style = thick,
                    grid=both,
                    minor tick num=2,
                    grid style = {gray,opacity=0.2},
                    every axis plot post/.append style={ultra thick},
                    tiny,
                    ymode=log,
                    xmin=0,xmax=22,
                    ymin=0.0001,ymax=1,
                    width=3.6cm,
                    legend style={at={(1.4,1)},anchor=north},
                    legend cell align=left,
                    xlabel style={font=\footnotesize},
                    ylabel style={font=\footnotesize},
                    x tick label style={font=\footnotesize},
                    y tick label style={font=\footnotesize},
                    legend style={font=\tiny},
                ]
                \addplot[color=black] table[x=Index,y=Values] {figures/data/V1D1V_singular_values/full_singular_values_V1D1V.txt};
                \addplot[color=red] table[x=Index,y=Values] {figures/data/V1D1V_singular_values/avg_singular_values_V1D1V.txt};
            \end{axis}
        \end{tikzpicture}%
    }},draw,circle,minimum size=1cm] at (spypoint) {};
\end{tikzpicture}
\caption{Vlasov 1D1V: \bl{(a)} Singular values of the global snapshots matrix $\mathcal{S}$ and time average of the singular values of the local trajectories matrix $\mathcal{S}_{\tau}$. The singular values are normalized using the largest singular value for each case.\bl{(b)} $\epsilon$-rank of the local trajectories matrix $\mathcal{S}_{\tau}$ for different values of $\epsilon$.}
\end{figure}
\noindent
Figure \ref{fig:Hamiltonian_error_Vlasov} shows the relative error in the conservation of the Hamiltonian for different dimensions of the reduced manifold,
and values of the control parameters $r$ and $c$.
Although exact Hamiltonian conservation is not guaranteed by the proposed partitioned RK methods, good control in the conservation error, almost independent of the reduced dimension and control parameters, results from the preservation of the symplectic structure both in the reduction and in the discretization.
The development of temporal integrators
for the Vlasov--Poisson problem
that are both structure and energy preserving should be a subject of future investigations.
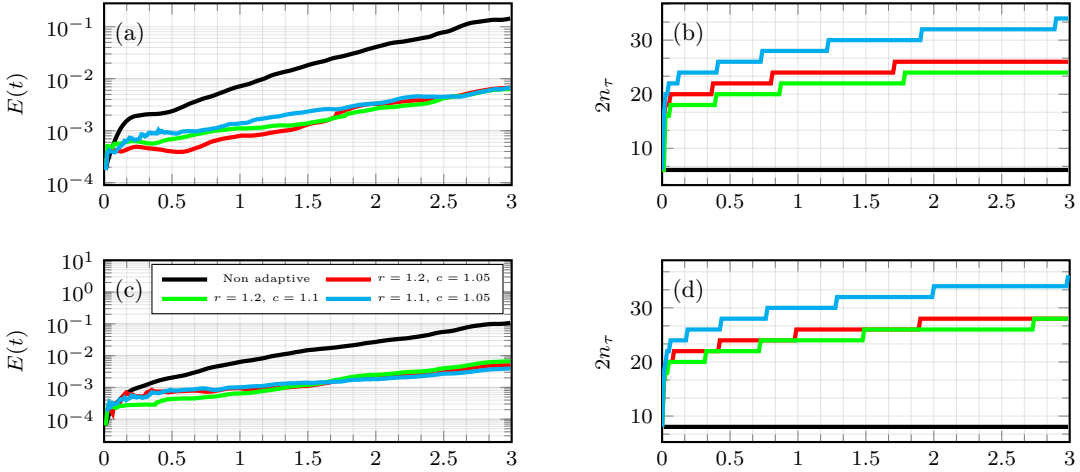
\begin{figure}[H]
\centering
\begin{tikzpicture}
    \begin{groupplot}[
      group style={group size=2 by 2,
                  horizontal sep=2cm},
      width=7cm, height=4cm
    ]
    \nextgroupplot[ylabel={$E_{\mathcal{H}_h}(t)$},
                  axis line style = thick,
                  grid=both,
                  minor tick num=2,
                  max space between ticks=20,
                  grid style = {gray,opacity=0.2},
                  every axis plot/.append style={ultra thick},
                  xmin=0, xmax=20,
                  ymode=log,
                  xlabel style={font=\footnotesize},
                  ylabel style={font=\footnotesize},
                  x tick label style={font=\footnotesize},
                  y tick label style={font=\footnotesize},
                  legend style={font=\tiny}]
        \addplot[color=black] table[x=Timing,y=Error] {figures/data/V1D1V_Hamiltonian/error_Hamiltonian_local_reduced_model_no_indicator_V1D1V_6.txt};
        \addplot[color=red] table[x=Timing,y=Error_1_2] {figures/data/V1D1V_Hamiltonian/error_Hamiltonian_local_reduced_model_indicator_V1D1V_6.txt};
        \addplot[color=blue] table[x=Timing,y=Error_1_3] {figures/data/V1D1V_Hamiltonian/error_Hamiltonian_local_reduced_model_indicator_V1D1V_6.txt};
        \addplot[color=green] table[x=Timing,y=Error_1_5] {figures/data/V1D1V_Hamiltonian/error_Hamiltonian_local_reduced_model_indicator_V1D1V_6.txt};
        \node [text width=1em,anchor=north west] at (rel axis cs: 0.01,1.05) {\subcaption{\label{fig:error_Hamiltonian_V1D1V_6}}};
    \nextgroupplot[axis line style = thick,
                  grid=both,
                  minor tick num=2,
                  max space between ticks=20,
                  grid style = {gray,opacity=0.2},
                  every axis plot/.append style={ultra thick},
                  xmin=0, xmax=20,
                  ymin=10^(-8),
                  ymode=log,
                  xlabel style={font=\footnotesize},
                  ylabel style={font=\footnotesize},
                  x tick label style={font=\footnotesize},
                  y tick label style={font=\footnotesize},
                  legend style={font=\tiny},
                  legend pos = south east,
                  legend columns = 2,
                  legend style={nodes={scale=0.83, transform shape}}]
        \addplot[color=black] table[x=Timing,y=Error] {figures/data/V1D1V_Hamiltonian/error_Hamiltonian_local_reduced_model_no_indicator_V1D1V_8.txt};
        \addplot[color=red] table[x=Timing,y=Error_1_2] {figures/data/V1D1V_Hamiltonian/error_Hamiltonian_local_reduced_model_indicator_V1D1V_8.txt};
        \addplot[color=blue] table[x=Timing,y=Error_1_3] {figures/data/V1D1V_Hamiltonian/error_Hamiltonian_local_reduced_model_indicator_V1D1V_8.txt};
        \addplot[color=green] table[x=Timing,y=Error_1_5] {figures/data/V1D1V_Hamiltonian/error_Hamiltonian_local_reduced_model_indicator_V1D1V_8.txt};
        \node [text width=1em,anchor=north west] at (rel axis cs: 0.01,1.05) {\subcaption{\label{fig:error_Hamiltonian_V1D1V_8}}};
        \legend{Non adaptive,
                $r=1.2 \quad c=1.1$,
                $r=1.3 \quad c=1.1$,
                $r=1.5 \quad c=1.1$};
    \nextgroupplot[ylabel={$E_{\mathcal{H}_h}(t)$},
                  xlabel={\bl{time $\left [ s \right ]$}},
                  axis line style = thick,
                  grid=both,
                  minor tick num=2,
                  max space between ticks=20,
                  grid style = {gray,opacity=0.2},
                  every axis plot/.append style={ultra thick},
                  xmin=0, xmax=20,
                  ymode=log,
                  xlabel style={font=\footnotesize},
                  ylabel style={font=\footnotesize},
                  x tick label style={font=\footnotesize},
                  y tick label style={font=\footnotesize},
                  legend style={font=\tiny}]
        \addplot[color=black] table[x=Timing,y=Error] {figures/data/V1D1V_Hamiltonian/error_Hamiltonian_local_reduced_model_no_indicator_V1D1V_10.txt};
        \addplot[color=red] table[x=Timing,y=Error_1_2] {figures/data/V1D1V_Hamiltonian/error_Hamiltonian_local_reduced_model_indicator_V1D1V_10.txt};
        \addplot[color=blue] table[x=Timing,y=Error_1_3] {figures/data/V1D1V_Hamiltonian/error_Hamiltonian_local_reduced_model_indicator_V1D1V_10.txt};
        \addplot[color=green] table[x=Timing,y=Error_1_5] {figures/data/V1D1V_Hamiltonian/error_Hamiltonian_local_reduced_model_indicator_V1D1V_10.txt};
        \node [text width=1em,anchor=north west] at (rel axis cs: 0.01,1.05) {\subcaption{\label{fig:error_Hamiltonian_V1D1V_10}}};
    \nextgroupplot[xlabel={\bl{time $\left [ s \right ]$}},
                  axis line style = thick,
                  grid=both,
                  minor tick num=2,
                  max space between ticks=20,
                  grid style = {gray,opacity=0.2},
                  every axis plot/.append style={ultra thick},
                  xmin=0, xmax=20,
                  ymode=log,
                  xlabel style={font=\footnotesize},
                  ylabel style={font=\footnotesize},
                  x tick label style={font=\footnotesize},
                  y tick label style={font=\footnotesize},
                  legend style={font=\tiny}]
        \addplot[color=black] table[x=Timing,y=Error] {figures/data/V1D1V_Hamiltonian/error_Hamiltonian_local_reduced_model_no_indicator_V1D1V_12.txt};
        \addplot[color=red] table[x=Timing,y=Error_1_2] {figures/data/V1D1V_Hamiltonian/error_Hamiltonian_local_reduced_model_indicator_V1D1V_12.txt};
        \addplot[color=blue] table[x=Timing,y=Error_1_3] {figures/data/V1D1V_Hamiltonian/error_Hamiltonian_local_reduced_model_indicator_V1D1V_12.txt};
        \addplot[color=green] table[x=Timing,y=Error_1_5] {figures/data/V1D1V_Hamiltonian/error_Hamiltonian_local_reduced_model_indicator_V1D1V_12.txt};
        \node [text width=1em,anchor=north west] at (rel axis cs: 0.01,1.05) {\subcaption{\label{fig:error_Hamiltonian_V1D1V_12}}};
    \end{groupplot}
\end{tikzpicture}
\caption{Vlasov 1D1V: Relative error \eqref{eqn:relative_error_Hamiltonian} in the conservation of the discrete Hamiltonian \eqref{eq:VPHamN}
for the dynamical adaptive reduced basis method with initial reduced dimensions
$\Nr_1=6$ (Fig.~\bl{(a)}),
$\Nr_1=8$ (Fig.~\bl{(b)}),
$\Nr_1=10$ (Fig.~\bl{(c)}) and
$\Nr_1=12$ (Fig.~\bl{(d)}).}
\label{fig:Hamiltonian_error_Vlasov}
\end{figure}
In Figure \ref{fig:error_final_V1D1V}, we compare the error \eqref{eqn:error_metric}
and the runtime of the global reduced model, the dynamical models for different values of $r$, and the high-fidelity model. For the global reduced method, we consider the complex SVD approach with and without the tensorial representation of the RHS and of its Jacobian, \emph{cf.} \Cref{sec:nonlinear}.
The results show that, as we increase the dimension of the reduced basis, the runtime cost of the global reduce model becomes larger than the one required to solve the high-fidelity problem, i.e., a global reduction proves ineffective.
Both the non-adaptive and the adaptive dynamical reduced approach outperforms the global reduced method by reaching comparable levels of accuracy at a much lower computational cost. 
In the adaptive algorithm, the additional computational cost associated with the evaluation of the error indicator and the evolution of a larger basis is balanced by a considerable error reduction.
\begin{figure}[H]
\centering
\begin{tikzpicture}[spy using outlines={rectangle, width=6.2cm, height=5cm, magnification=2, connect spies}]
    \begin{axis}[xlabel={runtime $\left[s\right]$},
                 ylabel={$E(t)$},
                 axis line style = thick,
                 grid=both,
                 minor tick num=2,
                 grid style = {gray,opacity=0.2},
                 xmode=log,
                 ymode=log,
                 ymax = 0.018, ymin = 0.00001,
                 xmax = 60001,
                 every axis plot/.append style={ultra thick},
                 width = 9cm, height = 6cm,
                 legend style={at={(1.4,1)},anchor=north},
                 legend cell align=left,
                 xlabel style={font=\footnotesize},
                 ylabel style={font=\footnotesize},
                 x tick label style={font=\footnotesize},
                 y tick label style={font=\footnotesize},
                 legend style={font=\tiny}]
        \addplot+[mark=x,color=black,size=2pt,
            every node near coord/.append style={xshift=0.65cm},
            every node near coord/.append style={yshift=-0.2cm},
            nodes near coords, 
            point meta=explicit symbolic,
            every node near coord/.append style={font=\footnotesize}] table[x=Timing,y=Error, meta index=2] 
            {figures/data/V1D1V_Pareto/error_final_local_reduced_model_no_indicator_V1D1V.txt};
        \addplot+[mark=x,color=red,size=2pt] 
        table[x=Timing,y=Error]
            {figures/data/V1D1V_Pareto/error_final_local_reduced_model_indicator_V1D1V_1.2_1.1.txt};   
        \addplot+[mark=x,color=blue,size=2pt] table[x=Timing,y=Error]
            {figures/data/V1D1V_Pareto/error_final_local_reduced_model_indicator_V1D1V_1.3_1.1.txt}; 
        \addplot+[mark=x,color=green,size=2pt] table[x=Timing,y=Error]
            {figures/data/V1D1V_Pareto/error_final_local_reduced_model_indicator_V1D1V_1.5_1.1.txt};   
        \addplot+[mark=x,color=violet,size=2pt,
            every node near coord/.append style={xshift=0.65cm},
            every node near coord/.append style={yshift=-0.2cm},
            nodes near coords, 
            point meta=explicit symbolic,
            every node near coord/.append style={font=\footnotesize}] table[x=Timing,y=Error, meta index=2] 
            {figures/data/V1D1V_Pareto/error_final_global_reduced_model_V1D1V_standard.txt};
        \addplot+[mark=x,color=violet,size=2pt,dashed,
            every node near coord/.append style={xshift=-0.35cm},
            every node near coord/.append style={yshift=-0.55cm},
            nodes near coords, 
            point meta=explicit symbolic,
            every node near coord/.append style={font=\footnotesize}] table[x=Timing,y=Error, meta index=2] 
            {figures/data/V1D1V_Pareto/error_final_global_reduced_model_V1D1V_cubic_expansion.txt}; 
        \addplot+[color=black,dashed] table[x=Timing,y=DummyError]
            {figures/data/V1D1V_Pareto/error_final_full_V1D1V.txt}; 
        \legend{Non adaptive,
                $r=1.2 \quad c=1.1$,
                $r=1.3 \quad c=1.1$,
                $r=1.5 \quad c=1.1$,
                Global,
                Global (Tensorial POD),
                Full model};
        \coordinate (spypoint) at (axis cs:800,0.0025);
        \coordinate (magnifyglass) at (80,0.0019);
    \end{axis}
\end{tikzpicture}
\caption{Vlasov 1D1V: Error \eqref{eqn:error_metric}, at final time, as a function of the runtime for the complex SVD method ({\color{violet}{\rule[.5ex]{1em}{1.2pt}}},{\color{violet}{\rule[.5ex]{0.45em}{1.2pt}}}{\color{white}{\rule[.5ex]{0.1em}{1.2pt}}}{\color{violet}{\rule[.5ex]{0.45em}{1.2pt}}}),  the dynamical RB method ({\color{black}{\rule[.5ex]{1em}{1.2pt}}}) and the adaptive dynamical RB method for different values of the control parameters $r$ and $c$     ({\color{red}{\rule[.5ex]{1em}{1.2pt}}},{\color{blue}{\rule[.5ex]{1em}{1.2pt}}},{\color{green}{\rule[.5ex]{1em}{1.2pt}}},{\color{cyan}{\rule[.5ex]{1em}{1.2pt}}}).
For the sake of comparison, we report the timing required by the high-fidelity solver ({\color{black}{\rule[.5ex]{0.4em}{1.2pt}}} {\color{black}{\rule[.5ex]{0.4em}{1.2pt}}}) to compute the numerical solution for all values of the parameter $\prmh\in\Sprmh$.}
\label{fig:error_final_V1D1V}
\end{figure}

In Figures \ref{fig:time_basis_V1D1V} we report the growth of the dimension of the reduced manifold for different initial dimension $\Nr_1$. As for the evolution of the error, we do not notice any significant difference as the parameter $r$ for the adaptive criterion \eqref{eq:ratio_tmp} varies. The increase of the rank of the full model solution, see Figure \ref{fig:singular_values_V1D1V_b}, is reproduced by the adaptive algorithm up to a tolerance of around $\epsilon=10^{-5}$.
\begin{figure}[H]
\centering
\begin{tikzpicture}
    \begin{groupplot}[
      group style={group size=2 by 3,
                   horizontal sep=2cm},
      width=7cm, height=4cm
    ]
    \nextgroupplot[ylabel={$E(t)$},
                   axis line style = thick,
                   grid=both,
                   minor tick num=2,
                   max space between ticks=20,
                   grid style = {gray,opacity=0.2},
                   every axis plot/.append style={ultra thick},
                   xmin=0, xmax=20,
                   ymode=log,
                   xlabel style={font=\footnotesize},
                   ylabel style={font=\footnotesize},
                   x tick label style={font=\footnotesize},
                   y tick label style={font=\footnotesize},
                   legend style={font=\tiny}]
        \addplot[color=black] table[x=Time,y=Error_4] {figures/data/V1D1V_error_time_basis/error_local_reduced_model_no_indicator_V1D1V.txt};
        \addplot[color=red] table[x=Timing,y=Error_1_2] {figures/data/V1D1V_error_time_basis/error_local_reduced_model_indicator_V1D1V_4.txt};
        \addplot[color=blue] table[x=Timing,y=Error_1_3] {figures/data/V1D1V_error_time_basis/error_local_reduced_model_indicator_V1D1V_4.txt};
        \addplot[color=green] table[x=Timing,y=Error_1_5] {figures/data/V1D1V_error_time_basis/error_local_reduced_model_indicator_V1D1V_4.txt};
        \node [text width=1em,anchor=north west] at (rel axis cs: 0.01,1.05) {\subcaption{\label{fig:time_error_V1D1V_4}}};
    \nextgroupplot[ylabel={$\Nr_{\tau}$},
                   axis line style = thick,
                   grid=both,
                   minor tick num=2,
                   max space between ticks=20,
                   grid style = {gray,opacity=0.2},
                   every axis plot/.append style={ultra thick},
                   xmin=0, xmax=20,
                   xlabel style={font=\footnotesize},
                   ylabel style={font=\footnotesize},
                   x tick label style={font=\footnotesize},
                   y tick label style={font=\footnotesize},
                   legend style={font=\tiny}]
        \addplot[color=black] table[x=Time,y=Basis_4] {figures/data/V1D1V_error_time_basis/basis_local_reduced_model_no_indicator_V1D1V.txt};  
        \addplot[color=red] table[x=Timing,y=Basis_1_2] {figures/data/V1D1V_error_time_basis/basis_local_reduced_model_indicator_V1D1V_4.txt};
        \addplot[color=blue] table[x=Timing,y=Basis_1_3] {figures/data/V1D1V_error_time_basis/basis_local_reduced_model_indicator_V1D1V_4.txt};
        \addplot[color=green] table[x=Timing,y=Basis_1_5] {figures/data/V1D1V_error_time_basis/basis_local_reduced_model_indicator_V1D1V_4.txt};
        \node [text width=1em,anchor=north west] at (rel axis cs: 0.01,1.05) {\subcaption{\label{fig:time_basis_V1D1V_4}}};
    \nextgroupplot[ylabel={$E(t)$},
                   axis line style = thick,
                   grid=both,
                   minor tick num=2,
                   max space between ticks=20,
                   grid style = {gray,opacity=0.2},
                   every axis plot/.append style={ultra thick},
                   xmin=0, xmax=20,
                   ymode=log,
                   xlabel style={font=\footnotesize},
                   ylabel style={font=\footnotesize},
                   x tick label style={font=\footnotesize},
                   y tick label style={font=\footnotesize},
                   legend style={font=\tiny}]
        \addplot[color=black] table[x=Time,y=Error_8] {figures/data/V1D1V_error_time_basis/error_local_reduced_model_no_indicator_V1D1V.txt};
        \addplot[color=red] table[x=Timing,y=Error_1_2] {figures/data/V1D1V_error_time_basis/error_local_reduced_model_indicator_V1D1V_8.txt};
        \addplot[color=blue] table[x=Timing,y=Error_1_3] {figures/data/V1D1V_error_time_basis/error_local_reduced_model_indicator_V1D1V_8.txt};
        \addplot[color=green] table[x=Timing,y=Error_1_5] {figures/data/V1D1V_error_time_basis/error_local_reduced_model_indicator_V1D1V_8.txt};
        \node [text width=1em,anchor=north west] at (rel axis cs: 0.01,1.05) {\subcaption{\label{fig:time_error_V1D1V_8}}};
    \nextgroupplot[ylabel={$\Nr_{\tau}$},
                   axis line style = thick,
                   grid=both,
                   minor tick num=2,
                   max space between ticks=20,
                   grid style = {gray,opacity=0.2},
                   every axis plot/.append style={ultra thick},
                   xmin=0, xmax=20,
                   xlabel style={font=\footnotesize},
                   ylabel style={font=\footnotesize},
                   x tick label style={font=\footnotesize},
                   y tick label style={font=\footnotesize},
                   legend style={font=\tiny},
                   legend columns = 2,
                   legend style={nodes={scale=0.83, transform shape}},
                   legend style={at={(axis cs:3.1,8.4)},anchor=south west}]
        \addplot[color=black] table[x=Time,y=Basis_8] {figures/data/V1D1V_error_time_basis/basis_local_reduced_model_no_indicator_V1D1V.txt}; 
        \addplot[color=red] table[x=Timing,y=Basis_1_2] {figures/data/V1D1V_error_time_basis/basis_local_reduced_model_indicator_V1D1V_8.txt};
        \addplot[color=blue] table[x=Timing,y=Basis_1_3] {figures/data/V1D1V_error_time_basis/basis_local_reduced_model_indicator_V1D1V_8.txt};
        \addplot[color=green] table[x=Timing,y=Basis_1_5] {figures/data/V1D1V_error_time_basis/basis_local_reduced_model_indicator_V1D1V_8.txt};
        \node [text width=1em,anchor=north west] at (rel axis cs: 0.01,1.05) {\subcaption{\label{fig:time_basis_V1D1V_8}}};
        \legend{{Non adaptive},
                {$r=1.2, c=1.1$},
                {$r=1.3, c=1.1$},
                {$r=1.5, c=1.1$}}
    \nextgroupplot[xlabel={time $\left [ s \right ]$},
                   ylabel={$E(t)$},
                   axis line style = thick,
                   grid=both,
                   minor tick num=2,
                   max space between ticks=20,
                   grid style = {gray,opacity=0.2},
                   every axis plot/.append style={ultra thick},
                   xmin=0, xmax=20,
                   ymode=log,
                   xlabel style={font=\footnotesize},
                   ylabel style={font=\footnotesize},
                   x tick label style={font=\footnotesize},
                   y tick label style={font=\footnotesize},
                   legend style={font=\tiny}]
        \addplot[color=black] table[x=Time,y=Error_12] {figures/data/V1D1V_error_time_basis/error_local_reduced_model_no_indicator_V1D1V.txt};
        \addplot[color=red] table[x=Timing,y=Error_1_2] {figures/data/V1D1V_error_time_basis/error_local_reduced_model_indicator_V1D1V_12.txt};
        \addplot[color=blue] table[x=Timing,y=Error_1_3] {figures/data/V1D1V_error_time_basis/error_local_reduced_model_indicator_V1D1V_12.txt};
        \addplot[color=green] table[x=Timing,y=Error_1_5] {figures/data/V1D1V_error_time_basis/error_local_reduced_model_indicator_V1D1V_12.txt};
        \node [text width=1em,anchor=north west] at (rel axis cs: 0.01,1.05) {\subcaption{\label{fig:time_error_V1D1V_12}}};
    \nextgroupplot[xlabel={time $\left [ s \right ]$},
                   ylabel={$\Nr_{\tau}$},
                   axis line style = thick,
                   grid=both,
                   minor tick num=2,
                   max space between ticks=20,
                   grid style = {gray,opacity=0.2},
                   every axis plot/.append style={ultra thick},
                   xmin=0, xmax=20,
                   xlabel style={font=\footnotesize},
                   ylabel style={font=\footnotesize},
                   x tick label style={font=\footnotesize},
                   y tick label style={font=\footnotesize},
                   legend style={font=\tiny}]
        \addplot[color=black] table[x=Time,y=Basis_12] {figures/data/V1D1V_error_time_basis/basis_local_reduced_model_no_indicator_V1D1V.txt}; 
        \addplot[color=red] table[x=Timing,y=Basis_1_2] {figures/data/V1D1V_error_time_basis/basis_local_reduced_model_indicator_V1D1V_12.txt};
        \addplot[color=blue] table[x=Timing,y=Basis_1_3] {figures/data/V1D1V_error_time_basis/basis_local_reduced_model_indicator_V1D1V_12.txt};
        \addplot[color=green] table[x=Timing,y=Basis_1_5] {figures/data/V1D1V_error_time_basis/basis_local_reduced_model_indicator_V1D1V_12.txt};
        \node [text width=1em,anchor=north west] at (rel axis cs: 0.01,1.05) {\subcaption{\label{fig:time_basis_V1D1V_12}}};
    \end{groupplot}
\end{tikzpicture}
\caption{Vlasov 1D1V: On the left, we show the evolution of the error $E(t)$ \eqref{eqn:error_metric} for the adaptive and non adaptive dynamical RB methods for different values of the control parameters $r$ and $c$, and for different dimensions $\Nr_1$ of the initial reduced manifold.
On the right, we show the evolution of the dimension of the dynamical reduced basis. We consider the cases
$\Nr_1=4$ (Figs.~\bl{(a)-(b)}),
$\Nr_1=8$ (Figs.~\bl{(c)-(d)}) and
$\Nr_1=12$ (Figs.~\bl{(e)-(f)}).}
\label{fig:time_basis_V1D1V}
\end{figure}
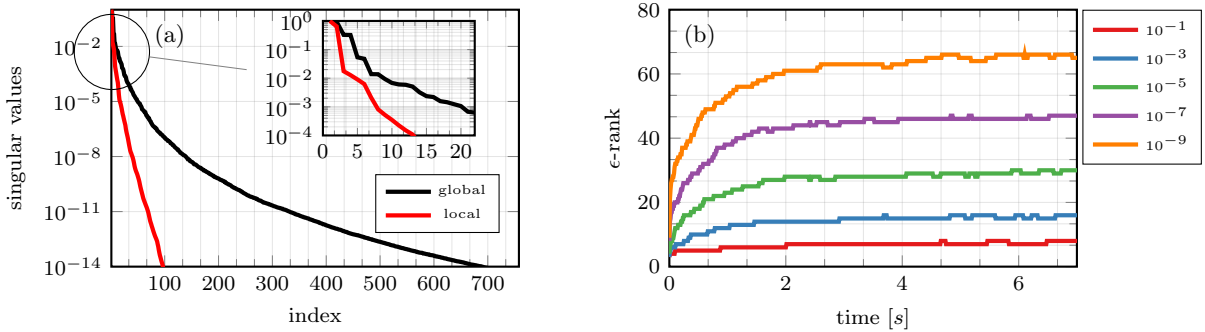

\section{Concluding remarks}\label{sec:conclusions}
We have considered parametrized non-dissipative problems in their canonical symplectic Hamiltonian formulation.
For their model order reduction, we propose a nonlinear structure-preserving reduced basis method
consisting in approximating the problem solution with a modal decomposition
where both the expansion coefficients and
the reduced basis are evolving in time.
Moreover, the dimension of the reduced basis is updated in time according to an adaptive strategy based on an error indicator.
The resulting reduced models allow to achieve stable and accurate results
with small reduced basis even for problems characterized by a slowly decaying Kolmogorov $n$-width.
The strength is the combination of the dynamical adaptivity of the reduced basis and the
preservation of the geometric structure underlying key physical properties of the dynamics, illustrated by examples.

The study of efficient and structure-preserving algorithms for general nonlinear Hamiltonian vector fields and the development of partitioned Runge--Kutta methods that ensure the exact preservation of (at least linear and quadratic) invariants are still open problems and provide interesting directions of investigation.
Moreover, the application of our rank-adaptive reduced basis method to fully kinetic plasma models, like the Vlasov--Poisson problem,
 might also be subject of future studies.

\begin{acknowledgment*}
This work was partially supported by AFOSR under grant FA9550-17-1-9241.
\end{acknowledgment*}

\printbibliography

\appendix

\section{Second and third order partitioned RK methods}\label{app:pRK}
To devise second order accurate partitioned RK schemes
we combine a $2$-stage explicit RK scheme of second order,
known as the modified Euler method (or explicit midpoint method),
with the implicit midpoint rule enlarged by means of a fictitious stage.
%
The following result follows directly from \Cref{lemma:pRKcond}.
\begin{lemma}\label{lem:2_stage_PRK}
Let a $2$-stage partitioned Runge--Kutta method be characterized by the set of coefficients
$P_Z=(\{b_i\}_{i=1}^{\Ns},\{a_{ij}\}_{i,j=1}^{\Ns})$
and
$\widehat{P}_U=(\{\widehat{b}_i\}_{i=1}^{\Ns},\{\widehat{a}_{ij}\}_{1\leq j<i\leq \Ns})$, where $P_Z$ is the implicit midpoint rule
and $\widehat{P}_U$ the explicit midpoint method, namely the non-zero coefficients have values
$b_2=\widehat{b}_2 = 1$, $a_{22}=\widehat{a}_{21} = 1/2$.
The resulting partitioned RK method has order of accuracy $2$ and the numerical integrator $P_Z$ is symplectic.
\end{lemma}
To derive a partitioned Runge--Kutta method of order 3,
we take $P_Z$ to be the $2$-stage Gauss--Legendre (GL)
method of order $4$ enlarged with a fictitious stage.
Starting from the enlarged $4$-stage GL scheme
in \Cref{table:pRK3s3} (left),
we derive an explicit RK method of order $3$
by imposing the conditions \eqref{eq:RKp3} and \eqref{eq:pRKp3}.
The resulting scheme is described by the Butcher tableau in
\Cref{table:pRK3s3} on the right. By construction and in view of
\Cref{lemma:pRKcond}, the following result holds.
\begin{lemma}
The $3$-stage partitioned Runge--Kutta method characterized by the set of coefficients
$P_Z$ and $\widehat{P}_U$ in \Cref{table:pRK3s3}
has order of accuracy $3$ and the numerical integrator $P_Z$ is symplectic.
\end{lemma}
\begin{table}[ht!]\setlength\belowcaptionskip{-10pt}
\begin{center}
\footnotesize
\begin{tabular}{c|c c c }
 0 & 0 & 0 & 0 \\ 
 $\frac12-\frac{\sqrt{3}}{6}$ & 0 & $\frac14$ & $\frac14-\frac{\sqrt{3}}{6}$ \\ 
 $\frac12+\frac{\sqrt{3}}{6}$ & 0 & $\frac14+\frac{\sqrt{3}}{6} $ & $\frac14$ \\[1ex]
 \hline
 & 0 & $\frac12$ & $\frac12$
\end{tabular}
\qquad
\begin{tabular}{c|c c c }
 0 & 0 & 0 & 0 \\ 
 $\frac12-\frac{\sqrt{3}}{6}$ & $\frac12-\frac{\sqrt{3}}{6}$ & 0 & 0 \\ 
 $\frac12+\frac{\sqrt{3}}{6}$ & -$\frac{1}{3-\sqrt{3}}$ & $\frac{2}{3-\sqrt{3}}$ & 0 \\[1ex]
 \hline
 & 0 & $\frac12$ & $\frac12$
\end{tabular}
\caption{Butcher tableau for the Gauss--Legendre scheme of order $4$, on the left, and for the explicit $3$-stage Runge--Kutta method of order $3$,
on the right.}
\label{table:pRK3s3}
\end{center}
\end{table}

We construct a partitioned RK scheme of order $3$ with a larger region of absolute stability
by including a further stage.
This can be obtained by coupling the Gauss--Legendre scheme
of order $6$, suitably enlarged with a fictitious stage, to an explicit RK method, as
in \Cref{table:pRK3s4}.
\begin{table}[ht!]\setlength\belowcaptionskip{-10pt}
\begin{center}
\footnotesize
\begin{tabular}{c|c c c c}
 0 & 0 & 0 & 0 & 0 \\ 
 $\frac12-\frac{\sqrt{15}}{10}$ & 0 & $\frac{5}{36}$ & $\frac29-\frac{\sqrt{15}}{15}$ & $\frac{5}{36}-\frac{\sqrt{15}}{30}$ \\ 
 $\frac12$ & 0 & $\frac{5}{36}+\frac{\sqrt{15}}{24}$ & $\frac29$ & $\frac{5}{36}-\frac{\sqrt{15}}{24}$ \\ 
 $\frac12+\frac{\sqrt{15}}{10}$ & 0 & $\frac{5}{36}+\frac{\sqrt{15}}{30}$ & $\frac29+\frac{\sqrt{15}}{15}$ & $\frac{5}{36}$ \\[1ex]
 \hline
 & 0 & $\frac{5}{18}$ & $\frac{4}{9}$ & $\frac{5}{18}$
\end{tabular}
\qquad
\begin{tabular}{c|c c c c}
 0 & 0 & 0 & 0 & 0 \\ 
 $\frac12-\frac{\sqrt{15}}{10}$ & $\frac12-\frac{\sqrt{15}}{10}$ & 0 & 0 & 0 \\ 
 $\frac12$ & $\widehat{a}_{31}$ & $\widehat{a}_{32}$ & 0 & 0 \\ 
 $\frac12+\frac{\sqrt{15}}{10}$ & $\widehat{a}_{41}$ & $\widehat{a}_{42}$ & $\widehat{a}_{43}$ & 0 \\[1ex]
 \hline
 & $\widehat{b}_1$ & $\widehat{b}_2$ & $\widehat{b}_3$ & $\widehat{b}_4$
\end{tabular}
\caption{Butcher tableau for the Gauss--Legendre scheme of order $6$, on the left and explicit $4$-stage Runge--Kutta method of order $3$, on the right.}
\label{table:pRK3s4}
\end{center}
\end{table}

The nine unknown coefficients of the explicit third order scheme in \Cref{table:pRK3s4} are obtained by solving the underdetermined system
derived by imposing the eight order conditions
\eqref{eq:RKp3} and \eqref{eq:pRKp3}.
A further equation can be imposed by adding a constraint on the
region of absolute stability of the scheme: this is given by
$\{z\in\mathbb{C}:\;|R(z)|<1\}$ where the stability function is
$R(z) = 1+ z+ z^2/2 + z^3/6 + K z^4$, with
$K:= \widehat{b}_4\,\widehat{a}_{43}\,\widehat{a}_{32}\,\widehat{a}_{21}$.

\end{document}

%% file: macros.tex
\newcommand{\Fcal}{\ensuremath{\mathcal{F}}}
\newcommand{\Gcal}{\ensuremath{\mathcal{G}}}
\newcommand{\Hcal}{\ensuremath{\mathcal{H}}}

\newcommand{\Jcal}{\ensuremath{\mathcal{J}}}

\newcommand{\Lcal}{\ensuremath{\mathcal{L}}}
\newcommand{\Mcal}{\ensuremath{\mathcal{M}}}

\newcommand{\Pcal}{\ensuremath{\mathcal{P}}}

\newcommand{\Rcal}{\ensuremath{\mathcal{R}}}
\newcommand{\Scal}{\ensuremath{\mathcal{S}}}
\newcommand{\Tcal}{\ensuremath{\mathcal{T}}}
\newcommand{\Ucal}{\ensuremath{\mathcal{U}}}
\newcommand{\Vcal}{\ensuremath{\mathcal{V}}}

\newcommand{\Xcal}{\ensuremath{\mathcal{X}}}

\newcommand{\Zcal}{\ensuremath{\mathcal{Z}}}
\newcommand{\Tcalt}{\ensuremath{{\Tcal_{\tau}}}}
\newcommand{\Ucalt}{\ensuremath{{\Ucal_{\tau}}}}
\newcommand{\Zcalt}{\ensuremath{{\Zcal_{\tau}}}}

\newcommand{\norm}[1]{\|{#1}\|}

\newcommand{\Vd}[1]{\ensuremath{\mathcal{V}_{#1}}}

\newcommand{\J}[1]{\ensuremath{J_{#1}}}

\newcommand{\Ham}{\ensuremath{\Hcal}}
\newcommand{\HamN}{\ensuremath{\Hcal}}
\newcommand{\HamU}{\ensuremath{\Hcal_{U}}}

\newcommand{\R}[2]{\ensuremath{\mathbb{R}^{{#1}\times{#2}}}}

\renewcommand{\r}[1]{\ensuremath{\mathbb{R}^{#1}}}

\newcommand{\rank}[1]{\ensuremath{\mathrm{rank}({#1})}}

\renewcommand{\ker}[1]{\ensuremath{\mathrm{ker}({#1})}}
\newcommand{\res}{\ensuremath{\rho}}

\newcommand{\tm}{\ensuremath{{\tau-1}}}

\newcommand{\Nf}{\ensuremath{{2N}}} 
\newcommand{\Nr}{\ensuremath{{2n}}} 
\newcommand{\Nrt}{\ensuremath{{2n_{\tau}}}} 

\newcommand{\Nrmt}{\ensuremath{{2n_{\tau}{-}2}}}
\newcommand{\Nfh}{\ensuremath{N}} 
\newcommand{\Nrh}{\ensuremath{n}} 
\newcommand{\Nrht}{\ensuremath{n_{\tau}}} 

\newcommand{\Np}{\ensuremath{p}} 
\newcommand{\Ns}{\ensuremath{{s}}} 
\newcommand{\Sreg}{\ensuremath{S_{\varepsilon}}}
\newcommand{\reg}{\ensuremath{\varepsilon}}

\newcommand{\dt}{\ensuremath{\Delta t\,}} 

\newcommand{\freqE}{\ensuremath{N_{\mathbf{E}}}}

\newcommand{\T}[2]{\ensuremath{T_{#1}#2}}

\newcommand{\Idm}{\ensuremath{I}}

\newcommand{\prm}{\ensuremath{\eta}} 
\newcommand{\prmh}{\ensuremath{\eta_h}} 
\newcommand{\prmhsub}{\ensuremath{\widetilde{\eta}_h}} 
\newcommand{\Sprm}{\ensuremath{\Gamma}} 
\newcommand{\Sprmh}{\ensuremath{\Sprm_h}} 
\newcommand{\cay}{\ensuremath{\mathrm{cay}}}

\numberwithin{equation}{section}

\theoremstyle{plain}
  \newtheorem{theorem}{\sffamily Theorem}[section]
  \Crefname{proposition}{Proposition}{Propositions}
  \newtheorem{lemma}[theorem]{\sffamily Lemma}\Crefname{lemma}{Lemma}{Lemmas}

  \theoremstyle{definition}
  \newtheorem{remark}[theorem]{\sffamily Remark}\Crefname{remark}{Remark}{Remarks}
  \newtheorem{definition}[theorem]{\sffamily Definition}
  
  \newtheorem*{acknowledgment*}{Acknowledgment}

\makeatletter
\def\tikz@auto@anchor{%
    \pgfmathtruncatemacro\angle{atan2(\pgf@x,\pgf@y)-90}
    \edef\tikz@anchor{\angle}%
}
\makeatother 

\makeatletter
\newcommand*\bigcdot{\mathpalette\bigcdot@{.5}}
\newcommand*\bigcdot@[2]{\mathbin{\vcenter{\hbox{\scalebox{#2}{$\m@th#1\bullet$}}}}}
\makeatother